\documentclass[11pt]{article}
\usepackage{amssymb,amsmath,amsfonts,amsthm,mathtools,color}

\usepackage{mathrsfs}
\usepackage{dsfont}
 \usepackage{bm} 

\usepackage{url} 

\usepackage[overload]{empheq}

\usepackage{comment} 
\usepackage[title,titletoc,header]{appendix}
\usepackage{graphicx,subfig}
\usepackage{float} 
\usepackage{dsfont} 

\usepackage{paralist}

\usepackage{cases}

\usepackage{tikz}
\usetikzlibrary{arrows,positioning,shapes.geometric}

\usepackage[linesnumbered, ruled,vlined]{algorithm2e}
\SetKwInput{KwInput}{Input}                

\graphicspath{ {./figures/} }

\usepackage{indentfirst}
\usepackage{multicol}
\usepackage{booktabs}
\usepackage{url}
\usepackage[outdir=./]{epstopdf} 

\usepackage[shortlabels]{enumitem}

\usepackage{hyperref}
\hypersetup{
    colorlinks=true, 
    linktoc=all,     
    linkcolor=blue,  
}

\newtheorem{lemma}{Lemma}[section]
\newtheorem{theorem}{Theorem}[section]
\newtheorem{proposition}{Proposition}[section]
\newtheorem{definition}{Definition}[section]
\newtheorem{remark}{Remark}[section]
\newtheorem{example}{Example}[section]
\newtheorem{assumption}{Assumption}[section]
\newtheorem{corollary}{Corollary}[section]
\numberwithin{equation}{section}

\def \bX{\boldsymbol{X}}

\def \ba{\boldsymbol{a}}

\def \bu{\boldsymbol{u}}

\def \bx{\boldsymbol{x}}

\def \cA{\mathcal{A}}
\def \cB{\mathcal{B}}
\def \cC{\mathcal{C}}
\def \cF{\mathcal{F}}
\def \cG{\mathcal{G}}
\def \cH{\mathcal{H}}
\def \cL{\mathcal{L}}

\def \cP{\mathcal{P}}

\def \cW{\mathcal{W}}

\def \tX{\Tilde{X}}
\def \tY{\Tilde{Y}}

\def \EE{\mathbb{E}}
\def \FF{\mathbb{F}}

\def \MM{\mathbb{M}}
\def \NN{\mathbb{N}}
\def \PP{\mathbb{P}}
\def \RR{\mathbb{R}}

\newcommand{\Tr}{\operatorname{tr}}
\newcommand{\op}{\operatorname{op}}

\DeclareMathOperator*{\esssup}{ess\,sup}

\title{Limit Theory for $N$-Player $\alpha$-Potential Games}
\author{
Xin Guo \thanks{University of California, Berkeley, Department of IEOR, email: xinguo@berkeley.edu}
\and
Meng Wang \thanks{Corresponding author. Imperial College London, Department of Mathematics, email: m.wang25@imperial.ac.uk} 
\and Yufei Zhang \thanks{Imperial College London, Department of Mathematics, email: yufei.zhang@imperial.ac.uk}
}
\date{}

\begin{document}

\maketitle 

\begin{abstract}

The recently introduced framework of $\alpha$-potential games facilitates the analysis of finite-player dynamic games by reducing the search for approximate Nash equilibria to the minimization of a single $\alpha$-potential function.
In this work, we show that the approximate potential structure of finite-player games persists in the large-population limit, leading naturally to a notion of $\alpha$-potential mean field games (MFGs).
Specifically, we show that
both the optimal values and the   minimizers of normalized $N$-player 
$\alpha_N$-potential functions converge to those of a mean field control (MFC)
problem with measure-valued controls. We identify the objective   of this limiting MFC problem as an $\alpha_\infty$-potential function for an   MFG, and characterize the parameter  $\alpha_\infty$   analytically.
We further show that  $\lim_{N\to\infty}\alpha_N= 0$ is equivalent to standard conditions for potential  MFGs, in which case the limiting MFC objective becomes a potential function for the   MFG.
{Together, these results provide new constructions  of potential MFGs from   finite-player games  through the asymptotic condition $\lim_{N\to \infty}\alpha_N= 0$, and establish propagation of chaos of potential-selected equilibria from $N$-player games to MFGs with general controlled diffusions, common noise, and non-separable state-control interactions.
}
\end{abstract}



\medskip

\section{Introduction}

\paragraph{Potential games and $\alpha$-potential games.} Potential games provide a powerful framework for analyzing strategic interactions among multiple agents by reducing equilibrium computation to the optimization of a single global objective, known as a potential function. Since the seminal work of Monderer and Shapley \cite{monderer1996potential}, potential structures have played a central role in game theory, optimization, and learning in multi-agent systems. 

In stochastic differential games, however,  identifying exact potential structures remains challenging due to the coupling between players' controls and state dynamics \cite{guo2025towards, learning205potential}. The work of \cite{guo2025Markov} introduces the notion of $\alpha$-potential games, which relaxes the exact potential condition by allowing a controlled approximation error $\alpha$. More precisely, for an $N$-player dynamic game in which player $i$ chooses controls in a set $\cA_i$ to minimize their objective function $J_i$, a function $\Phi$ is called an $\alpha$-potential function for some $\alpha\geq 0$, if for every player $i$, and every unilateral deviation from $u_i\in\cA_i$ to $u_i'\in\cA_i$, 
\begin{equation}
\label{eq:alpha_intro}
\vert J_i(u_i',u_{-i})-J_i(u_i,u_{-i}) - (\Phi(u_i',u_{-i})-\Phi(u_i,u_{-i}))\vert \leq \alpha.
\end{equation}
The case $\alpha=0$ recovers an exact potential game. This framework has proven effective for studying approximate Nash equilibria (NEs) in finite-player dynamic  games, transforming the equilibrium problem into the minimization of an $\alpha$-potential function 
\cite{guo2025Markov, 2025guobsde, guo2025alpha}. 
It has since been adopted and analyzed in a number of dynamic settings, including  
distributed games \cite{guo2025distributed}, transportation systems \cite{di2025alpha}, network games on random graphs \cite{neuman2026potential}, car racing \cite{kalaria2025alpha}, and learning in games
\cite{jordan2026independent,plank2026learning}.

\paragraph{MFGs and potential MFGs.} At the same time, mean field game (MFG) theory provides a tractable {\it asymptotic} framework for studying strategic interactions among a large number of {\it weakly} interacting {\it homogeneous} agents \cite{huang2006large,lasry2007mean}. By passing to the limit as the number of players tends to infinity, MFGs replace finite-player interactions with interactions through the population distribution, leading to  equilibrium problems on the space of probability  measures. 

A particularly relevant subclass is that of potential MFGs, where the mean field equilibrium (MFE) can be characterized through a mean field control (MFC) problem.  
Existing approaches characterize potential MFGs by matching the characterization of MFEs with that of the minimizers of an MFC problem  (see e.g., \cite{cardaliaguet2019master,cardaliaguet2017learning,carmona2018probabilistic1, cecchin2022weak,
graber2025remarks,hofer2026optimal, lasry2007mean}).   This typically requires restrictive structural assumptions to characterize solutions of  MFGs and   MFCs, 
and imposes   that the  cost functions of the MFG are   derivatives of suitable   potential  functions, so that the two characterizations coincide.
Specifically,  in 
\cite[Equation (6.132)]{carmona2018probabilistic1}, a dynamic MFG with a running cost   $f:[0,T]\times\RR^n\times A\times\cP_2(\RR^n\times A)\to\RR$ and terminal cost $g:\RR^n\times\cP_2(\RR^n)\to\RR$  is called a potential MFG if for some functions $F:[0,T]\times \cP_2(\RR^n\times A)\to\RR$ and $G:\cP_2(\RR^n)\to\RR$,
\begin{equation}
\label{eq:symmetric_intro}
\partial_{(x,a)}f(t,x,a,\nu) =  {\partial_\nu F}(t,\nu)(x,a),\quad \partial_x g(x,\mu)= {\partial_\mu G}(\mu)(x),
\end{equation}
where $\partial_\nu F$ and 
$\partial_\mu G$ denote Lions  derivatives in measures.
Equivalent conditions can also be given  in terms of linear functional derivatives 
\cite{cardaliaguet2017learning, graber2025remarks}.
These conditions are imposed directly on the cost functions of MFGs and appear conceptually different from the definition \eqref{eq:alpha_intro} in the $N$-player setting, which   quantifies the impact of players’ strategic deviations on their objectives.

Recent work \cite{hofer2026optimal} studies potential MFGs by
showing that, using the objective $\Phi$ of an MFC problem, one can construct 
a potential MFG such that, when a representative player changes their
strategy from $u$ to $u'$, their objective $J$ satisfies
\begin{equation}
\label{eq:potential_MFG_intro}
J(u';\nu^u)-J(u;\nu^u)
=
\lim_{\varepsilon\downarrow0}
\frac{\Phi(u^\varepsilon)-\Phi(u)}{\varepsilon},
\end{equation}
where $\nu^u$ denotes the population distribution induced by
the control $u$, and $(u^\varepsilon)_{\varepsilon\in(0,1)}$
is a suitable interpolation from $u$ toward $u'$. Although the analysis in \cite{hofer2026optimal} is limited to
the infinite-player setting and does not establish a connection
to $N$-player potential games, the characterization 
\eqref{eq:potential_MFG_intro}   narrows the conceptual gap
between the two frameworks.

\paragraph{Our work.} 

This paper identifies how the approximate potential structure
of finite-player games persists in the large-population limit
and connects    this limit with potential MFGs.
Although one expects suitable limits of finite-player
$\alpha$-potential games to exhibit a mean field potential
structure, the precise connection between these two frameworks
has not  been rigorously
established. This connection is in fact not automatic: the $N$-player
$\alpha$-potential game framework   accommodates general cost
functions that need not satisfy \eqref{eq:symmetric_intro},
so the mean field limit of such games, when it exists, need
not be a potential MFG. Conversely, even when the coefficients
of an MFG satisfy \eqref{eq:symmetric_intro}, the corresponding
finite-player games need not be exact potential games.

 We show that as $N\to\infty$, suitably
normalized $N$-player $\alpha_N$-potential functions satisfying
\eqref{eq:alpha_intro} converge to the objective $\Phi$ of an MFC problem.
This
limiting objective satisfies the  identity \eqref{eq:potential_MFG_intro} up to
an error $\alpha_\infty$,
and therefore defines an $\alpha_\infty$-potential for the limiting MFG.  
Moreover, 
limit points of 
approximate NEs selected by   $N$-player $\alpha_N$-potential functions converge
to $\alpha$-MFEs.
We further show that $\lim_{N\to\infty}\alpha_N=0$ 
is equivalent
to both  \eqref{eq:symmetric_intro} and $\alpha_\infty=0$, thereby recovering the classical notion of a potential MFG.  
Consequently, our result yields new constructions of potential MFGs as
limits of finite-player games under the asymptotic condition
$\lim_{N\to\infty}\alpha_N=0$.

 We emphasize that identifying the limiting MFC objective as an $\alpha_\infty$-potential for the MFG does not follow from standard MFC or MFG theory. In our setting, the two problems are formulated on different spaces of measure-valued controls, making the derivative in \eqref{eq:potential_MFG_intro} substantially more delicate than in the strict-control setting of \cite{hofer2025optimal}. A key technical step of our analysis is therefore to construct a suitable embedding and interpolation that rigorously bridge the two formulations and make the potential identity \eqref{eq:potential_MFG_intro} well defined (see \eqref{eq:alpha_potential_MFG_intro} and Lemma  \ref{lemma:mfg_mfc}). Our framework   accommodates  general nonlinear state--control interactions in cost functions, going beyond the separable cost structures considered in the existing literature \cite{cardaliaguet2018mean,carmona2015probabilistic,djete2023mean,lauriere2022convergence}; see Remark~\ref{rmk:separable}.

\paragraph{Main results.}

Here we  summarize our main results:

\begin{itemize}[wide]
\item

We present in   Theorem~\ref{thm:alpha2} a new construction of $\alpha_N$-potential functions for a class of $N$-player stochastic differential games in which each player's state follows a controlled diffusion driven by idiosyncratic and common noise. The drift and diffusion coefficients   depend nonlinearly on the player's own state and control, while the cost   depends on 
all players' states and controls. The resulting $\alpha_N$-potential function defines a stochastic control problem on the original state space, avoiding the sensitivity processes in \cite{guo2025alpha,guo2025distributed} and the infinite-dimensional McKean–Vlasov formulation in \cite{guo2025alpha}. The corresponding approximation error $\alpha_N$ is characterized analytically.

\item 

We establish the large-population limit of the $N$-player $\alpha_N$-potential functions. We show that, after normalization by $N$, the optimal values of the $N$-player $\alpha_N$-potential functions converge to that of an MFC problem with lifted measure-valued controls. Moreover, approximate minimizers of   $\alpha_N$-potential functions converge (up to subsequences) to optimal controls of   the   MFC problem 
(Theorem \ref{thm:mean_field_limit2}).
The lifted control space has a three-level hierarchy of probability measures, capturing idiosyncratic randomness, temporal oscillations in the control processes, and common noise or additional randomness  emerging in the limit. 
We provide  several  examples in Section \ref{sec:examples} demonstrating such a lifted control space 
is essential for capturing    limits of approximate minimizers of $N$-player $\alpha_N$-potential functions, 
both   with   and without common noise.

\item We identify the objective of the limiting MFC problem as an $\alpha_{\infty}$-potential function for an MFG with measure-valued controls, and   its minimizers as $\alpha_{\infty}$-mean field equilibria (Theorem \ref{thm:alpha_potential_MFG}). This establishes the mean field analogue of Theorem \ref{thm:alpha2}. 
Roughly speaking, given an MFC control $\tilde {\mathbb P}$, when a representative player of the MFG deviates to a control $\tilde {\mathbb P}'$ while the population distribution is held fixed, the limiting MFC objective $\tilde\Phi$ satisfies (cf.~\eqref{eq:potential_MFG_intro}):
\begin{equation}
\label{eq:alpha_potential_MFG_intro}
\left\vert \Tilde{J}(\Tilde{\PP}')-\Tilde{J}(\iota(\Tilde{\PP})) -\lim_{\varepsilon\downarrow 0} \frac{\Tilde{\Phi}(\kappa_{\varepsilon}(\Tilde{\PP}')) - \Tilde{\Phi}(\Tilde{\PP})}{\varepsilon} \right\vert \leq \alpha_{\infty}.
\end{equation}
where 
 $\iota(\tilde \PP)$  embeds  $\Tilde{\PP}$ as an  MFG strategy,
and 
$(\kappa_\varepsilon)_{\varepsilon\in(0,1)}$ is a suitable interpolation between  $\Tilde{\PP}$ and  $\Tilde{\PP}'$  (see Lemma  \ref{lemma:mfg_mfc} for their precise definitions).
The parameter  
$\alpha_{\infty}$ is characterized by the asymmetry of the mixed derivatives of the cost functions with respect to the  representative player's state-control variable    and the population distribution.
We show in Remark~\ref{rmk:pmfg_definition} that 
as $N\to \infty$, 
the difference representation of the $N$-player $\alpha_N$-potential function in \eqref{eq:alpha_intro} naturally leads to the derivative representation of the mean field $\alpha_\infty$-potential in \eqref{eq:alpha_potential_MFG_intro}, and the parameter $\alpha_N$ naturally gives rise to the limiting parameter $\alpha_\infty$.

 \item We  prove in Theorem  \ref{thm:alpha0_equiv} that the condition $\lim_{N\to \infty}\alpha_N=0$  is  equivalent to 
condition \eqref{eq:symmetric_intro} in existing  potential MFG literature.
This along with 
Theorem \ref{thm:alpha_potential_MFG} demonstrates that potential MFGs  arise naturally
as limits of $N$-player  $\alpha_N$-potential games with vanishing $\alpha_N$. 
The  mean field potential function is constructed  explicitly from the cost functions of MFGs, extending previous results for $N$-player potential games with cost functions having symmetric Hessians \cite{monderer1996potential, guo2025towards}.

By further combining Theorem \ref{thm:mean_field_limit2}, we show in  Corollary \ref{cor:mfg_poc} that   the limit points of minimizers of   $N$-player $\alpha_N$-potential functions are MFEs.
This propagation of chaos result 
from  $N$-player games  to MFGs  
 holds for general controlled diffusions with common noise and non-separable state–control interactions in the cost functions. Finally, Section \ref{sec:example_mfg} provides an example  where both the finite-player games and the limiting MFG are potential, as well as an example where the finite-player games admit no exact potential but the limiting MFG does.

\end{itemize}

\paragraph{Comparison with most related works.}

Our work differs from the earlier work  on potential MFGs 
(see e.g., \cite{
lasry2007mean,
carmona2018probabilistic1, cardaliaguet2017learning,cardaliaguet2019master,cecchin2022weak, 
graber2025remarks,hofer2026optimal})  in  three main respects.
First, existing works begin directly with the infinite-player games
and impose   conditions on the cost functions to match the equilibrium characterization with the minimizer characterization of an MFC problem.   
In contrast, we start from an $N$-player game, identify the limit of the $\alpha_N$-potential functions as the objective of an MFC problem, and establish its connection with an MFG.
Second, we introduce the notion of  $\alpha$-potential MFGs, capturing   limits of $N$-player approximate potential games,
and show 
it 
recovers the classical notion of potential MFGs when $\lim_{N\to\infty}\alpha_N = 0$.
Third, existing work on potential MFGs   focuses on strict controls. We instead develop a theory of ($\alpha$-)potential MFGs with measure-valued controls, which is essential for ensuring that limits of approximate Nash equilibria of the $N$-player games correspond to equilibria of the limiting MFG.

Our  convergence analysis of 
$N$-player $\alpha_N$-potential functions 
to an MFC problem 
adopts the  propagation-of-chaos framework  in \cite{djete2022extended}.
The technical difference from \cite{djete2022extended} 
is that 
the $N$-player $\alpha_N$-potential functions involve   cost functions that depend explicitly on the number of players $N$ and   do not coincide with the  cost functions in the limiting MFC problem. 
Consistent with \cite{djete2022extended}, the lifted control space  is essential 
for our analysis
due to the nonlinear dependence of the running cost on the control distribution. Different from \cite{djete2022extended}, we show that such a lifting is necessary even when the state dynamics are independent of the population distribution,
as highlighted by the examples in  Section \ref{sec:examples}.

Our MFG formulation   uses measure-valued controls as in \cite{djete2023mean}, but removes the separability assumption (see \eqref{eq:separable_cost}) on the running cost therein. 
Removing this restriction requires us to keep track of the joint distribution of the representative player's state and control together with the population distribution. 
This leads to a different formulation of the MFG in Definition~\ref{def:measure_control2} and, correspondingly, to a new notion of approximate MFE in Definition~\ref{def:MFG_solution}.

\paragraph{Organization of the paper.}
Section \ref{sec:N_player} introduces  the  finite-player stochastic differential game and constructs the associated $\alpha_N$-potential function. Section \ref{sec:convergence}   establishes the convergence 
of $\alpha_N$-potential functions
to an MFC problem with measure-valued controls. Section \ref{sec:alpha_pmfg} identifies the objective of the   MFC problem as an $\alpha_\infty$-potential function of an MFG with measure-valued controls, and characterizes $\alpha_\infty$ analytically.
Section \ref{sec:consistent_mfg} proves the equivalence between  $\lim_{N\to \infty}\alpha_N=0$  and classical potential MFG conditions, and establishes a propagation of chaos result for approximate   equilibria selected by the $N$-player $\alpha_N$-potential functions. Section \ref{sec:extensions} discusses possible extensions.
Proofs of main results are in Section \ref{sec:proof}.

\paragraph{Notation.}

Let $\NN^*$ be the set of positive integers, and for each $N\in \NN^*$, let $[N]\coloneqq \{1,\ldots, N\}$. 
Given a probability space $(\Omega,\cF,\PP)$ and a Euclidean space $(E,\vert\cdot\vert)$, for a square-integrable $E$-valued random variable $X$, denote $\Vert X\Vert_{L^2}\coloneqq \EE[\vert X\vert^2]^{\frac{1}{2}}$.
Given   $T>0$ and a filtered probability space $(\Omega,\cF,\FF=(\cF_t)_{t\in[0,T]},\PP)$, for each $p\geq 1$ and Euclidean space $(E,\vert\cdot\vert)$, denote by  $\mathcal{S}^p(E)$   the space of $E$-valued $\mathbb{F}$-progressively measurable processes $X:\Omega\times [0,T]\rightarrow E$ satisfying $\Vert X\Vert_{\mathcal{S}^p(E)}\coloneqq \mathbb{E}\left[\sup_{s\in [0,T]}\vert X_s\vert^p \right]^{1/p}<\infty$, and by  $\mathcal{H}^p(E)$   the space of $E$-valued $\mathbb{F}$-progressively measurable processes $X:\Omega\times [0,T]\rightarrow E$ satisfying $\Vert X\Vert_{\mathcal{H}^p(E)}\coloneqq \mathbb{E}\left[\int_0^T\vert X_s\vert^p ds\right]^{1/p}<\infty$.

For each  topological space $E$,  denote by  $\cB(E)$   the Borel $\sigma$-algebra and by $\cP(E)$   the set of all Borel probability measures on $E$. 
Given    topological spaces $E$ and $F$, for  each  Borel  measurable map $f:E\to F$ and each  $\mu\in\cP(E)$, denote by  $\mu\circ f^{-1}\in\cP(F)$  the pushforward of $\mu$ under $f$ such that  $(\mu\circ f^{-1})(A)\coloneqq \mu(f^{-1}(A))$ for all $A\in\cB(F)$.

For each metric space $(E,\rho)$ and $p\geq 1$, let $\cP_p(E)$ be  the set of probability measures with finite $p$-th moments, i.e.,
$ 
\cP_p(E)\coloneqq \left\{\mu\in\cP(E): \int_E \rho(e,e_0)^p\mu(de)<\infty \text{ for some } e_0\in E \right\},
$  equipped with the $p$-Wasserstein distance defined by
\begin{equation*}
\cW_p(\mu,\mu')\coloneqq \left( \inf_{\pi\in\Pi(\mu,\mu')} \int_{E\times E} \rho(e,e')^p \pi(de,de') \right)^{\frac{1}{p}},
\quad  \mu,\mu'\in\cP_p(E),
\end{equation*}
where $\Pi(\mu,\mu')$ denotes the set of all probability measures $\pi$ on $E\times E$ such that $\pi(de,E)=\mu(de)$ and $\pi(E,de')=\mu'(de')$.
Given a function $h:\cP_2(\RR^p) \to \RR$, denote by  $\partial_{\mu}h$ its Lions derivative (L-derivative) as defined in \cite[Definition 5.22]{carmona2018probabilistic1}, by $\frac{\delta h}{\delta \mu}$  its linear functional derivative as defined in \cite[Definition 5.43]{carmona2018probabilistic1}.

We introduce the following function spaces: 
given $m,n\in\NN^*$, a topological space  $E$,  and a  function $f:[0,T]\times E\ni (t,x)\mapsto f(t,x)\in \RR^{m\times n}$ that is measurable in $t$ and continuous in $x$, we define $\Vert f\Vert_{\infty}\coloneqq \esssup_{t\in [0,T]}\sup_{x\in E}\Vert f(t,x)\Vert_{\op}$, with $\Vert\cdot\Vert_{\op}$ being the operator norm of a matrix. For any continuous function $g:E\to \RR^{m\times n}$, define $\Vert g\Vert_{\infty}\coloneqq \sup_{x\in E}\Vert g(x)\Vert_{\op}$. 
Given $p\in\NN^*$,
denote by 
$C(\RR^p)$ (resp.~$C_b(\RR^p)$) the space of continuous 
(resp.~and bounded) functions on $\RR^p$,
  by $C^2(\RR^p)$ the space of twice continuously differentiable functions on $\RR^p$, and by $C_b^2(\RR^p)$ the space of  bounded functions in $C^2(\RR^p)$ with bounded first and second-order derivatives.
  
 For each $p\in\NN^*$,
 let  $\cC^{0,2,2}([0,T]\times\RR^p\times\cP_2(\RR^p))$ be the space of 
measurable functions   $h:[0,T]\times \RR^p\times\cP_2(\RR^p)\ni(t,x,\mu)\mapsto h(t,x,\mu)\in\RR$ such that  
for all $(t,x,\mu)\in [0,T]\times \RR^p\times \cP_2(\RR^p)$, 
\begin{enumerate}[label=(\roman*)]
\item $h(t,\cdot,\mu)$ is twice differentiable, and $\partial_x h(t,\cdot)$ and $\partial_{xx}^2 h(t,\cdot)$ are continuous.
\item   $h(t,x,\cdot)$ is twice   L-differentiable, and   $\RR^p\times\cP_2(\RR^p)\times\RR^p\ni (x,\mu,x')\mapsto \partial_{\mu} h(t,x,\mu)(x') \in\RR^p$ and $\RR^p\times\cP_2(\RR^p)\times\RR^p\times\RR^p\ni (x,\mu,x',x'')\mapsto \partial_{\mu}^2 h(t,x,\mu)(x',x'')\in\RR^{p\times p}$ are continuous.
\item  $\RR^p\times\cP_2(\RR^p)\times\RR^p\ni (x,\mu,x')\mapsto \partial_{\mu} h(t,x,\mu)(x')$ is differentiable in $(x,  x')$, and $\RR^p\times\cP_2(\RR^p)\times\RR^p\ni (x,\mu,x')\mapsto (\partial_x\partial_{\mu} h(t,x,\mu)(x'),\partial_{x'}\partial_{\mu} h(t,x,\mu)(x'))$ is continuous.
\item There exists   $C>0$ such that for all $(t,x,\mu,x')\in [0,T]\times\RR^p\times\cP_2(\RR^p)\times\RR^p$,
$
\vert \partial_x h(t,x,\mu)\vert + \vert \partial_{\mu} h(t,x,\mu)(x')\vert \leq  C\left(1+\vert x\vert +\vert x'\vert +\int_{\mathbb{R}^p}\vert x'' \vert  \mu(dx'') \right)
$. Moreover, 
  $\partial^2_{\mu}h$ and $\partial_x\partial_{\mu}h$ are bounded.
\end{enumerate}
Note that 
by \cite[Remark 4.16]{carmona2018probabilistic2}, for all $(t,x)\in[0,T]\times\RR^p$, $\cP_2(\RR^p)\ni\mu\mapsto \partial_x h(t,x,\mu)$ is L-differentiable, and $\partial_{\mu}\partial_x h(t,x,\mu)(x') =[\partial_x\partial_{\mu} h(t,x,\mu)(x')]^{\top}$. 
We denote by   $  \cC^{2,2}(\RR^p\times\cP_2(\RR^p))$
 the space of time-independent functions
in   $\cC^{0,2,2}([0,T]\times\RR^p\times\cP_2(\RR^p))$.

\section{$\alpha$-Potential Functions for $N$-Player Games}
\label{sec:N_player}

This section introduces the class of finite-player stochastic differential games considered in this paper, presents a new construction of $\alpha$-potential functions  
formulated as    finite-dimensional stochastic control problems, and characterizes the   upper bound of $\alpha$ analytically.  

\subsection{Formulation of $N$-player games}

We consider a sequence of $N$-player games defined on the same probability space.
Let $T>0$ be a given terminal time,  $n,d,k,l\in\mathbb{N}^*$, and $(\Omega, \mathcal{F},\mathbb{F}=(\mathcal{F}_t)_{t\in[0,T]},\mathbb{P})$ be a   filtered probability space 
satisfying the usual
conditions, 
which supports the following mutually independent objects: a sequence $(\xi_i)_{i\in\NN^*}$  of independent $\mathbb{R}^n$-valued random variables with laws $(\nu_i)_{i\in\NN^*}\subset\cP_2(\RR^n)$, a sequence $(W_i)_{i\in\NN^*}$  of independent $d$-dimensional Brownian motions, and an $l$-dimensional Brownian motion $B$. Here $\xi_i$ and $W_i$   represent the initial state and idiosyncratic noise  of player $i$'s state dynamics, 
respectively, 
and $B$ represents the common noise affecting all players' state dynamics.

 For each $N\in\NN^*$, 
 we define the admissible strategies of each player in the
$N$-player game. 
Let $\mathbb{F}^N= (\mathcal{F}_t^N)_{t\in [0,T]}$ 
be the filtration 
such that 
 for all $t\in [0,T]$,  
 $
\mathcal{F}_t^N\coloneqq \sigma\{(\xi_i)_{i\in[N]},(W_{t\wedge\cdot,i})_{i\in[N]},B_{t\wedge\cdot} \}\vee \mathcal N,
$
where $\mathcal N$ is the collection of   $\mathbb P$-null sets. For each $i\in [N]$, let $\cA_i^N$ be the set of 
player $i$'s admissible  controls,
consisting of 
processes $u_i\in\mathcal{H}^2(\mathbb{R}^k)$ that are progressively measurable with respect to $\FF^N$
and take values in 
   a measurable set  $A_i\subset \mathbb{R}^k$. Let 
$\cA^N\coloneqq \prod_{i\in[N]}\cA_i^N$ be the  set of joint control profiles
of all players, and for each $i\in[N]$, let $\mathcal{A}_{-i}^N\coloneqq \prod_{j\in[N]\backslash\{i\}}\cA_j^N$ be the   set of control profiles of all players except player $i$.
We denote by 
 $\bu=(u_j)_{j\in[N]}$ and $\bu_{-i}=(u_j)_{j\in[N]\backslash\{i\}}$   a generic element of $\mathcal{A}^N$ and $\mathcal{A}_{-i}^N$, respectively. 
 
 For each $i\in[N]$ 
 and $u_i\in \cA_i^N$, 
 player $i$ considers the following state dynamics: 
 \begin{equation}\label{eq:state}
dX_{t,i} = b_i(t,X_{t,i},u_{t,i})dt + \sigma_i(t,X_{t,i},u_{t,i}) dW_{t,i} + \gamma_i(t,X_{t,i},u_{t,i})dB_t,
\quad 
t\in [0,T]; 
\quad X_{0,i}=\xi_i,
\end{equation}
where $b_i:[0,T]\times \mathbb{R}^n  \times \mathbb{R}^k\rightarrow\mathbb{R}^n,\sigma_i:[0,T]\times \mathbb{R}^n  \times \mathbb{R}^k\rightarrow\mathbb{R}^{n\times d}$ and $\gamma_i:[0,T]\times \mathbb{R}^n  \times \mathbb{R}^k\rightarrow\mathbb{R}^{n\times l}$ are given measurable functions.
Player  
 $i$ determines their optimal strategy by minimizing the following cost functional $J_i:\mathcal{A}^N\rightarrow\mathbb{R}$:
\begin{equation}
\label{eq:N_player_cost}
\inf_{u_i\in \cA_i^N} J_i(\bu), \quad \textnormal{with} \quad
J_i(\bu) \coloneqq \mathbb{E}\left[\int_0^T f_i(t,\bX_t^{\bu},\bu_t) dt + g_i(\bX_T^{\bu})\right],
\end{equation}
where 
$\bX_t^{\bu} = (X_{t,i}^{u_i})_{i\in [N]}$, and 
$f_i:[0,T]\times\mathbb{R}^{Nn}\times \mathbb{R}^{Nk} \rightarrow \mathbb{R}$ and $g_i:\mathbb{R}^{Nn} \rightarrow \mathbb{R}$ are measurable functions.

We impose the following standard regularity conditions on  the   coefficients,
which ensure  that 
for all $u_i\in \cH^2(\mathbb R^k)$, 
\eqref{eq:state}  admits a unique strong solution  $X^{u_i}_i\in \mathcal{S}^2(\mathbb R^n)$, and the cost functional 
$J_i:\mathcal{A}^N\rightarrow\mathbb{R}$  is well defined.

\begin{assumption}\label{assumption:SDE}
 
  Each  $\phi\in\{b_i,\sigma_i,\gamma_i\}_{i\in \mathbb N^*}$ is Lipschitz continuous in $(x,a)$, uniformly in $t$, and $\int_0^T \vert \phi(t,0,0)\vert^2 dt <\infty$.
 For all $i\in \mathbb N^*$ and $t\in[0,T]$, $\mathbb{R}^{Nn}\times \mathbb{R}^{Nk} \ni (x,a)\mapsto (f_i(t,x,a),g_i(x))\in\mathbb{R}\times\mathbb{R}$ is twice continuously differentiable. Moreover, there exists $C\ge 0$ such that for all $(t,x,a)\in [0,T]\times \mathbb R^{Nn}\times \mathbb R^{Nk}$,  
 $$
 |f_i(t,0,0)|+
 |\partial_{(x,a)}f_i(t,0,0)|
 +
 |\partial^2_{(x,a)(x,a)}f_i(t,x,a)|
 + |g_i(0)| + |\partial_x g_i(0)|+ |\partial^2_{xx}g_i(x)|\le C.
 $$

\end{assumption}

To characterize the rational behavior in the   $N$-player game 
defined by 
\eqref{eq:N_player_cost}, 
we recall the solution concept of $\varepsilon$-Nash equilibrium,
defined as a joint control profile in which no player can improve their performance by more than $\varepsilon$ through any unilateral deviation.

\begin{definition}
For any $\varepsilon\geq 0$, $\bu=(u_i)_{i\in [N]}\in\mathcal{A}^N$ is  an $\varepsilon$-Nash equilibrium 
 ($\varepsilon$-NE)
of the $N$-player  game  if 
$$J_i(\bu)\leq J_i(u_i',\bu_{-i}) + \varepsilon, \quad \forall i\in [N], u_i'\in\mathcal{A}_i^N.$$
When $\varepsilon=0$, $\bu$ is called   a Nash equilibrium (NE).
\end{definition}

We analyze the $N$-player game
\eqref{eq:N_player_cost}
using the $\alpha$-potential game framework introduced in  \cite{guo2025alpha},
which   reduces the study of approximate NEs to the minimization of a single 
$\alpha$-potential
function.

\begin{proposition}[{\cite[Proposition 2.1]{guo2025alpha}}]
\label{prop:alpha_PG_Nplayer}
Let $\Phi:\mathcal{A}^N\rightarrow\mathbb{R}$  be an 
  $\alpha$-potential function of the $N$-player game   with some $\alpha\ge 0$, namely, for all $i\in [N]$, $u_i,u_i'\in\mathcal{A}_i^N$, and $\bu_{-i}\in\mathcal{A}_{-i}^N$,
\begin{equation}
\label{eq:N_player_alpha_PG}
\vert (J_i(u_i',\bu_{-i})-J_i(u_i,\bu_{-i})) - \left(\Phi(u_i',\bu_{-i})-\Phi(u_i,\bu_{-i}) \right) \vert \leq \alpha.
\end{equation}
If 
$\overline{\bu}\in\mathcal{A}^N$ satisfies $\Phi(\overline{\bu})\leq \inf_{\bu\in\mathcal{A}^N}\Phi(\bu)+\varepsilon$ for some $\varepsilon\ge 0$, then $\overline{\bu}$ is an $(\alpha+\varepsilon)$-NE of   the $N$-player game. 
\end{proposition}

\subsection{Construction of $N$-player $\alpha_N$-potential functions}

A crucial step in applying the $\alpha$-potential game framework is to identify some
$\alpha$-potential function and quantify its associated $\alpha$. For general stochastic differential games, 
\cite[Theorem 3.1]{guo2025alpha}  constructs one   $\alpha$-potential function by aggregating the derivatives of all players' objective functions  over the control space, which are expressed in terms of the derivatives of the state processes with respect to the controls. The resulting 
$\alpha$-potential function is   reformulated as a McKean-Vlasov control problem on  an  infinite-dimensional measure space.

Here
we construct a new $N$-player $\alpha$-potential function  for the game \eqref{eq:N_player_cost}  by leveraging the decoupled structure of the state processes, yielding a control problem with the original finite-dimensional states governed by \eqref{eq:state}.

\begin{theorem}\label{thm:alpha2}
Suppose Assumption  \ref{assumption:SDE} holds. Define $\Phi:\mathcal{A}^N\rightarrow\mathbb{R}$ by
\begin{equation}\label{eq:alpha_potential}
\Phi(\bu) = \mathbb{E}\left[\int_0^T F(t,\bX_t^{\bu},\bu_t) dt + G(\bX_T^{\bu})\right],
\end{equation}
where  
$F:[0,T]\times\mathbb{R}^{Nn}\times\mathbb{R}^{Nk}\rightarrow\mathbb{R}$ and $G:\mathbb{R}^{Nn}\rightarrow\mathbb{R}$
are defined by
 
\begin{equation}
\label{eq:F_G_bar}
\begin{aligned}
F(t,x,a) \coloneqq  &\int_0^1\sum_{i\in[N]} \begin{pmatrix}\partial_{x_i} f_i  \\ \partial_{a_i} f_i \end{pmatrix}^{\top}\left(t,rx ,ra\right) \begin{pmatrix}  x_i \\a_i \end{pmatrix} dr,\quad G(x) \coloneqq  \int_0^1 \sum_{i\in[N]}(\partial_{x_i} g_i)^{\top}(rx)  x_i  dr.
\end{aligned}
\end{equation}
Then $\Phi$ is an $\alpha_N$-potential function of the game \eqref{eq:N_player_cost},  with
\begin{align*}
\alpha_N \leq & \frac{1}{2}\sup_{i\in[N]}\sup_{u_i'\in\cA_i^N,\bu\in\cA^N}\sum_{j\in[N]\backslash\{i\}} \left\{    \Vert \partial_{(x_i,a_i)(x_j,a_j)}^2\Delta_{i,j}^f\Vert_{\infty}  \left\| \begin{pmatrix} X_i^{u_i'}-X_i^{u_i} 
\\   u_i'-u_i
\end{pmatrix}
\right\|_{\mathcal{H}^2(\mathbb{R}^{n+k})}
 \left\| \begin{pmatrix} X_j^{u_j} \\   u_j
\end{pmatrix}
\right\|_{\mathcal{H}^2(\mathbb{R}^{n+k})} \right.
\\
&\left.+ \Vert \partial_{x_ix_j}^2\Delta_{i,j}^g \Vert_{\infty} \Vert 
X_{T,i}^{u_i'}-X_{T,i}^{u_i}\Vert_{L^2} \Vert X_{T,j}^{u_j} \Vert_{L^2}\right\},
\end{align*}
with $\Delta_{i,j}^f \coloneqq f_i-f_j$ and $\Delta_{i,j}^g \coloneqq g_i-g_j$,  for all $i,j\in[N]$.
\end{theorem}

The $\alpha$-potential function \eqref{eq:alpha_potential} yields    an $Nn$-dimensional  control problem, whose state process $\bX^{\bu}$ evolves according to the   joint state dynamics 
\eqref{eq:state}.
It is constructed via a path integral in the \textit{state–control product space}, which differs from the existing constructions in
 \cite{guo2025alpha,guo2025distributed}
that use a path integral in the control space only. 
This construction exploits the decoupled structure of the players’ state dynamics, 
which removes the (nonlinear) dependence of player $i$'s   state on other players’ control perturbations.
Moreover,
by working on the enlarged state–control space, we avoid differentiating each player’s state dynamics with respect to their own control and can accommodate non-differentiable state coefficients. 
 Similar constructions have been used  in \cite{guo2025towards, di2025alpha} for differential  games with closed-loop controls. 
 
In the sequel, we analyze the convergence of the $\alpha_N$-potential function \eqref{eq:alpha_potential} and its minimizer as $N\to \infty$, 
and 
  study its connection with potential MFGs.
To facilitate the convergence analysis, we 
 impose homogeneity and weak interaction assumptions. 

 \begin{assumption}\label{assumption:mean_field}
\begin{enumerate}[label=(\roman*)] 
\item 
 $A_i=A$ for all $i\in [N]$.
There exists $b:[0,T]\times\mathbb{R}^n\times \mathbb{R}^k\rightarrow\mathbb{R}^n$, $\sigma:[0,T]\times\mathbb{R}^n\times \mathbb{R}^k\rightarrow\mathbb{R}^{n\times d}$ 
and $\gamma\in \mathbb{R}^{n\times l}$, 
such that for all $i\in [N]$, 
 $(b_i,\sigma_i,\gamma_i)=(b,\sigma, \gamma) $ and satisfy   Assumption \ref{assumption:SDE}. 

\item
\label{item:f_g_i}
There exists $f\in\cC^{0,2,2}([0,T]\times\RR^{n+k}\times\cP_2(\RR^{n+k}))$ and $g\in\cC^{2,2}(\RR^n\times\cP_2(\RR^n))$ such that  for all  $i\in[N]$, $t\in[0,T]$, $x=(x_j)_{j\in[N]}\in \mathbb{R}^{Nn}$, and $a=(a_j)_{j\in[N]} \in \mathbb{R}^{Nk}$,
\begin{equation*}
f_i(t,x,a) = f\left(t,x_i,a_i, \frac{1}{N}\sum_{j\in[N]}\delta_{(x_j,a_j)} \right),\quad g_i(x) = g\left(x_i, \frac{1}{N}\sum_{j\in[N]}\delta_{x_j}\right).
\end{equation*}
\item
\label{item:non-degenrate}
 $A$ is a compact set. The functions $b$ and $\sigma$ are bounded, $n=d$, and there exists a constant $\theta>0$ such that for all $t\in[0,T]$, $x\in\mathbb{R}^n$, $a\in A$, and $\eta\in \RR^n$, $\eta^\top \sigma\sigma^{\top}(t,x,a)\eta \geq \theta |\eta|^2$.
\end{enumerate}
\end{assumption}

\begin{remark}[Non-separable running cost]
\label{rmk:separable}
Assumption \ref{assumption:mean_field}\ref{item:f_g_i} allows for general  running costs that   depend  nonlinearly on the joint law of the state and control processes. In particular, 
it  does not impose a separable structure on  $f$, namely, 
\begin{equation}
\label{eq:separable_cost}
f(t,x,a,\nu)=f'(t,x,\nu)+f''(t,x,a,\nu(\cdot\times A)),
\end{equation}
which is often assumed in the MFG literature \cite{carmona2015probabilistic, cardaliaguet2018mean, lauriere2022convergence, djete2023mean}.
 The non-degeneracy of the idiosyncratic noise in Assumption \ref{assumption:mean_field}\ref{item:non-degenrate}  is    consistent with \cite[Assumption 2.1]{djete2022extended},
 and  can be removed when $f$  depends only on the state law (see e.g., \cite{djete2022mckean2} or \cite{lacker2017limit}).

\end{remark}

Observe that 
the magnitude of the $\alpha_N$-potential function $\Phi$ scales linearly with $N$,
so 
it must be normalized by $1/N$ in order to obtain a limit. 
Indeed, 
under Assumption \ref{assumption:mean_field}, 
the functions $F$ and $G$ in \eqref{eq:alpha_potential} can be written as 
\begin{equation}
\label{eq:F_G_homo}
F(t,x,a) =
\sum_{i\in[N]} F^N\left(t,x_i,a_i, \nu^N_{(x,a)}  \right),
\quad 
G(x) =
\sum_{i\in[N]} G^N\left(x_i, \mu^N_{x}  \right),
\end{equation}
where
$\nu^N_{(x,a)}= \frac{1}{N}\sum_{j\in[N]}\delta_{(x_j,a_j)}$, 
$\mu^N_{x}= \frac{1}{N}\sum_{j\in[N]}\delta_{ x_j}$, 
and  
$F^N:[0,T]\times\RR^n\times\RR^k\times\cP_2(\RR^n\times\RR^k)\to\RR$ and $G^N:\RR^n\times\cP_2(\RR^n)\to\RR$  
are given by
\begin{align}
\label{eq:F_G_N}
\begin{split}
 &F^N(t,x,a,\nu) 
 \\
 &\coloneqq \int_0^1 \begin{pmatrix}x  \\ a \end{pmatrix}^{\top} \left\{ \begin{pmatrix}\partial_x f \\ \partial_a f \end{pmatrix}\left(t,rx,ra, \nu\circ T_r^{-1} \right) + \frac{1}{N} \partial_{\nu}f \left(t,rx,ra, \nu\circ T_r^{-1}\right)(rx,ra) \right\} dr, \\
 & G^N(x,\mu) \coloneqq  \int_0^1 x^{\top}\left\{\partial_x g\left(rx, \mu\circ (T_r')^{-1}\right) + \frac{1}{N}\partial_{\mu} g \left(rx, \mu\circ (T_r')^{-1} \right) (rx)\right\}dr,
\end{split}
\end{align}
where for all $r\in [0,1]$, 
$T_r(x,a) \coloneqq (rx,ra)$, and    $T_r'(x) \coloneqq rx$.
It is clear from  \eqref{eq:F_G_N} that  
one has to normalize $F$ and $G$   by $1/N$ in order to obtain a limit.

\section{Convergence of Normalized   $\alpha_N$-Potential Functions to MFCs}\label{sec:convergence}

This section shows that as  $N\to \infty$,  
the optimal values of
$N$-player game $\alpha_N$-potential functions from  Theorem \ref{thm:alpha2},
 \emph{normalized by the number of players $N$}, 
   converge to the value of a mean-field control (MFC) problem with measure-valued controls. Approximate minimizers of $\alpha_N$-potential functions also converge (up to subsequences) to optimal controls of the limiting MFC problem.

\subsection{Formulation of the limiting   MFC problem}

Instead of using strict controls as in the $N$-player setting, 
we  define the limiting MFC problems using measure-valued controls following \cite{djete2022extended}. This is essential for proving the convergence of approximate minimizers of the 
$\alpha_N$-potential functions,
especially in the present setting 
where the cost function $f$ depends nonlinearly  on the control law.
In this case,  the limits of approximate minimizers of the 
$N$-player  $\alpha_N$-potential functions may fail to belong to the set of strict controls, 
as shown in Examples \ref{eg:common_noise} and \ref{eg:no_common_noise}.

   For each Polish space  $(E,\rho)$,  let $\mathbb{M}(E)$ denote the space of all Borel measures $q(dt,de)$ on $[0,T]\times E$ whose marginal distribution on $[0,T]$ is the Lebesgue measure $dt$,  i.e.,   $q(dt,de)=q(t,de)dt$ for a family $(q(t,de))_{t\in [0,T]}$ of Borel probability measures on $E$. Denote by  $\Lambda$   the canonical element on $\mathbb{M}(E)$. Fix an arbitrary $e_0\in E$, and  define
\begin{equation*}
\Lambda^t(ds,de)\coloneqq \Lambda(ds,de)\vert_{[0,t]\times E} + \delta_{e_0}(de)ds\vert_{(t,T]\times E}, \quad \forall t\in [0,T].
\end{equation*}
Consider the canonical space $\Tilde{\Omega}\coloneqq C([0,T];\cP(\RR^n)) \times \MM(\cP(\RR^n\times A)) \times C([0,T];\RR^l)$. Let $(\Tilde{\mu},\Tilde{\Lambda},\Tilde{B})$ be     the canonical element of $\Tilde{\Omega}$, and define the canonical filtration $\Tilde{\mathbb{F}}=(\Tilde{\mathcal{F}}_t)_{t\in [0,T]}$ such that  
$
\Tilde{\mathcal{F}}_t\coloneqq \sigma\left\{\Tilde{\mu}_{t\wedge\cdot},\Tilde{\Lambda}^t,\Tilde{B}_{t\wedge\cdot} \right\}
$
for all $t\in[0,T]$.
By \cite[Lemma 3.2]{lacker2015mean}, there exists a disintegration $\Tilde{\Lambda}(dm,dt)=\Tilde{\Lambda}_t(dm)dt$ with $(\Tilde{\Lambda}_t)_{t\in [0,T]}$ being a $\cP(\cP(\RR^n\times A))$-valued $\Tilde{\mathbb{F}}$-predictable process. Here $\Tilde{\mu}$ represents the conditional state laws given the common noise, $\Tilde{\Lambda}$ represents 
the law of the conditional joint state–control law,
and   $\Tilde{B}$ represents the common noise.

The admissible measure-valued controls are suitable probability   measures on $\Tilde{\Omega}$.
Specifically,    define the generator $\Tilde{\mathcal{L}}$ such that  for all $(t,x,a)\in [0,T]\times\mathbb{R}^n\times A$ and   $\varphi\in C^2(\mathbb{R}^n)$,
\begin{equation*}
\Tilde{\mathcal{L}}_t \varphi(x,a)\coloneqq b(t,x,a)^{\top} \nabla\varphi(x) + \frac{1}{2}\Tr\left[\sigma\sigma^{\top}(t,x,a)\nabla^2\varphi(x) \right],
\end{equation*}
and   for all $f\in C_b^2(\mathbb{R}^n)$, $\Tilde{\mu}\in C([0,T];\cP(\RR^n))$, and $\Tilde{\Lambda}\in\MM(\cP(\RR^n\times A))$, define the process $(N_t(f,\Tilde{\mu},\Tilde{\Lambda}))_{t\in [0,T]}$ on $\Tilde{\Omega}$ by  
\begin{align*}
N_t(f,\Tilde{\mu},\Tilde{\Lambda})
&\coloneqq \langle f(\cdot-\gamma \Tilde{B}_t),\Tilde{\mu}_t\rangle - \langle f,\Tilde{\mu}_0\rangle 
\\
&\quad - \int_0^t\int_{\cP(\RR^n\times A)}\int_{\mathbb{R}^n\times A} \Tilde{\mathcal{L}}_r[f(\cdot-\gamma \Tilde{B}_r)](x,a)m(dx,da)\Tilde{\Lambda}_r(dm)dr,
\end{align*}
where $\langle h,\mu\rangle =\int_{\RR^n} h(x)\mu (dx)$
for any  $\mu\in \cP(\RR^n)$ and  bounded measurable function $h:\RR^n\to\RR$.
For each $\pi\in\mathcal{P}(\mathbb{R}^n)$, let  $\mathbb{Z}_{\pi} $ be the set of probability measures
on $\mathbb{R}^n\times A$ whose   $\mathbb{R}^n$-marginal is  $\pi$, i.e.,
$
\mathbb{Z}_{\pi}\coloneqq \left\{ m\in \cP(\RR^n\times A)\mid m(dx,A)=\pi(dx) \right\}.
$
We now give the precise   definition of the admissible measure-valued controls, following \cite[Definition 2.6]{djete2022extended}.

\begin{definition}\label{def:measure_control}

Given an initial distribution  $\nu\in\mathcal{P}_2(\mathbb{R}^n)$, 
$\Tilde{\mathbb{P}}\in\mathcal{P}_2(\Tilde{\Omega})$ is an admissible  measure-valued control 
for  $\nu$ if:
\begin{enumerate}[label=(\roman*)]
\item
\label{item:initial_state}
 $\Tilde{\mathbb{P}}(\Tilde{\mu}_0=\nu)=1$.
\item
\label{item:state_martingale}
 $(\Tilde{B}_t)_{t\in[0,T]}$ is a $(\Tilde{\mathbb{P}},\Tilde{\mathbb{F}})$-Wiener process starting at zero, and for $\Tilde{\mathbb{P}}$-almost every $\omega\in\Tilde{\Omega}$, $N_t(f,\Tilde{\mu},\Tilde{\Lambda})=0$ for all $f\in C_b^2(\mathbb{R}^n)$ and all  $t\in [0,T]$.
\item
\label{item:state_marginal} 
 For $d\Tilde{\mathbb{P}}\otimes dt$ almost every $(t,\omega)\in [0,T]\times\Tilde{\Omega}$, $\Tilde{\Lambda}_t(\mathbb{Z}_{\Tilde{\mu}_t})=1$.
\end{enumerate}
We denote by $\cP_V(\nu)$ the set of all admissible measure-valued controls for  $\nu$.
\end{definition}

Condition \ref{item:initial_state}  in Definition \ref{def:measure_control}  
requires  that 
the flow of {(conditional)} state laws $\Tilde{\mu}$ at $t=0$ agrees with the given initial condition    $\nu$.
Condition \ref{item:state_martingale} requires that 
under $\Tilde{\mathbb{P}}$, 
the controlled state process solves a conditional martingale problem, which is a weak formulation of the original controlled state dynamics. 
Condition \ref{item:state_marginal}  requires that  under $\Tilde{\mathbb{P}}$, 
the measure $\Tilde{\Lambda}_t$ concentrates on the set $\mathbb{Z}_{\Tilde{\mu}_t}$ for each $t\in [0,T]$, 
implying that 
the first marginal of the conditional  state-control  distribution coincides with $\Tilde{\mu}_t$. 

We now introduce the   MFC problem associated with   $\cP_V(\nu)$.
Define the  objective function $\Tilde{\Phi}:\mathcal{P}_2(\Tilde{\Omega})\rightarrow\mathbb{R}$ such that 
 for all $\Tilde{\mathbb{P}}\in\mathcal{P}_2(\Tilde{\Omega})$,
\begin{align}\label{eq:mf_potential_measure}
\begin{split}
\Tilde{\Phi}(\Tilde{\mathbb{P}})&\coloneqq \mathbb{E}^{\Tilde{\mathbb{P}}}\bigg[ \int_0^T \int_{\cP(\RR^n\times A)}\int_{\mathbb{R}^n\times A} F^{\infty}(t,x,a,m)m(dx,da)\Tilde{\Lambda}_t(dm)dt \\
&\quad + \int_{\mathbb{R}^n} G^{\infty}(x,\Tilde{\mu}_T)\Tilde{\mu}_T(dx)  \bigg],
\end{split}
\end{align}
where $F^{\infty}:[0,T]\times \mathbb{R}^n\times \mathbb{R}^k\times \mathcal{P}_2(\mathbb{R}^n\times \mathbb{R}^k )\rightarrow \mathbb{R}$ and $G^{\infty}:\mathbb{R}^n\times \mathcal{P}_2(\mathbb{R}^n)\rightarrow\mathbb{R}$ are defined by
\begin{equation}\label{eq:mf_potential_cost}
\begin{aligned}
F^{\infty}(t,x,a,\nu) &  \coloneqq \int_0^1 \begin{pmatrix} x \\ a \end{pmatrix}^{\top}\begin{pmatrix} \partial_x f \\ \partial_a f \end{pmatrix}\left(t,rx,ra, \nu\circ T_r^{-1} \right) dr, \\
G^{\infty}(x,\mu) &\coloneqq \int_0^1 x^{\top}\partial_x g(rx,\mu\circ (T_r')^{-1})dr,
\end{aligned}
\end{equation}
where $T_r$ and $T_r'$ are the same maps as in \eqref{eq:F_G_N}. 
The functions $F^\infty$ and $G^\infty$ 
are pointwise  limits of $F^N$ and $G^N$   in \eqref{eq:F_G_N} as $N\to\infty$. 
Consider the following minimization problem
\begin{equation}\label{eq:mf_lifted}
V_V(\nu)\coloneqq \inf_{\Tilde{\mathbb{P}}\in\cP_V(\nu)}\Tilde{\Phi}(\Tilde{\mathbb{P}}). 
\end{equation}

\subsection{Convergence of  normalized $N$-player $\alpha_N$-potential functions}

To connect the $N$-player $\alpha_N$-potential function \eqref{eq:alpha_potential} with the MFC objective \eqref{eq:mf_potential_measure}, 
 we  normalize the 
$\alpha_N$-potential function by $1/N$, and embed the set $\mathcal A^N$ of strict controls in the set  of measure-valued controls.
Specifically, let $\Phi: \cA^N\to \mathbb R$ be given by \eqref{eq:alpha_potential},  define 
 \begin{equation}
 \label{eq:scaled_N_potential}
 \Phi^N(\bu) \coloneqq \frac{1}{N}\Phi(\bu), \quad \bu\in \cA^N.
 \end{equation}
 For each   $\bu^N=(u_i^N)_{i\in [N]}\in\mathcal{A}^N$, 
 let  $(X_i^{u_i^N})_{i\in[N]}$ be the associated state processes, 
 and define  
\begin{equation}
\label{eq:P_u_N}
\Tilde{\mathbb{P}}^N(\bu^N)\coloneqq \mathbb{P} \circ\left( (\varphi_t^{N,X})_{t\in [0,T]}, \delta_{\varphi_t^N}(dm)dt, B  \right)^{-1} \in \cP_2(  \Tilde{\Omega}), 
\end{equation}
where
$\varphi_t^{N,X}\coloneqq \frac{1}{N}\sum_{i\in [N]} \delta_{X_{t,i}^{u_i^N}},$and $ \varphi_t^N \coloneqq \frac{1}{N}\sum_{i\in [N]} \delta_{(X_{t,i}^{u_i^N},u_{t,i}^N)}.
$ 

The following theorem establishes the convergence of the optimal values of $\Phi^N$   to that of the MFC problem \eqref{eq:mf_lifted}, as well as the convergence (up to subsequences) of approximate minimizers of $\Phi^N$ to optimal controls of the MFC problem.

\begin{theorem}\label{thm:mean_field_limit2}
Suppose Assumption \ref{assumption:mean_field} holds and there exists $p>2$ such that  $(\nu_i)_{i\in\mathbb{N}^*}\subset \mathcal{P}_p(\mathbb{R}^n)$ and  $\sup_{N\geq 1}\frac{1}{N}\sum_{i=1}^N \int_{\mathbb{R}^n}\vert x'\vert^p \nu_i(dx')<\infty$. 
\begin{enumerate}[label=(\roman*)]
\item \label{item:converge_minimizer}
Let $(\varepsilon_N)_{N\in\mathbb{N}^*}\subset  (0,\infty)$ satisfy $\lim_{N\rightarrow\infty}\varepsilon_N=0$, 
and for each $N\in \NN^*$, let $\bu^N \in \cA^N$ satisfy 
$
\Phi^N(\bu^N)\leq \inf_{\bu\in\cA^N}\Phi^N(\bu)+\varepsilon_N,
$
and define  $\Tilde{\mathbb{P}}^N\coloneqq \Tilde{\mathbb{P}}^N(\bu^N)$ as in \eqref{eq:P_u_N}. 
Then $(\Tilde{\mathbb{P}}^N)_{N\in\mathbb{N}^*}\subset \mathcal P_2(\Tilde{\Omega})$ is relatively compact, and 
 for any convergent subsequence $(\Tilde{\mathbb{P}}^{N_m})_{m\in \mathbb{N}^*}$,
 there exists $ \nu\in \mathcal{P}_p(\mathbb{R}^n)$ and $\Tilde{\mathbb{P}}^{\infty} \in \arg\min_{\Tilde{\mathbb{P}}\in\cP_V(\nu)}\Tilde{\Phi}(\Tilde{\mathbb{P}})$ such that  
 $\lim_{m\rightarrow\infty}\mathcal{W}_2\left( \frac{1}{N_m} \sum_{i=1}^{N_m}\nu_i,\nu\right)=0$ and $ \lim_{m\rightarrow\infty} \mathcal{W}_2\left(\Tilde{\mathbb{P}}^{N_m},\Tilde{\mathbb{P}}^{\infty}\right)=0.
 $
\item\label{item:converge_value}  $ 
\lim_{N\rightarrow\infty}\left\vert  \inf_{\bu\in\mathcal{A}^N} \Phi^N(\bu) - V_V\left(\frac{1}{N}\sum_{i=1}^N \nu_i \right) \right\vert = 0.
$ 
\end{enumerate}
\end{theorem}

The proof of Theorem \ref{thm:mean_field_limit2} adapts the  arguments for propagation of chaos of MFC problems from 
\cite{djete2022extended} to the present setting. 
The main difference  from \cite{djete2022extended} is that the 
$N$-player cost functions  $F^N$ and  $G^N$  in 
\eqref{eq:F_G_N}
depend explicitly on $N$, and  do not coincide with the limiting cost functions    $F^{\infty}$ and  $G^{\infty}$ in \eqref{eq:mf_potential_cost}.
This is addressed by showing the discrepancies due to the cost-function mismatch are  of the order $O({1}/{N})$,
uniformly over controls, 
 and hence vanish as $N\to \infty$.

\section{Identifying MFC objective   as an $\alpha_{\infty}$-Potential
for MFGs}\label{sec:alpha_pmfg}

In this section, we show that the MFC objective \eqref{eq:mf_potential_measure}, obtained in Theorem~\ref{thm:mean_field_limit2} as the limit of $N$-player $\alpha_N$-potential functions, is an $\alpha_\infty$-potential function for an MFG and that its minimizers are $\alpha_\infty$-mean field equilibria. This establishes the mean-field analogue of Theorem~\ref{thm:alpha2} and leads to a notion of $\alpha$-potential MFGs, formalized in Theorem~\ref{thm:alpha_potential_MFG}. Here, $\alpha$ quantifies the discrepancy between the change in the representative player's objective induced by a control perturbation and the derivative of the $\alpha$-potential function along a suitable interpolation path.
 
\subsection{Formulation of MFGs}

The MFG considered here differs from the standard formulation~\cite{carmona2018probabilistic1} in two respects: 
(1) the representative player optimizes over measure-valued controls, analogously to  the MFC problem \eqref{eq:mf_lifted}; and
(2) the running cost depends jointly on the control variable and the joint state-control law, which requires   the joint distribution of the representative control and the population distribution.

Consider the canonical space $\Tilde{\Omega}'\coloneqq C([0,T];\cP(\RR^n))\times C([0,T];\cP(\RR^n)) \times \MM(\cP(\RR^n\times A)^2)\times C([0,T];\RR^l)$. Let $(\Tilde{\mu},\Tilde{\mu}',\Tilde{\Pi},\Tilde{B})$  be the canonical element of $\Tilde{\Omega}'$, and define
$\Tilde{\Lambda}\coloneqq \Tilde{\Pi}(dm,\cP(\RR^n\times A), dt)$,   $\Tilde{\Lambda}'\coloneqq \Tilde{\Pi}(\cP(\RR^n\times A),dm',dt)$, and
the canonical filtration $\Tilde{\mathbb{F}}'=(\Tilde{\cF}_t')_{t\in [0,T]}$  such that $
\Tilde{\cF}_t'\coloneqq \sigma\big\{ \Tilde{\mu}_{t\wedge\cdot},\Tilde{\mu}_{t\wedge\cdot}', \Tilde{\Pi}^t, \Tilde{B}_{t\wedge\cdot}  \big\}
$
  for all $t\in [0,T]$.
By \cite[Lemma 3.2]{lacker2015mean}, there exists  a disintegration $\Tilde{\Pi}(dm,dm',dt)=\Tilde{\Pi}_t(dm,dm')dt$ with $(\Tilde{\Pi}_t)_{t\in [0,T]}$ being a $\mathcal{P}(\cP(\RR^n\times A)^2)$-valued $\Tilde{\mathbb{F}}'$-predictable process. Similarly, let  $(\Tilde{\Lambda}_t)_{t\in [0,T]}$ and $(\Tilde{\Lambda}_t')_{t\in [0,T]}$ be the $\Tilde{\mathbb{F}}'$-predictable disintegration of $\Tilde{\Lambda}$ and $\Tilde{\Lambda}'$, respectively.
Here $\Tilde{\mu}'$ and $\Tilde{\Lambda}'$ denote, respectively,
the representative player's conditional state law and 
the conditional law of state-action joint law given the common noise, 
$(\Tilde{\mu}, \Tilde{\Lambda})$ represent  the corresponding conditional distributions for the population, and  $\Tilde{B}$ represents the common noise.

The set of admissible measure-valued controls for the representative player is defined as follows.

 \begin{definition}\label{def:measure_control2}
Given an initial distribution $\nu\in\cP_2(\mathbb{R}^n)$, $\Tilde{\mathbb{P}}\in\cP_2(\Tilde{\Omega}')$ is an
admissible  measure-valued control for $\nu$ if 
\begin{enumerate}[label=(\roman*)]
\item\label{item:initial_state'} $\Tilde{\mathbb{P}}(\Tilde{\mu}_0'=\nu)=1$.
\item\label{item:state_martingale'} $(\Tilde{B}_t)_{t\in[0,T]}$ is a $(\Tilde{\mathbb{P}},\Tilde{\mathbb{F}}')$-Wiener process starting at zero, and for $\Tilde{\mathbb{P}}$-almost every $\omega\in\Tilde{\Omega}'$, $N_t(f,\Tilde{\mu}',\Tilde{\Lambda}')=0$ for all $f\in C_b^2(\mathbb{R}^n)$ and all $t\in [0,T]$.
\item\label{item:state_marginal'} For $d\Tilde{\mathbb{P}}\otimes dt$ almost every $(t,\omega)\in [0,T]\times\Tilde{\Omega}'$, $\Tilde{\Lambda}_t'(\mathbb{Z}_{\Tilde{\mu}_t'})=1$.
\end{enumerate}
We denote by $\cP_V'(\nu)$ the set of all admissible measure-valued controls for  $\nu$.
\end{definition}

 By exploiting     the   fact that  the state dynamic is independent of the population distribution, 
 Definition \ref{def:measure_control2}
 states that,   conditional on the  population  law  $(\Tilde{\mu}, \Tilde{\Lambda})$ and the common noise $\Tilde{B}$, the   law $(\Tilde{\mu}', \Tilde{\Lambda}')$ of the representative player satisfies the martingale condition corresponding to an admissible control for the MFC problem (cf.~Definition \ref{def:measure_control}).
 Indeed, for any $\Tilde{\PP}\in\cP_V'(\nu)$,   $\Tilde{\PP}\circ(\Tilde{\mu}',\Tilde{\Lambda}',\Tilde{B})^{-1}\in\cP_V(\nu)$, 
 since  $\tilde B$ is also an $\Tilde{\mathbb F}$-Brownian motion.
 
We define the representative player's objective function $\Tilde{J}:\cP_2(\Tilde{\Omega}')\rightarrow\mathbb{R}$
as follows: for all $\Tilde{\mathbb{P}}\in\cP_2(\Tilde{\Omega}')$,
\begin{equation}\label{eq:mfg_measure}
\Tilde{J}(\Tilde{\mathbb{P}})\coloneqq \mathbb{E}^{\Tilde{\mathbb{P}}}\left[\int_0^T \int_{\cP(\RR^n\times A)^2} \left\langle f(t,\cdot,\cdot,m),m' \right\rangle \Tilde{\Pi}_t(dm,dm')dt + \left\langle g(\cdot,\Tilde{\mu}_T),\Tilde{\mu}_T' \right\rangle \right],
\end{equation}
with the functions $f$ and $g$ given in 
Assumption \ref{assumption:mean_field}.
Since $f$ depends jointly on $m$ and $m'$ in a non-separable manner, \eqref{eq:mfg_measure} depends on $\Tilde{\Pi}$ rather than only on its  marginals, in contrast to the separable cases considered in \cite{carmona2015probabilistic, cardaliaguet2018mean, lauriere2022convergence, djete2023mean}.

The precise definition of an  $\varepsilon$-mean field equilibrium  of the above MFG is given below.

\begin{definition}\label{def:MFG_solution}
Given $\nu\in\cP_2(\mathbb{R}^n)$, for any $\varepsilon\geq 0$, $\hat{\mathbb{P}}^*\in\cP_V'(\nu)$ is an $\varepsilon$-mean field equilibrium ($\varepsilon$-MFE)   for the MFG \eqref{eq:mfg_measure} if
\begin{enumerate}[label=(\roman*)]
\item
\label{item:optimality}
($\varepsilon$-Suboptimality) For every $\Tilde{\mathbb{P}}\in\cP_V'(\nu)$ such that $\hat{\mathbb{P}}^*\circ (\Tilde{\mu},\Tilde{\Lambda},\tilde B)^{-1}=\Tilde{\mathbb{P}}\circ (\Tilde{\mu},\Tilde{\Lambda},\tilde B)^{-1}$, $\Tilde{J}(\hat{\mathbb{P}}^*)\leq \Tilde{J}(\Tilde{\mathbb{P}}) + \varepsilon $.
\item
\label{item:consistency}
 (Consistency) For $\hat{\mathbb{P}}^*$-almost every $\omega\in \Tilde{\Omega}'$, $\Tilde{\mu}=\Tilde{\mu}'$, and for $d\hat{\mathbb{P}}^*\otimes dt$ almost every $(t,\omega)\in [0,T]\times \Tilde{\Omega}'$, $\Tilde{\Pi}_t(\{m=m'\})=1$.
\end{enumerate}
When $\varepsilon = 0$, $\hat{\mathbb{P}}^*\in\cP_V'(\nu)$ is a mean field equilibrium (MFE).
\end{definition}

Condition \ref{item:optimality}  
requires 
$\hat{\mathbb{P}}^*$ to minimize
\eqref{eq:mfg_measure}   over all measure-valued controls up to an error of $\varepsilon$ with the fixed population and common-noise distribution, and 
Condition \ref{item:consistency} requires that    the conditional state-control law of the representative player agrees with that of the population almost surely.  This latter requirement imposes a stronger consistency condition than that in \cite[Definition 2.7]{djete2023mean}, which only requires the  conditional laws to coincide in distribution, i.e., $\Tilde{\Lambda}' = \Tilde{\Lambda}$.

 \subsection{Objective functions of MFC problems as $\alpha_{\infty}$-potential functions for MFGs}

We proceed to identify the objective    \eqref{eq:mf_potential_measure} of the MFC problem \eqref{eq:mf_lifted} as an $\alpha_{\infty}$-potential function for the MFG \eqref{eq:mfg_measure}.
Note that the MFC problem  \eqref{eq:mf_lifted} is formulated on $\Tilde{\Omega}$, whereas the MFG is defined on a different space $\Tilde{\Omega}'$ involving duplicated variables.
To connect the MFC and MFG,
the following lemma introduces a canonical embedding that maps any control of the MFC problem \eqref{eq:mf_lifted} to a control of the MFG \eqref{eq:mfg_measure}. 
It further constructs an interpolation between the marginal distributions of the population and the representative player in an admissible control for the MFG, which will be used to quantify the variation of the MFC objective $\Tilde{\Phi}$ when the representative player perturbs its control.

\begin{lemma}
\label{lemma:mfg_mfc}
\begin{enumerate}[label=(\roman*)]
\item\label{item:embedding}
Define the map $\iota:\cP(\Tilde{\Omega})\to \cP(\Tilde{\Omega}')$ such that  for each $\Tilde{\PP}\in \cP(\Tilde{\Omega})$, 
$
\iota(\Tilde{\PP})\coloneqq \Tilde{\PP}\circ(\Tilde{\mu},\Tilde{\mu},\Tilde{\Lambda}_t(dm)\delta_m(dm')dt, \Tilde{B})^{-1}.
$
For all $\Tilde{\PP}\in \cP_V(\nu)$,  it holds that $\iota(\Tilde{\PP})\in \cP_V'(\nu)$,
$\iota(\Tilde{\PP})\circ(\Tilde{\mu},\Tilde{\Lambda},\Tilde{B})^{-1} = \Tilde{\PP} $,
and $\iota(\Tilde{\PP})$ satisfies Condition \ref{item:consistency} in Definition \ref{def:MFG_solution}.
\item\label{item:perturbation}
For each $\varepsilon\in [0,1]$, define the map $\kappa_{\varepsilon}:\cP(\Tilde{\Omega}')\to\cP(\Tilde{\Omega})$ such that 
\begin{equation*}
\kappa_{\varepsilon}(\Tilde{\PP}')\coloneqq \Tilde{\mathbb{P}}'\circ\left(\Tilde{\mu}^{\varepsilon}, \Tilde{\Lambda}^{\varepsilon},\Tilde{B} \right)^{-1},
\quad  \forall \Tilde{\PP}'\in\cP(\Tilde{\Omega}'),
\end{equation*}
where $\Tilde{\mu}^{\varepsilon}\coloneqq \varepsilon \Tilde{\mu}'+(1-\varepsilon)\Tilde{\mu}$, and $\Tilde{\Lambda}^{\varepsilon}\coloneqq (\Tilde{\Pi}_t\circ\tau_{\varepsilon}^{-1})dt$, with the function $\tau_{\varepsilon}:\cP(\RR^n\times A)\times \cP(\RR^n\times A)\rightarrow\cP(\RR^n\times A)$ defined as $\tau_{\varepsilon}(m,m')=\varepsilon m'+(1-\varepsilon)m$ for all $m,m'\in \cP(\RR^n\times A)$.

For all $\Tilde{\PP}\in\cP_V'(\nu)$ with $\Tilde{\PP}\circ (\Tilde{\mu},\Tilde{\Lambda},\Tilde{B}  )^{-1}\in\cP_V(\nu)$, it holds that 
 $\kappa_{\varepsilon}(\Tilde{\PP})\in \cP_V(\nu)$ for all $\varepsilon\in[0,1]$,
 $\kappa_{0}(\Tilde{\PP}) = \Tilde{\PP}\circ(\Tilde{\mu},\Tilde{\Lambda},\Tilde{B})^{-1}$, and $\kappa_{1}(\Tilde{\PP}) = \Tilde{\PP}\circ(\Tilde{\mu}',\Tilde{\Lambda}',\Tilde{B})^{-1}$.
\end{enumerate}
\end{lemma}

The map $\iota$ duplicates the representative player’s state and control distributions to form the population distribution.
The following remark shows that the interpolation $(\kappa_\varepsilon)_{\varepsilon \in [0,1]}$ is consistent with the interpolation in \cite{hofer2026optimal}  when restricted to strict controls.

\begin{remark}
   [Interpolation $\kappa_\varepsilon$ for strict controls]
\label{rmk:consistency_strict_control}
The random interpolation constructed in 
\cite{hofer2026optimal} for two strict controls 
is a special case 
of the interpolation map $\kappa_\varepsilon$ in Lemma \ref{lemma:mfg_mfc}.

Specifically, 
as in \cite{hofer2026optimal},
consider an MFG on a  filtered probability space $(\Omega,\mathcal{F},\mathbb{F},\mathbb{P})$, with   $  \mathcal{G}\subset \mathcal{F}$ being a $\sigma$-algebra  generated by the common noise $B$. Let $\mathcal{A}$ be a set of $\mathbb{F}$-predictable $A$-valued processes representing the representative player’s admissible controls, and 
for each $u \in \mathcal{A}$, let $X^u$ be the representative player’s state process controlled by $u$.
For any  $u,u'\in \cA$,   define  
their interpolation $(u^\varepsilon)_{\varepsilon\in [0,1]}$,
where for each $\varepsilon$, $u^\varepsilon \coloneqq u' 1_{\eta=0}
+u 1_{\eta=1}$ is defined
  on an extended space $(\hat \Omega, \hat \cF, \hat{\PP}^{\varepsilon})\coloneqq ( \Omega\times\{0,1\}, \cF\otimes\cF^{\diamond},\hat{\PP}^{\varepsilon})$, 
  $\cF^{\diamond}$ is the power set of $\{0,1\}$, and  $\eta $
is a Bernoulli random variable, independent of $\mathcal F$, such that $\hat \PP^{\varepsilon}(\eta=0)=\varepsilon$. Let  $ \hat{\cG} \coloneqq \cG \otimes \{\emptyset,  \{0,1\}\}$ be the extension of  $\cG$ on $\hat \Omega$.
Then $\cL^{\hat{\PP}^{\varepsilon}}(X_t^{u^{\varepsilon}},u_t^{\varepsilon}\vert\hat{\cG}) = \varepsilon \cL^{{\PP}}(X_t^{u'},u_t'\vert\cG) + (1-\varepsilon)\cL^{{\PP}}(X_t^u,u_t\vert\cG)$ (see  \cite[Lemma 2]{hofer2026optimal}).

To interpret the  distribution of  $\cL^{\hat{\PP}^{\varepsilon}}(X_t^{u^{\varepsilon}},u_t^{\varepsilon}\vert\hat{\cG})$ as an interpolation map $\kappa_\varepsilon$ in Lemma \ref{lemma:mfg_mfc},   set 
\begin{equation*}
\begin{aligned}
& \Tilde{\PP}\coloneqq \PP\circ\left(  (\cL^{\PP}(X_t^u\vert\cG))_{t\in[0,T]}, \delta_{\cL^{\PP}(X_t^u,u_t\vert\cG)}(dm)dt, B\right)^{-1} ,\\
& \Tilde{\PP}'\coloneqq \PP\circ\left( (\cL^{\PP}(X_t^{u}\vert\cG))_{t\in[0,T]}, (\cL^{\PP}(X_t^{u'}\vert\cG))_{t\in[0,T]}, \delta_{\cL^{\PP}(X_t^{u'},u_t'\vert\cG)}(dm')\delta_{\cL^{\PP}(X_t^u,u_t\vert\cG)}(dm)dt, B\right)^{-1}.
\end{aligned}
\end{equation*}
 Then it holds
 for all 
$\varepsilon\in(0,1)$,
\begin{equation*}
\begin{aligned}
\kappa_{\varepsilon}(\Tilde{\PP}') =& \PP\circ\left( (\varepsilon\cL(X_t^{u'}\vert\cG) + (1-\varepsilon)\cL(X_t^u\vert\cG) )_{t\in[0,T]}, \delta_{\varepsilon \cL(X_t^{u'},u_t'\vert\cG) + (1-\varepsilon)\cL(X_t^u,u_t\vert\cG)}(dm)dt, B\right)^{-1} \\
=& \hat{\PP}^{\varepsilon}\circ\left( (\cL^{\hat{\PP}^{\varepsilon}}(X_t^{u^{\varepsilon}}\vert\hat{\cG}))_{t\in[0,T]}, \delta_{ \cL^{\hat{\PP}^{\varepsilon}}(X_t^{u^{\varepsilon}},u_t^{\varepsilon}\vert\hat{\cG}) }(dm)dt, B \right)^{-1}.
\end{aligned}
\end{equation*}
 
\end{remark}

Using the maps $\iota$ and $\kappa$, the following theorem   shows that  the objective of the MFC problem \eqref{eq:mf_potential_measure} is an $\alpha_{\infty}$-potential function for the MFG \eqref{eq:mfg_measure}, and consequently  any minimizer of the MFC problem \eqref{eq:mf_potential_measure} is an $\alpha_{\infty}$-MFE of the MFG \eqref{eq:mfg_measure}. 

\begin{theorem}
\label{thm:alpha_potential_MFG}
Suppose Assumption \ref{assumption:mean_field} holds.  
The functional 
$\Tilde{\Phi}:\mathcal{P}_2(\Tilde{\Omega})\rightarrow\mathbb{R}$
given by \eqref{eq:mf_potential_measure}
is an $\alpha_{\infty}$-potential function for the MFG  \eqref{eq:mfg_measure}, 
namely,
for all $\Tilde\PP\in\cP_V(\nu)$ and $\Tilde\PP'\in\cA(\Tilde\PP)\coloneqq \left\{\Tilde\PP'\in\cP_V'(\nu) \mid \Tilde\PP'\circ(\Tilde\mu,\Tilde\Lambda,\Tilde B)^{-1}=\Tilde \PP\right\}$,
\begin{equation}
\label{eq:alpha_potential_MFG}
\left\vert \Tilde{J}(\Tilde{\PP}')-\Tilde{J}(\iota(\Tilde{\PP})) -\lim_{\varepsilon\downarrow 0} \frac{\Tilde{\Phi}(\kappa_{\varepsilon}(\Tilde{\PP}')) - \Tilde{\Phi}(\Tilde{\PP})}{\varepsilon} \right\vert \leq \alpha_{\infty},
\end{equation}
with
\begin{equation*}
\begin{aligned}
\alpha_{\infty} \leq& \frac{1}{2}\sup_{\Tilde{\PP}\in\cP_V(\nu),\Tilde{\PP}'\in \cA(\Tilde{\PP})} \mathbb{E}^{\Tilde{\mathbb{P}}'}\bigg[ \Vert\partial_{\nu}\partial_{(x,a)} \Delta f\Vert_{\infty} \int_0^T \int_{\cP_2(\RR^n\times A)^2} \cW_2(m,\delta_0)\cW_2(m,m')\Tilde{\Pi}_t(dm,dm') dt \\
& + \Vert \partial_{\mu}\partial_x \Delta  g\Vert_{\infty} \cW_2(\Tilde{\mu}_T,\delta_0) \cW_2(\Tilde{\mu}_T,\Tilde{\mu}_T') \bigg],
\end{aligned}
\end{equation*}
where the constants 
$\Vert \partial_{\nu}\partial_{(x,a)}\Delta f\Vert_{\infty}$ and 
$\Vert \partial_{\mu}\partial_x \Delta g\Vert_{\infty} $
are given by
\begin{equation*}
\begin{aligned}
\Vert \partial_{\nu}\partial_{(x,a)}\Delta  f\Vert_{\infty} \coloneqq & \esssup_{t\in[0,T]}\sup_{(x,a),(x',a')\in \mathbb{R}^n\times \mathbb{R}^k, \nu\in \cP_2(\mathbb{R}^n\times \mathbb{R}^k)}\vert \partial_{\nu}\partial_{(x,a)}f(t,x,a,\nu)(x',a')\\
&-[\partial_{\nu}\partial_{(x,a)}f(t,x',a',\nu)(x,a)]^{\top} \vert,\\
\Vert \partial_{\mu}\partial_x \Delta g\Vert_{\infty} \coloneqq & \sup_{x,x'\in\mathbb{R}^n, \mu\in \cP_2(\mathbb{R}^n)} \vert \partial_{\mu}\partial_x g(x,\mu)(x')-[\partial_{\mu}\partial_x g(x',\mu)(x)]^{\top} \vert.
\end{aligned}
\end{equation*}
Moreover, if  $\PP^*\in \arg\min_{P\in \cP_V(\nu)}\Tilde \Phi(P)$, then $\iota(\PP^*)$ is an $\alpha_{\infty}$-MFE as in Definition \ref{def:MFG_solution}.

\end{theorem}

Theorem \ref{thm:alpha_potential_MFG} is 
the mean field analogue of Theorem~\ref{thm:alpha2}. 
Property 
\eqref{eq:alpha_potential_MFG}
states that when the representative player perturbs their control, the change in their objective equals the \emph{derivative}
 of the mean-field $\alpha_{\infty}$-potential function along the interpolation path $\kappa$, up to an error $\alpha_{\infty}$.
The  parameter $\alpha_{\infty}$ 
can be bounded above by the asymmetry of the mixed derivatives of the cost functions with respect to the representative player's state-control variables and the population measure, analogously to the upper bound of $\alpha_N$
in Theorem~\ref{thm:alpha2}.

 We highlight that, 
although the $N$-player $\alpha_N$-potential function is scaled by $1/N$ to ensure its convergence to the MFC objective \eqref{eq:mf_potential_measure},  
both the property  \eqref{eq:alpha_potential_MFG} of the mean-field $\alpha_{\infty}$-potential function
and the upper bound of  $\alpha_{\infty}$,  
arise naturally as limits of their counterparts in the $N$-player setting. 
We explain this consistency for strict controls below.

\begin{remark}[Consistency with    $N$-Player $\alpha_N$-potential game]
\label{rmk:pmfg_definition}
Let $\Phi$ be the $\alpha_N$-potential function for $N$-player games constructed in Theorem \ref{thm:alpha2}.
Setting $\Phi^N=\Phi/N$ as in \eqref{eq:N_player_alpha_PG} yields 
\begin{equation*}
\sup_{i\in[N]}\sup_{u_i'\in\cA_i^N,\bu\in\cA^N} \left\vert \frac{1}{N}\big(J_i(u_i', \bu_{-i})-J_i(\bu)\big) - \left(\Phi^N(u_i', \bu_{-i})-\Phi^N(\bu) \right)   \right\vert \leq \frac{\alpha_N}{N},
\end{equation*}
which implies that for all $\bu\in\cA^N$, and $u_i'\in\cA_i^N$ for all $i\in [N]$,
\begin{equation}\label{eq:N_potential}
\left\vert \frac{1}{N}\sum_{i\in [N]}\big(J_i(u_i', \bu_{-i})-J_i(\bu)\big) - \sum_{i\in [N]} \left(\Phi^N(u_i', \bu_{-i})-\Phi^N(\bu) \right) \right\vert \leq \alpha_N.
\end{equation}

Now let $\cA$ denote the set of admissible strict controls, and let $J:\cA\to\RR$ and $\Phi^{\infty}:\cA\to\RR$ denote the objective functionals of the limiting MFG and MFC problems, respectively, formulated in terms of strict controls. Assume that appropriate propagation of chaos results  hold  for $(J_i)_{i \in [N]}$ toward $J$ and for $\Phi^N$ toward $\Phi^\infty$, 
we aim to show that \eqref{eq:N_potential} naturally yields, 
for all $u',u\in \cA$,
\begin{equation}
\label{eq:mf_potential_discussion}
\left\vert J(u';\nu^u) - J(u;\nu^u) - \lim_{\varepsilon\downarrow 0}\frac{\Phi^{\infty}(u^{\varepsilon}) - \Phi^{\infty}(u)}{\varepsilon} \right\vert \leq \alpha_{\infty},
\end{equation}
for some $\alpha_{\infty}\geq 0$, where $\nu^u=(\cL(X_t^u, u_t\vert B))_{t\in [0,T]}$
and $u^{\varepsilon} = u 1_{\eta=1} + u'1_{\eta=0}$, with $\eta$  being an independent  $\{0,1\}$-valued random variable  satisfying  $\mathbb{P}(\eta=0)=\varepsilon$. Note that when $\alpha_{\infty}=0$, \eqref{eq:mf_potential_discussion} recovers    \cite[Lemma 3]{hofer2026optimal}, which is a special case of \eqref{eq:alpha_potential_MFG} by Remark~\ref{rmk:consistency_strict_control}.

To prove  \eqref{eq:mf_potential_discussion}, 
 consider 
 $u,u'\in \cA$ such that $u_t=\phi(t,\xi,W_{t\wedge\cdot},B_{t\wedge\cdot})$  and $u_t'=\phi'(t,\xi,W_{t\wedge\cdot},B_{t\wedge\cdot})$ for all $t\in [0,T]$, where $\phi$ and $\phi'$ are measurable  functions. 
For each $N\in \NN^*$ and $i\in [N]$,
let $u_{i,t} = \phi(t,\xi_i,W_{t\wedge\cdot,i},B_{t\wedge\cdot})$,  $u_{i,t}'=\phi'(t,\xi_i,W_{t\wedge\cdot,i},B_{t\wedge\cdot})$,
and  $\bu =(u_i)_{i\in [N]}$.  
The propagation of chaos of $(J_i)_{i \in [N]}$ to  $J$  (see e.g.~\cite[Proposition 3.16]{djete2023mean}) implies that
\begin{align}\label{eq:limit_mfg}
\begin{split}
J(u';\nu^u)-J(u;\nu^u) = \lim_{N\rightarrow\infty} \frac{1}{N}\sum_{i\in [N]} [J_i(u_i', \bu_{-i})-J_i(\bu)],
\end{split}
\end{align}
We now rewrite the sum of the changes in $\Phi^N$   in \eqref{eq:N_potential} as a derivative of $\Phi^N$ along random perturbations. 
Let $(\eta_i)_{i\in\mathbb{N}^*}$ be a sequence of i.i.d.~$\{0,1\}$-valued random variables with $\mathbb{P}(\eta_i=0)=\varepsilon$, and define $u_i^{\varepsilon} = u_i 1_{\eta_i=1} + u_i'1_{\eta_i=0}$ for all $i\in [N]$. 
Using the form of $\Phi^N$ in \eqref{eq:scaled_N_potential} and the   independence of $(\eta_i)_{i\in\NN^*}$,
\begin{align*}
 \Phi^N(\bu^{\varepsilon}) &= \sum_{D\subset [N]}\varepsilon^{|D|}(1-\varepsilon)^{N-|D|}\Phi^N\left( (u_j')_{j\in D}, (u_i)_{i\notin D} \right) \\
&=  (1-\varepsilon)^N \Phi^N(\bu) + \sum_{i\in [N]}\varepsilon(1-\varepsilon)^{N-1}\Phi^N\left( u_i',\bu_{-i} \right) \\
&\quad+ \sum_{D\subset [N],|D|\geq 2}\varepsilon^{|D|}(1-\varepsilon)^{N-|D|}\Phi^N\left( (u_j')_{j\in D}, (u_i)_{i\notin D} \right),
\end{align*}
which along with    $\sum_{D\subset [N]}\varepsilon^{|D|}(1-\varepsilon)^{N-|D|}=1$ implies that 
\begin{align}\label{eq:phi_deriv_N1}
&\lim_{\varepsilon\downarrow 0}\frac{\Phi^N(\bu^{\varepsilon})-\Phi^N(\bu) }{\varepsilon}  =    \sum_{i\in [N]}[\Phi^N\left( u_i',\bu_{-i} \right) -\Phi^N\left( \bu\right) ].
\end{align}
Assume that the propagation of chaos of   $(\Phi^N)_{N\in \NN^*}$ to  $\Phi^\infty$  gives 
$\lim_{N\rightarrow\infty} \lim_{\varepsilon\downarrow 0}\frac{\Phi^N(\bu^{\varepsilon})-\Phi^N(\bu) }{\varepsilon}  = \lim_{\varepsilon\downarrow 0} \frac{\Phi^{\infty}(u^{\varepsilon})-\Phi^{\infty}(u)}{\varepsilon}.
$
Then combining \eqref{eq:N_potential}, \eqref{eq:limit_mfg} and \eqref{eq:phi_deriv_N1} yields \eqref{eq:mf_potential_discussion}, thereby justifying the definition    \eqref{eq:alpha_potential_MFG} of   mean field   $\alpha$-potential functions. 
\end{remark}

\section{Equivalent Characterizations of Potential MFGs}
\label{sec:consistent_mfg}

In this section, we 
analyze the limit of $N$-player $\alpha_N$-potential functions by   assuming further that  $\lim_{N\to \infty}\alpha_N=0$.
We prove that this vanishing $\alpha_N$ condition 
is equivalent to the existing conditions in the literature for potential MFGs.
In this case,   the limit of normalized $N$-player $\alpha_N$-potential function
is a potential function for the MFG \eqref{eq:mfg_measure}.
Moreover,  approximate NEs of   $N$-player games  selected by  the $\alpha_N$-potential functions converge to an MFE.  Finally, we provide an example  where both the finite-player games and the limiting MFG admit exact potentials, as well as an example where the finite-player games admit no exact potential but the limiting MFG does.

\subsection{Equivalence between  vanishing $\alpha_N$ and potential MFGs}

 By Theorem \ref{thm:alpha2},
the parameter $\alpha_N$   vanishes  as $N \to \infty$ if the Hessian  of    cost functions becomes asymptotically pairwise symmetric. The following theorem provides several equivalent characterizations of this condition.

\begin{theorem}\label{thm:alpha0_equiv}
Suppose Assumption \ref{assumption:mean_field}\ref{item:f_g_i} holds. The following statements are equivalent:
\begin{enumerate}[label=(\roman*)]
\item 
\label{item:alpha_N_limit}$\lim_{N\rightarrow\infty}\max_{i\in [N]}\sum_{j\in [N]}\Vert \partial^2_{(x_i,a_i)(x_j,a_j)}(f_i-f_j)\Vert_{\infty}=0, $ and\\ $\lim_{N\rightarrow\infty}\max_{i\in [N]}\sum_{j\in [N]}\Vert \partial^2_{x_i x_j}(g_i-g_j)\Vert_{\infty}=0.
$ 
\item
\label{item:curl-free}
For a.e.~$t\in [0,T]$, and all 
$x,x'\in \RR^n$, 
$a,a'\in \RR^k$,
$\nu\in \mathcal{P}_2(\mathbb{R}^n\times\mathbb{R}^k)$, and 
$\mu\in \mathcal{P}_2(\mathbb{R}^n)$, 
\begin{align}\label{eq:curl_free}
\begin{split} \partial_{\nu}\partial_{(x,a)} f(t,x,a,\nu)(x',a')&=\left[\partial_{\nu}\partial_{(x,a)} f(t,x',a',\nu)(x,a) \right]^{\top},\\
  \partial_{\mu}\partial_x g(x,\mu)(x')&=\left[\partial_{\mu}\partial_x g(x',\mu)(x) \right]^{\top}.
\end{split}
\end{align}
\item
\label{item:cost_potential}
There exist measurable functions $F:[0,T]\times\mathcal{P}_2(\mathbb{R}^n\times\mathbb{R}^k)\rightarrow\mathbb{R}$ and $G:\mathcal{P}_2(\mathbb{R}^n)\rightarrow\mathbb{R}$ such that $\frac{\delta F}{\delta \nu}(t,\cdot)$ and $\frac{\delta G}{\delta \mu}$ exist for all $t\in[0,T]$, $(t,x,a,\nu)\mapsto \frac{\delta F}{\delta \nu}(t,\nu)(x,a)\in\cC^{0,2,2}([0,T]\times\RR^{n+k}\times\cP_2(\RR^{n+k}))$, $(x,\mu)\mapsto \frac{\delta G}{\delta \mu}(\mu)(x)\in \cC^{2,2}(\RR^n\times \cP_2(\RR^n))$, and
  for  a.e.~$t\in [0,T]$, 
  and all $(x,a,\nu,\mu)\in  \mathbb{R}^n\times\mathbb{R}^k\times \mathcal{P}_2(\mathbb{R}^n\times\mathbb{R}^k) \times \mathcal{P}_2(\mathbb{R}^n)$,
\begin{equation}\label{eq:carmona_condition}
\begin{aligned}
\partial_{(x,a)}f(t,x,a,\nu)=\partial_{\nu}F(t,\nu)(x,a),\quad \partial_x g(x,\mu) =  \partial_{\mu} G(\mu)(x).
\end{aligned}
\end{equation}
Moreover, $F$ and $G$ can be constructed as: for $(t,\nu,\mu)\in [0,T]\times\cP_2(\RR^{n+k})\times\cP_2(\RR^n)$,
\begin{equation}\label{eq:mf_potential_integrated}
F(t,\nu)\coloneqq \int_{\RR^{n+k}} F^{\infty}(t,x,a,\nu)\nu(dx,da),\quad G(\mu)\coloneqq \int_{\RR^n} G^{\infty}(x,\mu)\mu(dx),
\end{equation}
where $F^{\infty}$ and $G^{\infty}$ are defined in \eqref{eq:mf_potential_cost}. 
\end{enumerate}
\end{theorem}

\begin{remark}

Theorem \ref{thm:alpha0_equiv} shows that the condition $\lim_{N\to \infty}\alpha_N = 0$ is  equivalent to the standard conditions imposed for potential MFGs,
thereby demonstrating that these conditions arise naturally as limits of $N$-player $\alpha_N$-potential games with vanishing $\alpha_N$. 

Indeed, in the mean field literature, Condition \eqref{eq:carmona_condition} is often imposed as a sufficient condition for the existence of a potential structure (see, e.g., \cite[Equation 6.132]{carmona2018probabilistic1}). 
 A sufficient condition for \eqref{eq:carmona_condition}
is the existence of sufficiently regular functions $F$ and $G$ such that
(see  \cite[Equation (3)]{hofer2026optimal}):
\begin{align}
\label{eq:f_linear}
\begin{split}
f(t,x,a,\nu)=& \int_{\RR^n\times \RR^k}\frac{\delta F}{\delta \nu}(t,x',a',\nu)(x,a)\nu(dx',da') + F(t,x,a,\nu),\quad \\
g(x,\mu)=& \int_{\RR^n}\frac{\delta G}{\delta \mu}(x',\mu)(x)\mu(dx') + G(x,\mu).
\end{split}
\end{align}
In this case, by setting  
 $\hat F(t, \nu)\coloneqq \int_{\RR^n\times \RR^k} F(t,x,a,\nu)\nu(dx,da)$ and 
$\hat G(\mu)\coloneqq \int_{\RR^n} G(x,\mu)\mu(dx)$ for $\mu\in\cP_2(\RR^n)$, it holds that 
$ 
f(t,x,a,\nu)=\frac{\delta \hat F}{\delta \nu}(t,\nu)(x,a) + \hat F(t,\nu),$ 
and $g(x,\mu)=\frac{\delta \hat G}{\delta \mu}(\mu)(x) + \hat G(\mu)$, which implies Condition \eqref{eq:carmona_condition}.
Theorem~\ref{thm:alpha0_equiv} rigorously establishes the equivalence between   Conditions \eqref{eq:curl_free} and \eqref{eq:carmona_condition}.
In the special case where the mean-field interaction appears only in the state variable, \cite[Proposition 1.2]{cardaliaguet2017learning} establishes the analogous equivalence between Conditions \eqref{eq:curl_free} and \eqref{eq:carmona_condition} using the notion of linear functional derivatives.

\end{remark}

Combining  Theorems \ref{thm:mean_field_limit2}, \ref{thm:alpha_potential_MFG} and  \ref{thm:alpha0_equiv},
we show  that under the condition $\lim_{N\to\infty}\alpha_N=0$, the objective \eqref{eq:mf_potential_measure} of the MFC problem \eqref{eq:mf_lifted} is a potential function of the MFG \eqref{eq:mfg_measure}. Moreover,   any sequence of approximate NEs of $N$-player games that minimize the $\alpha_N$-potential functions \eqref{eq:scaled_N_potential} converges, up to subsequences, to an MFE of the  MFG \eqref{eq:mfg_measure}.

\begin{corollary}\label{cor:mfg_poc}
Suppose the assumptions of Theorem \ref{thm:mean_field_limit2} and any of the equivalent conditions in Theorem \ref{thm:alpha0_equiv} hold.  For every initial law $\nu\in\cP_2(\RR^n)$, the functional 
$\Tilde{\Phi}:\mathcal{P}_2(\Tilde{\Omega})\rightarrow\mathbb{R}$
given by \eqref{eq:mf_potential_measure}
is a potential function for the MFG  \eqref{eq:mfg_measure}, 
in the sense that 
for all $\Tilde\PP\in\cP_V(\nu)$ and $\Tilde\PP'\in\cA(\Tilde\PP)$,
\begin{equation}
\label{eq:mf_potential}
 \Tilde{J}(\Tilde{\PP}')-\Tilde{J}(\iota(\Tilde{\PP})) =\lim_{\varepsilon\downarrow 0} \frac{\Tilde{\Phi}(\kappa_{\varepsilon}(\Tilde{\PP}')) - \Tilde{\Phi}(\Tilde{\PP})}{\varepsilon}.
\end{equation}
Moreover, let $(\varepsilon_N)_{N\in\mathbb{N}^*}\subset  (0,\infty)$ satisfy $\lim_{N\rightarrow\infty}\varepsilon_N=0$, 
and for each $N\in \NN^*$, let $\bu^N \in \cA^N$ be an $(\alpha_N+\varepsilon_N)$-NE
of the $N$-player game \eqref{eq:N_player_cost} satisfying 
$$
\Phi^N(\bu^N;\nu_1,\cdots,\nu_N)\leq V^N(\nu_1,\cdots,\nu_N)+\frac{\varepsilon_N}{N},
$$
and define  $\Tilde{\mathbb{P}}^N\coloneqq \Tilde{\mathbb{P}}^N(\bu^N)$ as in \eqref{eq:P_u_N}. 
Then $(\Tilde{\mathbb{P}}^N)_{N\in\mathbb{N}^*}\subset \mathcal P_2(\Tilde{\Omega})$ is relatively compact, and 
 for any convergent subsequence $(\Tilde{\mathbb{P}}^{N_m})_{m\in \mathbb{N}^*}$,
 there exists $ \nu\in \mathcal{P}_p(\mathbb{R}^n)$ and $\Tilde{\mathbb{P}}^{\infty} \in \arg\min_{\Tilde{\mathbb{P}}\in\cP_V(\nu)}\Tilde{\Phi}(\Tilde{\mathbb{P}})$, such that $\lim_{m\to\infty} \cW_2(\frac{1}{N_m}\sum_{i=1}^{N_m}\nu_i,\nu)=0$,
$\lim_{m\rightarrow\infty} \mathcal{W}_2\left(\Tilde{\mathbb{P}}^{N_m},\Tilde{\mathbb{P}}^{\infty}\right)=0,
 $
 and  $\iota(\Tilde{\mathbb{P}}^{\infty})$ is an MFE    with initial condition $\nu$.
\end{corollary}

Condition \eqref{eq:mf_potential} serves as a definition of potential MFGs at the level of objective functionals that arises naturally as the limit of $N$-player $\alpha_N$-potential games when $\lim_{N\to\infty}\alpha_N=0$. It is consistent with the classical notion of potential MFGs \cite{carmona2018probabilistic1}, where  potential functions are defined as   antiderivatives of the running and terminal cost functions (cf.~\eqref{eq:carmona_condition} and \eqref{eq:f_linear}). It extends \cite[Lemma 3]{hofer2026optimal} 
from the setting of strict controls to that of measure-valued controls.

By the notion of  $\alpha$-potential functions, 
Corollary \ref{cor:mfg_poc}
 establishes propagation of chaos of potential-selected equilibria for $N$-player games in a general setting involving controlled diffusions, common noise, and non-separable state–control interactions,
   beyond the scope of existing frameworks   
\cite{carmona2015probabilistic,lauriere2022convergence,djete2023mean,possamai2025non,cecchin2026convergence,djete2026approximate}. 
  This shows that 
  NEs selected through minimization of the $\alpha$-potential functions form a more stable class of equilibria, 
generalizing similar observations for specific classes of potential MFGs in 
\cite{delarue2020selection,cecchin2022selection}.

\subsection{Examples}
\label{sec:example_mfg}

Theorem \ref{thm:alpha0_equiv} and Corollary \ref{cor:mfg_poc}
allow for identifying  explicit structural conditions 
on the cost functions $f$ and $g$
such that the MFG \eqref{eq:mfg_measure} is potential. The following examples illustrate these conditions, covering both finite-player games that admit exact potentials and those that do not, while their limiting MFGs are potential.

\begin{example}[Symmetric  interaction kernels]
Suppose that  $f$ and $g$ are of the form: 
\begin{align*}
f(t,x,a,\nu) =& f_0(t,x,a) + \int_{\RR^n\times\RR^k} h(t,x,a,x',a')\nu(dx',da'),\\
g(x,\mu) =& g_0(x) + \int_{\RR^n} k(x,x')\mu(dx'),
\end{align*}
where for all $t\in [0,T]$, $h(t,\cdot)$ and $k$ are twice continuously differentiable, and for all $t\in[0,T]$, $x,x'\in\RR^n$, and $a,a'\in\RR^k$, 
$$
h(t,x,a,x',a')=h(t,x',a',x,a), \quad  k(x,x')=k(x',x).$$ 
Then $f$ and $g$ satisfy Condition \eqref{eq:curl_free},
and hence the MFG \eqref{eq:mfg_measure} is potential. 

The corresponding finite-player games are also exact potential games. Indeed, for each $N\in\NN^*$, $i\in [N]$, $t\in [0,T]$, $\bx=(x_i)_{i\in[N]}\in \RR^{Nn}$, and $\ba=(a_i)_{i\in[N]}\in\RR^{Nk}$,
\begin{align*}
f_i(t,\bx,\ba) =& f_0(t,x_i,a_i) + \frac{1}{N}\sum_{j\in [N]}h(t,x_i,a_i,x_j,a_j),\\
g_i(\bx)=& g_0(x_i) + \frac{1}{N}\sum_{j\in [N]}k(x_i,x_j), 
\end{align*}
which implies that $\partial^2_{(x_i,a_i)(x_j,a_j)}(f_i-f_j)=0$ and $\partial^2_{x_ix_j}(g_i-g_j) = 0$ for all $i\neq j$ by the symmetry of $h$ and $k$. Hence the finite-player construction in Theorem~2.1 yields an exact
potential for every $N\in\NN^*$.

This generalizes existing examples based on control-independent convolution kernels, where  
$h(t,x,a,x',a')=\hat h(t,x-x')$
and $k(x,x')=\hat k(x-x')$, as in  the 
Kuramoto synchronization model 
\cite{carmona2023synchronization},
the   Cucker-Smale flocking model 
\cite{santambrogio2021cucker},
and the pedestrian crowd model   \cite{aurell2018mean}.
\end{example}

\begin{example}[Nonlinear aggregated  interaction]
Suppose that $f$ and $g$ have the following form:
\begin{align*}
f(t,x,a,\nu) =& \Phi\left(\int_{\RR^n\times\RR^k}\phi(t,x',a')\nu(dx',da') \right)\phi(t,x,a),\\
g(x,\mu) =& \Psi\left(\int_{\RR^n} \psi(x')\mu(dx') \right)\psi(x),
\end{align*}
where  $\Phi,\phi(t,\cdot),\Psi$ and $\psi$ are twice continuously differentiable (scalar-valued) functions. 
Then $f$ and $g$ satisfy Condition \eqref{eq:curl_free},
and hence the MFG \eqref{eq:mfg_measure} is potential.
This includes the     MFG of controls for energy producers with price interactions as in \cite[Section 6.1]{hofer2026optimal}.

However, the corresponding finite-player games need not be exact potential games. For each $N\in\NN^*$, $i\in [N]$, $t\in [0,T]$, $\bx=(x_i)_{i\in[N]}\in \RR^{Nn}$, and $\ba=(a_i)_{i\in[N]}\in\RR^{Nk}$,
\begin{align*}
f_i(t,\bx,\ba) =& \Phi\left(\frac{1}{N}\sum_{j\in[N]} \phi(t,x_j,a_j)\right)\phi(t,x_i,a_i),\\
g_i(\bx) = & \Psi\left(\frac{1}{N}\sum_{j\in [N]}\psi(x_j)\right)\psi(x_i).
\end{align*}
Consequently, the second-order derivatives of $(f_i)_{i\in[N]}$ and $(g_i)_{i\in[N]}$ satisfy: for $i\neq j$, $t\in [0,T]$, $\bx=(x_i)_{i\in[N]}\in \RR^{Nn}$, and $\ba=(a_i)_{i\in[N]}\in\RR^{Nk}$,
\begin{align*}
&\partial^2_{(x_i,a_i)(x_j,a_j)}(f_i-f_j)(t,\bx,\ba) \\
&\quad= \frac{1}{N^2}\Phi''\left( \frac{1}{N}\sum_{l\in[N]}\phi(t,x_l,a_l)\right)(\phi(t,x_i,a_i)-\phi(t,x_j,a_j))\partial_{(x_i,a_i)}\phi(t,x_i,a_i)\partial_{(x_j,a_j)}\phi(t,x_j,a_j)^{\top}, \\
&\partial^2_{x_ix_j}(g_i-g_j)(\bx) = \frac{1}{N^2}\Psi''\left(\frac{1}{N}\sum_{l\in [N]} \psi(x_l) \right)(\psi(x_i)-\psi(x_j))\partial_{x_i}\psi(x_i)\partial_{x_j}\psi(x_j)^{\top}.
\end{align*}
Theorem \ref{thm:alpha2} then yields an upper bound $\alpha_N=O(1/N)$.
\end{example}

\section{Discussions and Extensions}\label{sec:extensions} 

\subsection{Necessity of lifted control spaces in  the limit}\label{sec:examples}

To capture the limits of  $\varepsilon_N$-optimal controls for  $N$-player $\alpha_N$-potential functions,  
    the control space of the limiting MFC problem is lifted through a triple hierarchy of measures (Theorem \ref{thm:mean_field_limit2}):
  one level accounts for idiosyncratic noise,
  another captures the aggregating effects of temporal oscillations of the control process,
   and the last incorporates the effect of common noise or additional randomness arising in the limit. 
   This lifting is essential  when the running cost involves nonlinear interactions in the control.
 Below we present  concrete examples illustrating the necessity of such lifted control spaces, both in the presence and absence of common noise.   The proofs   are given in   Appendix \ref{sec:eg_proof}.

\begin{example}[With common noise]\label{eg:common_noise}

Let  $T=1$, $n=d=2$, $l=k=1$, $A=[0,1]$,   the state dynamics be given by: for all $t\in [0,1]$ and $i\in [N]$,
\begin{equation*}
\begin{cases}
dX_{t,i}^1=& u_{t,i}^N dt + dW_{t,i}^1 +  dB_t,\quad X_{0,i}^1=0, \\
dX_{t,i}^2=& dW_{t,i}^2 + dB_t,\quad X_{0,i}^2=0,
\end{cases}
\end{equation*}
and the cost functions be  given by: for all $(t,x,a,\nu,\mu)\in [0,T]\times\RR^2\times A\times \cP_2(\RR^2\times A)\times\cP_2(\RR^2)$,
\begin{align*}
f(t,x,a,\nu)&=a^2-2a\left( \int_A a'\nu(\RR^2,da') \right),\\ g(x,\mu)&=2\left(\int_{\RR^2} \begin{pmatrix} 1 \\ -1 \end{pmatrix}^{\top} x' \mu(dx') - \frac{1}{2}\right) \begin{pmatrix} 1 \\ -1 \end{pmatrix}^{\top} x.
\end{align*}

For each $N\in \NN^*$ and $i\in [N]$,
define the following control 
\begin{equation*}
u_{t,i}^N=\frac{1}{2}\left(1_{t<\frac{1}{2}} + 1_{t\geq \frac{1}{2}}1_{B_{\frac{1}{2}}\geq 0} \right) + 1_{t\in\mathbb{T}_N}1_{B_{\frac{1}{2}}< 0},
\quad \textnormal{with 
 $\mathbb{T}_N\coloneqq \bigcup_{j=N}^{2N-1} \left[\frac{j}{2N},\frac{j}{2N}+\frac{1}{4N}\right)$.}
\end{equation*}
In the sequel, we omit the dependence on the fixed initial distribution.
It holds  that 
for each $N\in \NN^*$,
$\bu^N=(u^N_i)_{i\in [N]}$
satisfies 
$\Phi^N(\bu^N) \le V^N+  \frac{79}{16N}$, 
\begin{align}\label{eq:eg1_measure}
\begin{split}
\lim_{N\to\infty}\Tilde{\mathbb{P}}^N(\bu^N) 
&=
\lim_{N\to\infty}\mathbb{P}\circ\left( (\varphi_t^{N,X})_{t\in [0,T]}, \delta_{\varphi_t^N}(dm)dt, B  \right)^{-1}  
\\
&= \PP\circ\left( \left(\cL\left( \frac{t}{2}+W_t^1+B_t,W_t^2+B_t\,\bigg\vert\, B_t \right)\right)_{t\in [0,T]}, \Tilde{\Lambda},B \right)^{-1}\coloneqq \PP^{\infty}\in \mathcal P_2(\Tilde{\Omega}), 
\end{split}
\end{align}
where $\Tilde{\Lambda}(dm,dt)=\Tilde{\Lambda}_t(dm)dt$ with
\begin{align*}
\Tilde{\Lambda}_t(dm) &= \bigg\{\left(1_{t<\frac{1}{2}} + 1_{t\geq \frac{1}{2}}1_{B_{\frac{1}{2}}\geq 0} \right) \delta_{\cL((\frac{t}{2}
+W_t^1+B_t,W_t^2+B_t),\frac{1}{2}\vert B_t)} \\
&\quad + 1_{t\geq \frac{1}{2}}1_{B_{\frac{1}{2}}< 0}\left(\frac{1}{2}\delta_{\cL((\frac{t}{2}+W_t^1+B_t,W_t^2+B_t),0\vert B_t)} +\frac{1}{2}\delta_{\cL((\frac{t}{2}+W_t^1+B_t,W_t^2+B_t),1\vert B_t)}\right)  \bigg\}(dm),
\end{align*}
and 
$\lim_{N\to\infty}V^N=\lim_{N\to\infty} \Phi^N(\bu^N) = \Tilde{\Phi}(\PP^{\infty})=-\frac{1}{4}$, which confirms Theorem \ref{thm:mean_field_limit2}.

Note that $\widetilde{\Phi}$ depends on the law of the $\mathcal{P}(\mathcal{P}(\mathbb{R}^2 \times A))$-valued stochastic process $(\widetilde{\Lambda}_t)_{t\in [0,T]}$, which is consistent with the triple hierarchy of measures in  Theorem \ref{thm:mean_field_limit2}.
The two levels of measures in $\cP(\cP(\RR^2\times A))$ separately capture the idiosyncratic noise $W$ and the aggregation effects arising from the oscillations of the controls $(u^N)_{N\in \NN^*}$ in time.
The remaining randomness of $\tilde{\Lambda}$
 is due to the realization of the common noise $B$.

To see that it is essential to consider the limit in $\cP_2(\Tilde{\Omega})$,
consider instead $\overline{\Omega}\coloneqq C([0,T];\cP(\RR^2))\times \MM(\RR^2\times A)\times C([0,T];\RR)$   as the canonical space with  element $(\mu,\overline{\Lambda},B)$, and define the corresponding objective function $\overline{\Phi}:\cP_2(\overline{\Omega})\to\RR$ by
\begin{equation*}
\overline{\Phi}(\overline{\PP})\coloneqq \EE^{\overline{\PP}}\left[\int_0^T\int_{\RR^n\times A}F^{\infty}(t,x,a,\overline{\Lambda}_t)\overline{\Lambda}_t(dx,da)dt + \int_{\RR^n}G^{\infty}(x,\mu_T)\mu_T(dx) \right],
\end{equation*}
where $F^\infty$ and $G^\infty$ are defined by  \eqref{eq:mf_potential_cost}.
It holds that 
\begin{align}\label{eq:eg1_measure_compressed}
\begin{split}
&\lim_{N\to\infty}\PP\circ\left( (\varphi_t^{N,X})_{t\in [0,1]}, \varphi_s^N(dx,da)ds,B  \right)^{-1} 
\\
&= \PP\circ\left( \left(\cL\left( \frac{t}{2}+W_t^1+B_t,W_t^2+B_t\vert B_t \right)\right)_{t\in [0,T]}, \overline{\Lambda},B \right)^{-1}\coloneqq \overline{\PP}^{\infty},
\end{split}
\end{align}
where $\overline{\Lambda}(dx, da,dt)=\overline{\Lambda}_t(dx, da)dt$ with
\begin{align*}
\begin{split}
\overline{\Lambda}_t(dx, da) &=  \cL\left(\frac{t}{2}+W_t^1+B_t,W_t^2+B_t \,\bigg\vert\, B_t\right)(dx) \bigg(\left(1_{t<\frac{1}{2}} + 1_{t\geq \frac{1}{2}}1_{B_{\frac{1}{2}}\geq 0} \right)\delta_{\frac{1}{2}}(da) 
\\
&\quad + 1_{t\geq \frac{1}{2}}1_{B_{\frac{1}{2}}< 0}\left(\frac{1}{2}\delta_0 +\frac{1}{2}\delta_1\right)(da)  \bigg).
\end{split}
\end{align*}
However, $\overline{\Phi}(\overline{\PP}^{\infty})=-\frac{3}{16} \neq \lim_{N\to\infty} \Phi^N(\bu^N)$. 

Here  $\overline{\Lambda}_t$ is a 
$\mathcal{P}(\mathbb{R}^2 \times A)$-valued random variable, 
which is the pointwise barycenter   of  $\Tilde{\Lambda}_t$
satisfying for all $\phi\in C_b(\RR^2\times A)$,  
$\int \phi d\overline{\Lambda}_t =\int_{\cP(\RR^2\times A)}\left(\int \phi dm\right)\Tilde{\Lambda}_t(dm)$.
Although the associated measures also converge   in $\cP_2( \overline{\Omega})$, the map $\overline{\Phi}$ is discontinuous because the integrated running cost depends nonlinearly on the measure argument. As a result, it fails to capture the limit of the $\varepsilon_N$-optimal controls $(\bu^N)_{N\in \NN^*}$.

\end{example}

\begin{example}[Without common noise]\label{eg:no_common_noise}
Let $T=1$, $n=d=k=1$, $\gamma=0$, $A=[0,2]$,   the state dynamics be given by: for all $t\in[0,1]$ and $i\in[N]$,
\begin{equation*}
dX_{t,i}=u_{t,i}^N dt +  dW_{t,i},\quad X_{0,i}=0,
\end{equation*}
and  $f$ and $g$ be given by: for all $(t,x,a,\nu,\mu)\in [0,T]\times\RR\times A\times \cP_2(\RR\times A)\times\cP_2(\RR)$,
\begin{equation*}
f(t,x,a,\nu)=a^2-2a\left( \int_A a'\nu(\RR,da') \right),\quad g(x,\mu)=2\left(\int_{\RR}  x' \mu(dx') - \frac{1}{2}\right)  x.
\end{equation*}

For each $N\in \NN^*$ and $i\in [N]$,
define the following control 
\begin{equation*}
u_{t,i}^N=2 \left(1_{t\in [\frac{1}{2},\frac{3}{4}]}1_{W_{\frac{1}{2},1}\geq 0}  +  1_{t\in\mathbb{T}_N}1_{W_{\frac{1}{2},1}< 0}\right),
\quad \textnormal{with $\mathbb{T}_N\coloneqq \bigcup_{j=N}^{2N-1} \left[\frac{j}{2N},\frac{j}{2N}+\frac{1}{4N}\right)$. }
\end{equation*}
It holds  that 
for each $N\in \NN^*$,
$\bu^N=(u^N_i)_{i\in [N]}$
satisfies 
$\Phi^N(\bu^N) \le V^N+  \frac{21}{4N}$, and
\begin{align}\label{eq:eg2_measure}
\begin{split}
&\lim_{N\to\infty}\PP \circ\left( (\varphi_t^{N,X})_{t\in [0,1]}, \delta_{\varphi_s^N}(dm)ds  \right)^{-1}
\\
&= \PP\circ\left( \left(\cL\left( \theta \left[(2t-1)^+ \wedge \frac{1}{2} \right] + (1-\theta)\left(t-\frac{1}{2}\right)^+ +W_t \vert \theta \right)\right)_{t\in [0,T]}, \Tilde{\Lambda} \right)^{-1}\coloneqq \PP^{\infty},
\end{split}
\end{align}
where
$\theta\sim\mathrm{Bernoulli}\left(\frac{1}{2}\right)$ 
is an independent random variable, and 
$\Tilde{\Lambda} (dm,dt)=\Tilde{\Lambda}_t(dm) d t$ with
\begin{equation*}
\Tilde{\Lambda}_t  = 1_{t<\frac{1}{2}}\delta_{\cL(W_t,0)} +1_{t\geq \frac{1}{2}}\left(\theta\delta_{\cL((2t-1) \wedge \frac{1}{2} +W_t, 2\cdot 1_{t\leq \frac{3}{4}})} + (1-\theta) \frac{ \delta_{\cL(t-\frac{1}{2}+W_t,2)} + \delta_{\cL(t-\frac{1}{2}+W_t,0)}  }{2}\right).
\end{equation*}
Moreover, $\lim_{N\to\infty}V^N=\lim_{N\to\infty} \Phi^N(\bu^N) = \Tilde{\Phi}(\PP^{\infty})=-\frac{1}{4}$, confirming Theorem \ref{thm:mean_field_limit2}.
Note that  $\tilde{\Lambda}_t$ is a $\cP(\cP(\RR\times A))$-valued  random variable,
where the randomness is given by the additional random variable $\theta$.

To see that it is essential to consider the limit in $\cP_2(\Tilde{\Omega})$,
consider instead $\overline{\Omega}\coloneqq C([0,T];\cP(\RR))\times \MM(\RR\times A)\times C([0,T];\RR)$   as the canonical space with  element $(\mu,\overline{\Lambda},B)$, and define the corresponding objective function $\overline{\Phi}:\cP_2(\overline{\Omega})\to\RR$ by
\begin{equation*}
\overline{\Phi}(\overline{\PP})\coloneqq \EE^{\overline{\PP}}\left[\int_0^T\int_{\RR^n\times A}F^{\infty}(t,x,a,\overline{\Lambda}_t)\overline{\Lambda}_t(dx,da)dt + \int_{\RR^n}G^{\infty}(x,\mu_T)\mu_T(dx) \right],
\end{equation*}
where $F^\infty$ and $G^\infty$ are defined by  \eqref{eq:mf_potential_cost}.
It holds that 
\begin{align}\label{eq:eg2_measure_compressed}
\begin{split}
&\lim_{N\to\infty}\PP\circ\left( (\varphi_t^{N,X})_{t\in [0,1]}, \varphi_s^N(dx,da)ds  \right)^{-1} 
\\
&= \PP\circ\left( \left(\cL\left( \theta \left[(2t-1)^+ \wedge \frac{1}{2} \right] + (1-\theta)\left(t-\frac{1}{2}\right)^+ +W_t \vert \theta \right)\right)_{t\in [0,T]}, \overline{\Lambda} \right)^{-1}\coloneqq \overline{\PP}^{\infty},
\end{split}
\end{align}
where $\overline{\Lambda}(dx, da,dt)=\overline{\Lambda}_t(dx, da)dt$ with
\begin{align*}
\begin{split}
\overline{\Lambda}_t(dx, da) &=  \cL\left( \theta \left[(2t-1)^+ \wedge \frac{1}{2} \right] + (1-\theta)\left(t-\frac{1}{2}\right)^+ +W_t \vert \theta \right)(dx) \bigg( \left(1_{t<\frac{1}{2}}+\theta 1_{t>\frac{3}{4}} \right)\delta_0(da) \\
&\quad+ \theta 1_{t\in [\frac{1}{2},\frac{3}{4}]} \delta_2(da) 
 + (1-\theta)1_{t\geq \frac{1}{2}}\left(\frac{1}{2}\delta_0 +\frac{1}{2}\delta_2\right)(da)  \bigg).
\end{split}
\end{align*}
However, $\overline{\Phi}(\overline{\PP}^{\infty})=0 \neq \lim_{N\to\infty} \Phi^N(\bu^N)$.

\end{example}

When the cost function $f$ depends only on the state law,
it suffices to  consider a two-layer hierarchy of measures.
\begin{example}
Suppose that  the function  $f$ is of the form  
$ 
f(t,x,a,\nu) = \hat f(t,x,a,\nu^x),
$ 
with $\nu^x(d x)\coloneqq \nu(dx, A)$. In this case,   $F^{\infty}$ in \eqref{eq:mf_potential_cost} reduces to 
\begin{equation*}
F^{\infty}(t,x,a,\nu) = \int_0^1  \begin{pmatrix} x\\a\end{pmatrix}^{\top} \partial_{(x,a)} \hat f(t,rx,ra,\nu^x\circ T_r'^{-1})dr,
\end{equation*}
which along with   Item \ref{item:state_marginal} of Definition \ref{def:measure_control} implies that  $\Tilde{\Phi}$ in \eqref{eq:mf_potential_measure} reduces to  
\begin{align*}
& \Tilde\Phi(\Tilde\PP) = \EE^{\Tilde\PP}\left[\int_0^T \int_{\RR^n\times A} \int_0^1  \begin{pmatrix} x\\a\end{pmatrix}^{\top} \partial_{(x,a)} \hat f(t,rx,ra,\Tilde\mu_t \circ T_r'^{-1})dr \overline\Lambda_t(dx,da)dt \right.\\
& \quad \left.+ \int_{\RR^n}G^{\infty}(x,\Tilde\mu_T)\Tilde\mu_T(dx) \right],
\end{align*}
where $\overline\Lambda_t(dx,da) \coloneqq \int_{\cP(\RR^n\times A)} m(dx,da)\Tilde\Lambda_t(dm)$,
in the sense that  for all $\phi\in C_b(\RR^n\times A)$,  
$\int \phi d\overline{\Lambda}_t =\int_{\cP(\RR^n\times A)}\left(\int \phi dm\right)\Tilde{\Lambda}_t(dm)$.
Hence,  it suffices to 
consider $\Tilde \Phi$
as a functional on 
$\mathcal P_2(\overline{\Omega})$ with $\overline\Omega\coloneqq C([0,T];\cP(\RR^n)) \times \MM(\RR^n\times A)  \times C([0,T];\RR^l)$.

\end{example}

\subsection{Extensions to other game structures}

We have established the convergence of $\alpha$-potential games with symmetric, weakly interacting players and open-loop strategies to potential MFGs. Naturally one wonders   whether analogous convergence results can be established for other classes of games.    

One extension is to games of inhomogeneous agents with moderate interactions, as studied in  \cite{flandoli2022n, djete2025non,djete2026non}. 
In fact, the $N$-player $\alpha$-potential functions in Theorem \ref{thm:alpha2}   allow for inhomogeneous players with general interaction structures. For instance,  along with the propagation of chaos for non-exchangeable control problems in \cite{djete2025non},
it provides a route to constructing potential functions for graphon MFGs and MFGs with moderate interactions.

Another direction is the convergence of $N$-player $\alpha$-potential games with closed-loop controls to the corresponding potential MFGs. Since open-loop and closed-loop controls are equivalent in the mean-field limit \cite{djete2023large}, the MFC objective \eqref{eq:mf_potential_measure}  should serve as a potential function for the closed-loop MFG. However, the construction of the $N$-player $\alpha$-potential function in Theorem \ref{thm:alpha2} relies on the decoupled structure of the state dynamics, which holds in the open-loop setting but fails for closed-loop controls, where perturbations of one player's policy can affect the states and controls of all players. It remains   to construct a suitable modification of the $N$-player $\alpha_N$-potential function for closed-loop controls that converges to \eqref{eq:mf_potential_measure}   
and has a vanishing $ \alpha_N$.
  
A similar difficulty arises for $N$-player games with interactions in the state dynamics. To the best of our knowledge, there is no existing characterization of potential MFGs with law dependence in the state dynamics. The most   related work is \cite{hofer2026optimal}, which starts from an MFC problem and shows that, under specific invertibility conditions on the law dependence in the state dynamics, the MFC objective serves as a potential function for a corresponding MFG. Meanwhile, for $N$-player games with state interactions, \cite{guo2025alpha} constructs an $\alpha$-potential function using sensitivity processes, which leads to an infinite-dimensional control problem. It remains to see   whether such an infinite-dimensional control formulation admits a meaningful limit under suitable conditions, and whether this limit can be interpreted as a potential function for the corresponding MFG.

\section{Proofs of Main Results}
\label{sec:proof}

\subsection{Proof of Theorem  \ref{thm:alpha2}}

\begin{proof}[Proof of Theorem  \ref{thm:alpha2}]
Fix an arbitrary   $t\in [0,T]$, $x\in \RR^{Nn}$, $a\in \RR^{Nk}$,
  $r,\varepsilon\in [0,1]$, 
$i\in [N]$, $x_i'\in \RR^n$ and $a_i'\in \RR^k$.
Write   $x'=(x_i',x_{-i})$, 
    and $x_{r,\varepsilon}= \varepsilon rx'+(1-\varepsilon)rx$. 
 Similarly, write 
 $a'=(a_i',a_{-i})$,
  and $a_{r,\varepsilon}= \varepsilon ra'+(1-\varepsilon) ra$.
 By \eqref{eq:F_G_bar} and  the fundamental theorem of calculus, 
\begin{align*}
F(t,x,a)&=\int_0^1 \sum_{j=1}^N 
\begin{pmatrix}  x_j \\a_j  \end{pmatrix}^\top  \partial_{(x_j,a_j)} f_j (t,rx,ra)  dr \\
& = \int_0^1 \sum_{j=1}^N 
\begin{pmatrix}  x_j \\a_j  \end{pmatrix}^\top \partial_{(x_j,a_j)} f_i (t,rx,ra)  dr - \int_0^1 \sum_{j\in[N]\backslash\{i\}} 
\begin{pmatrix}  x_j \\a_j  \end{pmatrix}^\top
 \partial_{(x_j,a_j)} \Delta_{i,j}^f(t,rx,ra)   dr \\
&= f_i(t,x,a) - f_i(t,0,0)- \int_0^1 \sum_{j\in[N]\backslash\{i\}} 
\begin{pmatrix}  x_j \\a_j  \end{pmatrix}^\top \partial_{(x_j,a_j)} \Delta_{i,j}^f (t,rx,ra)   dr.
\end{align*}
Then it follows that
\begin{align*}
& F(t,x',a')-F(t,x,a)\\
=&f_i(t,x',a')-f_i(t,x,a)
\\
&\quad - \int_0^1 \sum_{j\in[N]\backslash\{i\}} 
\left[
(\partial_{(x_j,a_j)} \Delta_{i,j}^f)(t,rx',ra') -   (\partial_{(x_j,a_j)} \Delta_{i,j}^f)(t,rx,ra) 
\right]^{\top}
\begin{pmatrix}  x_j \\a_j \end{pmatrix} dr \\
=&f_i(t,x',a')-f_i(t,x,a)
\\
&\quad - \int_0^1 \int_0^1 \sum_{j\in[N]\backslash\{i\}}   r \begin{pmatrix} x_i'-x_i \\ a_i'-a_i \end{pmatrix}^{\top}  \partial_{(x_i,a_i)(x_j,a_j)}^2 \Delta_{i,j}^f (t, x_{r,\varepsilon}, a_{r,\varepsilon})\begin{pmatrix} x_j   \\  a_j \end{pmatrix} d\varepsilon dr,
\end{align*}
where the last equality uses  the fundamental theorem of calculus. 
Similar computation shows that 
\begin{align*}
G(x')-G(x)=g_i(x')-g_i(x) - \int_0^1 \int_0^1 \sum_{j\in[N]\backslash\{i\}}   r ( x_i'-x_i )^{\top}  \partial_{x_ix_j}^2 \Delta_{i,j}^g (x_{r,\varepsilon}) x_j d\varepsilon dr.
\end{align*}

Now let  $i\in[N]$,
$u_i,u_i'\in \cA_i^N$ and $\bu_{-i}\in \cA_{-i}^N$. 
For each $r,\varepsilon\in[0,1]$, write 
$\bX^{i,r,\varepsilon}\coloneqq \varepsilon r (X_i^{u_i'},\bX_{-i}^{\bu_{-i}}) +(1-\varepsilon) r\bX^{\bu}$, and $\bu^{i,r,\varepsilon}\coloneqq \varepsilon r (u_i',\bu_{-i})+(1-\varepsilon) r \bu$.
Then 
\begin{equation*}
\begin{aligned}
& \left\vert \Phi(u_i',\bu_{-i}) - \Phi(\bu)- (J_i(u_i',\bu_{-i}) - J_i(\bu)) \right\vert\\
=& \left\vert\mathbb{E}\left[ \int_0^T \int_0^1 \int_0^1 \sum_{j\in[N]\backslash\{i\}}   r \begin{pmatrix} X_{t,i}^{u_i'}-X_{t,i}^{u_i} \\ u_{t,i}'-u_{t,i} \end{pmatrix}^{\top}  \partial_{(x_i,a_i)(x_j,a_j)}^2 \Delta_{i,j}^f (t, \bX_t^{i,r,\varepsilon}, \bu_t^{i,r,\varepsilon})\begin{pmatrix} X_{t,j}^{u_j}   \\  u_{t,j} \end{pmatrix} d\varepsilon dr dt\right.\right.\\
&\left.\left. + \int_0^1\int_0^1 \sum_{j\in [N]\backslash\{i\}} r (X_{T,i}^{u_i'}-X_{T,i}^{u_i})^{\top} \partial_{x_ix_j}^2 \Delta_{i,j}^g(\bX_T^{i,r,\varepsilon}) X_{T,j}^{u_j}  d\varepsilon dr   \right] \right\vert \\
\leq & \frac{1}{2}\sum_{j\in[N]\backslash\{i\}} \mathbb{E}\left[\int_0^T    \Vert \partial_{(x_i,a_i)(x_j,a_j)}^2\Delta_{i,j}^f\Vert_{\infty}  \left\vert \begin{pmatrix}X_{t,i}^{u_i'}-X_{t,i}^{u_i} \\  u_{t,i}'-u_{t,i}\end{pmatrix} \right\vert   \left\vert \begin{pmatrix}X_{t,j}^{u_j} \\  u_{t,j}\end{pmatrix} \right\vert dt  \right.\\
&\left.+ \Vert \partial_{x_ix_j}^2\Delta_{i,j}^g \Vert_{\infty} \vert X_{T,i}^{u_i'}-X_{T,i}^{u_i}\vert \vert X_{T,j}^{u_j} \vert\right] \\
\leq & \frac{1}{2}\sum_{j\in[N]\backslash\{i\}} \left\{ \Vert \partial_{(x_i,a_i)(x_j,a_j)}^2\Delta_{i,j}^f\Vert_{\infty}  \left\Vert \begin{pmatrix}X_i^{u_i'}-X_i^{u_i} \\ u_i'-u_i\end{pmatrix}\right\Vert_{\mathcal{H}^2(\mathbb{R}^{n+k})}     \left\Vert \begin{pmatrix}X_j^{u_j}\\     u_j\end{pmatrix} \right\Vert_{\mathcal{H}^2(\mathbb{R}^{n+k})} \right.\\
&\left.+ \Vert \partial_{x_ix_j}^2\Delta_{i,j}^g \Vert_{\infty} \Vert X_{T,i}^{u_i'}-X_{T,i}^{u_i}\Vert_{L^2} \Vert X_{T,j}^{u_j} \Vert_{L^2}\right\},
\end{aligned}
\end{equation*}
where the last line is by   the Cauchy-Schwarz inequality. Taking the supremum over all $i\in [N]$, $u_i,u_i'\in \cA_i^N$, and $\bu_{-i}\in \cA_{-i}^N$ yields the desired upper bound of $\alpha_N$. 
\end{proof}

\subsection{Proof of Theorem \ref{thm:mean_field_limit2}}

The following  growth and continuity properties of $F^{\infty}$ and $G^{\infty}$ will be used in the proof.
\begin{lemma}\label{lemma:cost_mean_field}
Suppose Assumption \ref{assumption:mean_field} holds. There exists a constant $C>0$ such that for all $(t,x,a,\nu,\mu)\in[0,T]\times\mathbb{R}^n\times\mathbb{R}^k\times\mathcal{P}_2(\mathbb{R}^n\times\mathbb{R}^k)\times \mathcal{P}_2(\mathbb{R}^n)$,
\begin{equation*}
\begin{aligned}
\vert F^{\infty}(t,x,a,\nu)\vert  \leq & C\left(1+\vert x\vert^2 +   \vert a\vert^2  + \mathcal{W}_2(\nu,\delta_0)^2 \right),\quad \vert G^{\infty}(x,\mu)\vert \leq  C\left(1+\vert x\vert^2  + \mathcal{W}_2(\mu,\delta_0)^2 \right).
\end{aligned}
\end{equation*}
Moreover, for each $t\in[0,T]$, $F^{\infty}(t,\cdot)$ and $G^{\infty}$ are continuous.
\end{lemma}
\begin{proof}
We only prove the required properties for $F^\infty$, since the corresponding properties of  $G^{\infty}$ follow by similar arguments.
By the definition \eqref{eq:mf_potential_cost} and  the growth conditions of $\partial_x f$ and $\partial_a f$ in Assumption \ref{assumption:mean_field}, 
\begin{equation*}
\begin{aligned}
 & \vert F^{\infty}(t,x,a,\nu)\vert 
 \\
 &
 \leq  C \int_0^1 \left( 1+ \vert rx\vert + \vert ra \vert + \int_{\mathbb{R}^d\times \mathbb{R}^k} (\vert x'\vert + \vert a'\vert) (\nu\circ T_r^{-1})(dx',da') \right)(\vert x\vert + \vert a \vert) dr \\
&\leq  C \left( 1+ \vert x\vert  +\vert a \vert +  \int_{\mathbb{R}^n\times\mathbb{R}^k} (\vert x' \vert + \vert a'\vert ) \nu(dx',da')  \right)(\vert x\vert + \vert a\vert)
\\
& 
\leq  C \left(1+\vert x\vert^2  + \vert a\vert^2  + \mathcal{W}_2(\nu,\delta_0)^2 \right),
\end{aligned}
\end{equation*}
where the last inequality is by the Cauchy-Schwarz inequality. 
To show the continuity, 
let $(\nu_k)_{k\in\mathbb{N}^*}\subset \mathcal{P}_2(\mathbb{R}^n\times\mathbb{R}^k)$ such that  $\lim_{k\to \infty}\mathcal{W}_2(\nu_k,\nu)=0$ for some $\nu \in \mathcal{P}_2(\mathbb{R}^n\times\mathbb{R}^k)$.
By   \cite[Proposition~A.1]{lacker2015mean}, $\{\nu_k\}_{k\in \NN^*}$ converges to $\nu$ weakly, and $(\nu_k)_{k\in\mathbb{N}^*}$ satisfies the uniform integrability condition:
\begin{equation*}
\lim_{R\rightarrow \infty} \sup_k \int_{\{(x,a):\vert x\vert^2 +\vert a\vert^2 \geq R\}} (\vert x\vert^2+\vert a\vert^2) \nu_k(dx,da) =0,
\end{equation*}
which implies that for each $r\in [0,1]$, $\{\nu_k\circ T_r^{-1}\}_{k\in \NN^*}$ converges to  $ \nu \circ T_r^{-1}$ weakly, and $(\nu_k\circ T_r^{-1})_{k\in\mathbb{N}}$ satisfies the uniform integrability condition:
\begin{equation*}
\lim_{R\rightarrow \infty} \sup_k \int_{\{(x,a):\vert x\vert^2+\vert a\vert^2 \geq R\}} (\vert x\vert^2 + \vert a\vert^2) (\nu_k\circ T_r^{-1})(dx,da) =0.
\end{equation*}
Thus, 
for each $r\in [0,1]$, $\{\nu_k\circ T_r^{-1}\}_{k\in \NN^*}$ converges to 
$  \nu \circ T_r^{-1}$ under $\mathcal{W}_2$. For each $t\in[0,T]$, the continuity of $F^{\infty}(t,\cdot)$ follows from the continuity of $\partial_x f(t,\cdot)$ and $\partial_a f(t,\cdot)$, and the Dominated Convergence Theorem.  
\end{proof}

\begin{proof}[Proof of Theorem \ref{thm:mean_field_limit2}]
For each $N\in\NN^*$, $\Phi^N$ has the following decomposition: for all $\bu=(u_1,\cdots,u_N)\in\cA^N$,
\begin{align}\label{eq:potential_diff}
\begin{split}
&\Phi^N(\bu) = \Tilde{\Phi}\left( \Tilde{\mathbb{P}}^N(\bu) \right) \\
&  + \frac{1}{N^2}\sum_{i\in[N]}\mathbb{E}\left[ \int_0^T \int_0^1 \begin{pmatrix} X_{t,i}^{u_i}  \\ u_{t,i}  \end{pmatrix}^{\top}  \partial_{\nu}f \left(t,rX_{t,i}^{u_i},ru_{t,i}, \frac{1}{N}\sum_{j\in[N]}\delta_{(rX_{t,j}^{u_j},ru_{t,j})}\right)(rX_{t,i}^{u_i},ru_{t,i})  dr dt \right.\\
&   \left.+ \int_0^1 (X_{T,i}^{u_i})^{\top}\partial_{\mu} g \left(rX_{T,i}^{u_i},\frac{1}{N}\sum_{j\in[N]}\delta_{rX_{T,j}^{u_j}} \right) (rX_{T,i}^{u_i}) dr\right].
\end{split}
\end{align}
For all $N\in\mathbb{N}^*$ and $\bu\in\cA^N$, define $\Tilde{\Phi}^N(\bu)\coloneqq \Tilde{\Phi}\left(\Tilde{\mathbb{P}}^N(\bu) \right)$. 
We shall prove the desired conclusions by leveraging  the convergence of   $(\Tilde{\Phi}^N)_{N\in \NN^*}$.

To this end, by Assumption \ref{assumption:mean_field} and standard SDE estimates \cite[Theorem 3.4.3]{zhang2017backward}, 
there exists a constant $C\ge 0$ such that for all $i\in \NN$, $u_i\in \cA_i^N$, and $t\in [0,T]$, $\mathbb E[|X^{u_i}_{t,i}|^2]\le C \left(1+\int|x'|^2\nu_i(dx')\right),$
which along with \eqref{eq:potential_diff} and $\sup_{N\geq 1}\frac{1}{N}\sum_{i=1}^N \int_{\mathbb{R}^n}\vert x'\vert^p \nu_i(dx')<\infty$ implies that  
\begin{equation}
\label{eq:Phi-tildePhi}
\lim_{N\rightarrow\infty}\sup_{\bu\in\mathcal{A}^N}\vert \Phi^N(\bu) - \Tilde{\Phi}^N(\bu)\vert = 0.
\end{equation}
Let 
$V^N(\nu_1,\cdots,\nu_N) \coloneqq\inf_{\bu\in\mathcal{A}^N} \Phi^N(\bu)$, and $\Tilde{V}^N(\nu_1,\cdots,\nu_N)\coloneqq \inf_{\bu\in\mathcal{A}^N} \Tilde{\Phi}^N(\bu)$ be the optimal value functions of $\Phi^N$ and
$\Tilde{\Phi}^N$, respectively. We prove by contradiction that 
\begin{equation}
\label{eq:V_N_tildeV_N}
\lim_{N\rightarrow\infty} \vert V^N(\nu_1,\cdots,\nu_N) - \Tilde{V}^N(\nu_1,\cdots,\nu_N)\vert = 0.
\end{equation}
Suppose that the statement fails, then there exists $\delta>0$ and a subsequence $(N_j)_{j\in\mathbb{N}^*}$ such that $\vert\Delta V^{N_j}\vert >\delta$ with $\Delta V^{N_j}\coloneqq  V^{N_j}(\nu_1,\cdots,\nu_{N_j}) - \Tilde{V}^{N_j}(\nu_1,\cdots,\nu_{N_j}) $, which implies either $\Delta V^{N_j}>\delta$ or $\Delta V^{N_j}<-\delta$. Without loss of generality, we can assume  that $\Delta V^{N_j} > \delta$. 
For all $N\in\mathbb{N}^*$, let $\Tilde{\bu}^N\in\cA^N$ be an $\varepsilon_N$-optimal control of $\Tilde{\Phi}^N$ with $\lim_{N\rightarrow\infty}\varepsilon_N=0$, then $\Tilde{\Phi}^N(\Tilde{\bu}^N)\leq \Tilde{V}^N(\nu_1,\cdots,\nu_N)+\varepsilon_N$. By \eqref{eq:Phi-tildePhi},   $\lim_{N\rightarrow\infty}\vert \Phi^N(\Tilde{\bu}^N)-\Tilde{V}^N(\nu_1,\cdots,\nu_N)\vert = \lim_{N\rightarrow\infty} \vert \Tilde{\Phi}^N(\Tilde{\bu}^N)-\Tilde{V}^N(\nu_1,\cdots,\nu_N)\vert=0$, which contradicts    $$\Phi^{N_j}(\Tilde{\bu}^{N_j})\geq V^{N_j}(\nu_1,\cdots,\nu_{N_j})>\Tilde{V}^{N_j}(\nu_1,\cdots,\nu_{N_j})+\delta.$$ The proof for $\Delta V^{N_j} <- \delta$ follows by similar arguments with $\Tilde{\Phi}^{N_j}$ and $\Phi^{N_j}$ interchanged, and $\Tilde V^{N_j}$ and $V^{N_j}$ interchanged. 
This proves \eqref{eq:V_N_tildeV_N}, which along with Lemma \ref{lemma:cost_mean_field} and \cite[Theorems 3.1 and 3.3]{djete2022extended} yields
\begin{equation*}
\lim_{N\rightarrow\infty}\left\vert  V^N(\nu_1,...,\nu_N)- V_V\left(\frac{1}{N}\sum_{i=1}^N \nu_i \right) \right\vert =\lim_{N\rightarrow\infty}\left\vert  \Tilde{V}^N(\nu_1,...,\nu_N)- V_V\left(\frac{1}{N}\sum_{i=1}^N \nu_i \right) \right\vert= 0.
\end{equation*}
This proves Item \ref{item:converge_value} of the statement. 

Now consider the sequence of $\varepsilon_N$-optimal controls $(\bu^N)_{N\in\NN^*}$ satisfying  $\Phi^N(\bu^N)\leq V^N(\nu_1,\cdots,\nu_N)+\varepsilon_N$ for all $N\in\NN^*$.  By \eqref{eq:Phi-tildePhi} and \eqref{eq:V_N_tildeV_N},
\begin{align*}
& \lim_{N\to\infty} \vert \Tilde{\Phi}^N(\bu^N)-\Tilde{V}^N(\nu_1,\cdots,\nu_N)\vert \\
\leq&  \lim_{N\to\infty} \vert \Tilde{\Phi}^N(\bu^N)- \Phi^N(\bu^N)\vert +\lim_{N\to\infty} \vert \Phi^N(\bu^N)-V^N(\nu_1,\cdots,\nu_N)\vert \\
&+\lim_{N\to\infty} \vert V^N(\nu_1,\cdots,\nu_N)-\Tilde{V}^N(\nu_1,\cdots,\nu_N)\vert = 0,
\end{align*}
which implies that $(\bu^N)_{N\in\NN^*}$ is also a sequence of $\varepsilon'_N$-optimal controls of $\Tilde{\Phi}^N$ for a  sequence $(\varepsilon'_N)_{N\in\mathbb{N}^*}\subset(0,\infty) $ such that $\lim_{N\to \infty}\varepsilon'_N=0$. Applying \cite[Proposition 3.4]{djete2022extended} yields the desired conclusion in Item \ref{item:converge_minimizer}.
\end{proof}

\subsection{Proofs of Lemma \ref{lemma:mfg_mfc} and  Theorem \ref{thm:alpha_potential_MFG}}

\begin{proof}[Proof of Lemma \ref{lemma:mfg_mfc}]
To prove Item \ref{item:embedding}, we  first verify that $\iota(\Tilde{\PP})$ satisfies Definition \ref{def:measure_control2}. The definition of $\iota$ implies that $\iota(\Tilde{\PP})(\Tilde{\mu}_0'=\nu)=\Tilde{\PP}(\Tilde{\mu}_0=\nu)=1$, which gives Item \ref{item:initial_state'} of Definition \ref{def:measure_control2}. To verify that $\Tilde{B}$ is an $(\iota(\Tilde{\PP}),\Tilde{\FF}')$-Wiener process, note that $\iota(\Tilde{\PP})\circ \Tilde{B}^{-1}=\Tilde{\PP}\circ\Tilde{B}^{-1} $, it remains to show that $\Tilde{B}$ has independent increments. For all $0<s<t<T$, $\phi\in C_b(\Tilde{\Omega}')$, and $\psi\in C_b(\RR^l)$,
\begin{align*}
& \EE^{\iota(\Tilde{\PP})}\left[ \psi\left(\Tilde{B}_t-\Tilde{B}_s\right)\phi\left(\Tilde{\mu}_{s\wedge\cdot},\Tilde{\mu}_{s\wedge\cdot}', \Tilde{\Pi}^s, \Tilde{B}_{s\wedge\cdot} \right) \right]\\
=& \EE^{\Tilde{\PP}}\left[ \psi\left(\Tilde{B}_t-\Tilde{B}_s\right)\phi\left(\Tilde{\mu}_{s\wedge\cdot},\Tilde{\mu}_{s\wedge\cdot}, \left(\Tilde{\Lambda}_r(dm)\delta_m(dm')dr \right)^s, \Tilde{B}_{s\wedge\cdot} \right) \right] \\
=& \EE^{\Tilde{\PP}}\left[ \psi\left(\Tilde{B}_t-\Tilde{B}_s\right) \right] \EE^{\Tilde{\PP}}\left[\phi\left(\Tilde{\mu}_{s\wedge\cdot},\Tilde{\mu}_{s\wedge\cdot}, \left(\Tilde{\Lambda}_r(dm)\delta_m(dm')dr \right)^s, \Tilde{B}_{s\wedge\cdot} \right) \right] \\
=& \EE^{\iota(\Tilde{\PP})}\left[ \psi\left(\Tilde{B}_t-\Tilde{B}_s\right)\right]\EE^{\iota(\Tilde{\PP})}\left[\phi\left(\Tilde{\mu}_{s\wedge\cdot},\Tilde{\mu}_{s\wedge\cdot}', \Tilde{\Pi}^s, \Tilde{B}_{s\wedge\cdot} \right) \right],
\end{align*}
where the second equality follows from $\Tilde B$ being a $(\Tilde{\mathbb{P}},\Tilde{\mathbb{F}})$-Wiener process (cf.~Definition \ref{def:measure_control}). Moreover, since 
for all $\varphi\in C_b([0,T]\times\cP(\RR^n\times A))$,
$$\int_0^T \int_{\cP(\RR^n\times A)^2} \varphi(t,m')\Tilde{\Lambda}_t(dm)\delta_m(dm')dt =\int_0^T \int_{\cP(\RR^n\times A)} \varphi(t,m)\Tilde{\Lambda}_t(dm)dt,$$
it follows that $\iota(\Tilde{\PP})(\Tilde{\Lambda}'=\Tilde{\Lambda})=1$, and 
\begin{align*}
&\iota(\Tilde{\PP})\left(N_t(f,\Tilde{\mu}',\Tilde{\Lambda}')=0,\forall f\in C_b^2(\RR^n),\forall t\in [0,T]\right)
\\
&= \Tilde{\PP}\left(N_t(f,\Tilde{\mu},\Tilde{\Lambda})=0,\forall f\in C_b^2(\RR^n),\forall t\in [0,T]\right)=1,
\end{align*}
which verifies Item \ref{item:state_martingale'} of Definition \ref{def:measure_control2}. 
Item \ref{item:state_marginal'} of Definition \ref{def:measure_control2} follows from 
$$\EE^{\iota(\Tilde{\PP})}\left[ \int_0^T 1_{\left\{\Tilde{\Lambda}_t'(\mathbb{Z}_{\Tilde{\mu}_t'})=1\right\}} dt \right] = \EE^{\Tilde{\PP}}\left[ \int_0^T 1_{\left\{\Tilde{\Lambda}_t(\mathbb{Z}_{\Tilde{\mu}_t})=1\right\}} dt \right] = T.
$$
The identity that 
$\iota(\Tilde{\PP})\circ(\Tilde{\mu},\Tilde{\Lambda},\Tilde{B})^{-1} = \Tilde{\PP} $ follows directly from the definition of $\iota$, and 
the fact that  $\iota(\Tilde{\PP})$ satisfies the consistency condition in Definition \ref{def:MFG_solution}  follows from $\iota(\Tilde{\PP})(\Tilde{\mu}=\Tilde{\mu}')=1$, and
\begin{equation*}
\EE^{\iota(\Tilde{\PP})}\left[\int_0^T 1_{\Tilde{\Pi}_t(\{m=m'\})=1} dt \right] = \EE^{\Tilde{\PP}}\left[\int_0^T 1_{\int_{\cP(\RR^n\times A)} \delta_m(\{m\})\Tilde{\Lambda}_t(dm)  =1} dt \right] = T.
\end{equation*}

We proceed to prove Item \ref{item:perturbation}.  It suffices to show  that $\Tilde{\mathbb{P}}^{\varepsilon}\coloneqq \kappa_{\varepsilon}(\Tilde{\PP})$, $\varepsilon\in [0,1]$, satisfies Definition \ref{def:measure_control}. First, by $\Tilde{\mathbb{P}}(\Tilde{\mu}_0'=\nu)=1$ and $\Tilde{\mathbb{P}}(\Tilde{\mu}_0=\nu)=1$,  
$ 
\Tilde{\mathbb{P}}^{\varepsilon}(\Tilde{\mu}_0=\nu)=\Tilde{\mathbb{P}}( \varepsilon\Tilde{\mu}_0'+(1-\varepsilon)\Tilde{\mu}_0=\nu)=1,
$ 
which implies Item \ref{item:initial_state} of Definition \ref{def:measure_control}. For $\Tilde{\mathbb{P}}$-almost every $\omega\in\Tilde{\Omega}'$, for each $f\in C_b^2(\mathbb{R}^n)$ and $t\in [0,T]$,
\begin{equation*}
\begin{aligned}
 &N_t(f,\Tilde{\mu}^{\varepsilon},\Tilde{\Lambda}^{\varepsilon})= \langle f(\cdot-\gamma \Tilde{B}_t),\varepsilon\Tilde{\mu}_t'+(1-\varepsilon)\Tilde{\mu}_t\rangle - \langle f,\varepsilon\Tilde{\mu}_0'+(1-\varepsilon)\Tilde{\mu}_0\rangle \\
 &\quad - \int_0^t\int_{\cP(\RR^n\times A)^2}\int_{\mathbb{R}^n\times A} \mathcal{L}_r[f(\cdot-\gamma \Tilde{B}_r)](x,a)(\varepsilon m'+(1-\varepsilon)m)(dx,da)\Tilde{\Pi}_r(dm,dm')dr \\
=& \varepsilon \left( \langle f(\cdot-\gamma \Tilde{B}_t),\Tilde{\mu}_t' \rangle - \langle f,\Tilde{\mu}_0'\rangle - \int_0^t\int_{\cP(\RR^n\times A)}\int_{\mathbb{R}^n\times A} \mathcal{L}_r[f(\cdot-\gamma \Tilde{B}_r)](x,a)m'(dx,da)\Tilde{\Lambda}'_r(dm')dr \right) \\
&+ (1-\varepsilon) \bigg( \langle f(\cdot-\gamma \Tilde{B}_t),\Tilde{\mu}_t \rangle - \langle f,\Tilde{\mu}_0\rangle 
\\
&
- \int_0^t\int_{\cP(\RR^n\times A)}\int_{\mathbb{R}^n\times A} \mathcal{L}_r[f(\cdot-\gamma \Tilde{B}_r)](x,a)m(dx,da)\Tilde{\Lambda}_r(dm)dr \bigg)\\
=& \varepsilon N_t(f,\Tilde{\mu}',\Tilde{\Lambda}') + (1-\varepsilon)N_t(f,\Tilde{\mu},\Tilde{\Lambda}) = 0.
\end{aligned}
\end{equation*}
This along with  $\Tilde{\mathbb{P}}^{\varepsilon}\circ\Tilde{B}^{-1}=\Tilde{\mathbb{P}}\circ\Tilde{B}^{-1}$ verifies  Item \ref{item:state_martingale} of Definition \ref{def:measure_control}. To verify   Item \ref{item:state_marginal} of Definition \ref{def:measure_control}, for $d\Tilde{\PP}\otimes dt$ almost every $(t,\omega)\in[0,T]\times\Tilde{\Omega}'$,
\begin{equation*}
\begin{aligned}
\Tilde{\Lambda}_t^{\varepsilon}(\mathbb{Z}_{\Tilde{\mu}_t^{\varepsilon}}) =& \Tilde{\Pi}_t\left(\{ (\varepsilon m'+ (1-\varepsilon)m)(dx,A)= \varepsilon \Tilde{\mu}_t'(dx)+ (1-\varepsilon)\Tilde{\mu}_t(dx) \} \right) \\
\geq & \Tilde{\Pi}_t\left(\{ m(dx,A)= \Tilde{\mu}_t (dx)\} \cap  \{ m'(dx,A)= \Tilde{\mu}_t'(dx) \} \right) \\
\geq &\Tilde{\Pi}_t\left(\{ m(dx,A)= \Tilde{\mu}_t (dx)\}  \right)+\Tilde{\Pi}_t\left( \{ m'(dx,A)= \Tilde{\mu}_t'(dx) \} \right)-1 \\
= & \Tilde{\Lambda}_t\left(\{ m(dx,A)= \Tilde{\mu}_t (dx)\}  \right)+\Tilde{\Lambda}'_t\left( \{ m'(dx,A)= \Tilde{\mu}_t'(dx) \} \right)-1=1,
\end{aligned}
\end{equation*}
where the last line follows from 
$\Tilde{\Lambda}_t'(\mathbb{Z}_{\Tilde{\mu}_t'})=1$ due to  $\Tilde{\PP}\in\cP_V'(\nu)$, 
and $\Tilde{\Lambda}_t(\mathbb{Z}_{\Tilde{\mu}_t})=1$ due to $\Tilde{\PP}\circ (\Tilde{\mu},\Tilde{\Lambda},\Tilde{B}  )^{-1}\in\cP_V(\nu)$.
This finishes the proof.
\end{proof}

\begin{proof}[Proof of Theorem \ref{thm:alpha_potential_MFG}]
In the subsequent proof, 
we frequently interchange  differentiation and integration, which can be  justified by the Leibniz integral rule  (see, e.g., \cite[Theorem A.5.1]{durrett2019probability})
and the regularity of the cost functions. Let $\Tilde{\PP}\in\cP_V(\nu)$, and $\Tilde{\PP}'\in \cA(\Tilde{\PP})\subset \cP_V'(\nu)$. By the definitions of 
$\Tilde{\Phi}$ and $\kappa_\varepsilon$ and that $\Tilde{\PP}'\circ(\Tilde{\mu},\Tilde{\Lambda},\Tilde{B})^{-1} = \Tilde{\PP}$,
\begin{equation}\label{eq:mf_alpha_potential0}
\begin{aligned}
& \Tilde{J}(\Tilde{\mathbb{P}}') - \Tilde{J}(\iota(\Tilde{\mathbb{P}})) - \lim_{\varepsilon\downarrow 0}\frac{\Tilde{\Phi}(\kappa_{\varepsilon}(\Tilde{\mathbb{P}}')) - \Tilde{\Phi}(\Tilde{\mathbb{P}})}{\varepsilon} \\
&= \mathbb{E}^{\Tilde{\mathbb{P}}'}\left[ \int_0^T \int_{\cP(\RR^n\times A)^2} \bigg\{\left\langle f(t,\cdot,\cdot,m),m'-m \right\rangle \right.
\\
&\quad   - \lim_{\varepsilon\downarrow 0}\frac{ \langle F^{\infty}(t,\cdot,\cdot, \varepsilon m'+(1-\varepsilon)m), \varepsilon m'+(1-\varepsilon)m \rangle - \langle F^{\infty}(t,\cdot,\cdot, m), m \rangle  }{\varepsilon}  \bigg\}   \Tilde{\Pi}_t(dm,dm')dt  \\
&\quad  +\left\langle g(\cdot,\Tilde{\mu}_T),\Tilde{\mu}_T'-\Tilde{\mu}_T \right\rangle \\
&\left. - \lim_{\varepsilon\downarrow 0}\frac{ \langle G^{\infty}(\cdot, \varepsilon \Tilde{\mu}_T'+(1-\varepsilon)\Tilde{\mu}_T), \varepsilon \Tilde{\mu}_T'+(1-\varepsilon)\Tilde{\mu}_T \rangle - \langle G^{\infty}(\cdot, \Tilde{\mu}_T), \Tilde{\mu}_T \rangle  }{\varepsilon}  \right],
\end{aligned}
\end{equation}
where $F^{\infty}$ and $G^{\infty}$ are defined in \eqref{eq:mf_potential_cost}. 
We focus on the term with  the terminal cost $g$, since the same argument applies to  the term with the running cost $f$. 
For all $\mu,\mu'\in\cP_2(\mathbb{R}^n)$,  
\begin{equation*}
\begin{aligned}
& \lim_{\varepsilon\downarrow 0}\frac{ \langle G^{\infty}(\cdot, \varepsilon \mu'+(1-\varepsilon)\mu), \varepsilon \mu'+(1-\varepsilon)\mu \rangle - \langle G^{\infty}(\cdot, \mu), \mu \rangle  }{\varepsilon} \\
=& \int_0^1\lim_{\varepsilon\downarrow 0}\frac{1}{\varepsilon}\left[  \int_{\mathbb{R}^n} x^{\top} \partial_x g(rx,(\varepsilon\mu'+(1-\varepsilon)\mu)\circ T_r'^{-1})(\varepsilon\mu'+(1-\varepsilon)\mu)(dx) \right. \\
&\left.- \int_{\mathbb{R}^n} x^{\top} \partial_x g(rx,\mu\circ T_r'^{-1})\mu(dx) \right] dr \\
=& \int_0^1 \int_{\mathbb{R}^n} x^{\top} \partial_x g(rx,\mu\circ T_r'^{-1})(\mu'-\mu)(dx)dr \\
& + \int_0^1 \int_{\mathbb{R}^n} \int_{\mathbb{R}^n} (x')^{\top} \frac{\delta (\partial_x g)}{\delta \mu} (rx',\mu\circ T_r'^{-1})(rx) \mu(dx')  (\mu'-\mu)(dx)dr\\
\coloneqq & I_1 + I_2.
\end{aligned}
\end{equation*}

The first term $I_1$ can be decomposed into two parts: the term corresponding to the unilateral deviation of the representative player and a residual term:
\begin{equation*}
\begin{aligned}
I_1=& \int_0^1 \int_{\mathbb{R}^n} x^{\top} \partial_x g(rx, \mu)(\mu'-\mu)(dx)dr -\int_0^1 \int_{\mathbb{R}^n} x^{\top} \left[ \partial_x g(rx, \mu) -\partial_x g(rx,\mu\circ T_r'^{-1})\right](\mu'-\mu)(dx)dr \\
=& \int_{\mathbb{R}^n }g(x,\mu)(\mu'-\mu)(dx) - \int_{\mathbb{R}^n}\int_{\mathbb{R}^n}\int_0^1\int_r^1 x^{\top} [\partial_{\mu}\partial_x g(rx,\mu\circ T_s'^{-1})(sx')]x' dsdr\mu(dx') (\mu'-\mu)(dx) \\
=& \int_{\mathbb{R}^n }g(x,\mu)(\mu'-\mu)(dx) - \int_{\mathbb{R}^n}\int_{\mathbb{R}^n}\int_0^1\int_0^r x^{\top} [\partial_{\mu}\partial_x g(sx,\mu\circ T_r'^{-1})(rx')]x' dsdr\mu(dx') (\mu'-\mu)(dx) \\
=& \int_{\mathbb{R}^n }g(x,\mu)(\mu'-\mu)(dx) - \int_{\mathbb{R}^n}\int_{\mathbb{R}^n}\int_0^1\int_0^r (x')^{\top} [\partial_x \partial_{\mu} g(sx,\mu\circ T_r'^{-1})(rx')]x dsdr\mu(dx') (\mu'-\mu)(dx)\\
=& \int_{\mathbb{R}^n }g(x,\mu)(\mu'-\mu)(dx) - \int_{\mathbb{R}^n}\int_{\mathbb{R}^n}\int_0^1 (x')^{\top} [\partial_{\mu} g(rx,\mu\circ T_r'^{-1})(rx')] dr\mu(dx') (\mu'-\mu)(dx),
\end{aligned}
\end{equation*}
where the second equality follows from the fundamental theorem of calculus, the third equality follows from Fubini's theorem, the fourth equality is due to \cite[Remark 4.16]{carmona2018probabilistic2}, and the last equality uses $\int_{\RR^n}g(0,\mu)(\mu'-\mu)(dx)=0$. 

Therefore, let $\pi$ be an optimal coupling of $\mu,\mu'\in \cP_2(\mathbb{R}^n)$, and $(X,X')\sim\pi$, then the following holds:
\begin{align*}
& \left\langle g(\cdot,\mu),\mu'-\mu \right\rangle - \lim_{\varepsilon\downarrow 0}\frac{ \langle G^{\infty}(\cdot, \varepsilon \mu'+(1-\varepsilon)\mu), \varepsilon \mu'+(1-\varepsilon)\mu \rangle - \langle G^{\infty}(\cdot, \mu), \mu \rangle  }{\varepsilon}  \\
=&  \int_{\mathbb{R}^n} \int_{\mathbb{R}^n} \int_0^1 (x')^{\top}\left(\partial_{\mu} g(rx,\mu\circ T_r'^{-1})(rx') - \frac{\delta(\partial_x g)}{\delta \mu}(rx',\mu\circ T_r'^{-1})(rx) \right) dr \mu(dx') (\mu'-\mu)(dx)  \\
=&  \mathbb{E}_{(X,X')\sim\pi}\left[  \int_{\mathbb{R}^n} \int_0^1 \int_0^1 (x')^{\top}\left([\partial_{\mu} \partial_x g(r(sX'+(1-s)X),\mu\circ T_r'^{-1})(rx')]^{\top} \right.\right.\\
&\left.\left.- \partial_{\mu}\partial_x g(rx',\mu\circ T_r'^{-1})(r(sX'+(1-s)X)) \right) r(X'-X) dsdr \mu(dx') \right],
\end{align*}
which together with the Cauchy-Schwarz inequality implies that
\begin{equation}\label{eq:g_asymmetry}
\begin{aligned}
& \left\vert\left\langle g(\cdot,\mu),\mu'-\mu \right\rangle - \lim_{\varepsilon\downarrow 0}\frac{ \langle G^{\infty}(\cdot, \varepsilon \mu'+(1-\varepsilon)\mu), \varepsilon \mu'+(1-\varepsilon)\mu \rangle - \langle G^{\infty}(\cdot, \mu), \mu \rangle  }{\varepsilon}\right\vert  \\
\leq & \frac{1}{2} \sup_{x,x'\in\mathbb{R}^n, \mu\in \cP_2(\mathbb{R}^n)} \vert \partial_{\mu}\partial_x g(x,\mu)(x')-[\partial_{\mu}\partial_x g(x',\mu)(x)]^{\top} \vert \cW_2(\mu,\delta_0) \cW_2(\mu,\mu').
\end{aligned}
\end{equation}
Similarly, it can be shown that for all $t\in[0,T]$, and $m,m'\in\cP_2(\mathbb{R}^n\times A)$,
\begin{equation*}
\begin{aligned}
& \left\vert \left\langle f(t,\cdot,\cdot,m),m'-m \right\rangle - \lim_{\varepsilon\downarrow 0}\frac{ \langle F^{\infty}(t,\cdot,\cdot, \varepsilon m'+(1-\varepsilon)m), \varepsilon m'+(1-\varepsilon)m \rangle - \langle F^{\infty}(t,\cdot,\cdot, m), m \rangle  }{\varepsilon}  \right\vert \\
\leq & \frac{1}{2} \sup_{\substack{(x,a),(x',a')\in \mathbb{R}^n\times \mathbb{R}^k\\ \nu\in \cP_2(\mathbb{R}^n\times \mathbb{R}^k)}}\vert \partial_{\nu}\partial_{(x,a)}f(t,x,a,\nu)(x',a')-[\partial_{\nu}\partial_{(x,a)}f(t,x',a',\nu)(x,a)]^{\top} \vert \cW_2(m,\delta_0)\cW_2(m,m'),
\end{aligned}
\end{equation*}
which together with \eqref{eq:mf_alpha_potential0} and \eqref{eq:g_asymmetry} implies \eqref{eq:alpha_potential_MFG}.

Finally, let  
$\PP^*$ be a   minimizer  of $\Tilde{\Phi}$. 
Given 
$\Tilde{\mathbb{P}}\in\cP_V'(\nu)$ such that $\Tilde{\mathbb{P}}\circ (\Tilde{\mu},\Tilde{\Lambda},\tilde B)^{-1}
=\iota(\PP^*)\circ (\Tilde{\mu},\Tilde{\Lambda},\tilde B)^{-1}=\PP^*$, and hence 
by \eqref{eq:alpha_potential_MFG}, 
$\Tilde{J}(\Tilde{\PP})-\Tilde{J}(\iota(\PP^*)) \geq \lim_{\varepsilon\downarrow 0} \frac{\Tilde{\Phi}(\kappa_{\varepsilon}(\Tilde{\PP})) - \Tilde{\Phi}(\PP^*)}{\varepsilon} - \alpha_{\infty}\ge -\alpha_{\infty}.
$ 
This along with the fact that 
$\iota(\PP^*)$ satisfies Condition \ref{item:consistency} in Definition \ref{def:MFG_solution} shows that $\iota(\PP^*)$ is an $\alpha_{\infty}$-MFE for the MFG. 
\end{proof}

\subsection{Proof of Theorem \ref{thm:alpha0_equiv}}

\begin{proof}[Proof of Theorem \ref{thm:alpha0_equiv}]
The proof proceeds by showing the equivalence between Items \ref{item:alpha_N_limit}
and 
\ref{item:curl-free},
and 
the equivalence between Items \ref{item:curl-free} and \ref{item:cost_potential}.
We only present  the proof for $f$, as the proof for $g$ is analogous.
To simplify the notation, we write $z=(x,a)\in \RR^{Nn}\times \RR^{Nk}$, $z_i=(x_i,a_i)\in \RR^{n}\times \RR^{k}$, and $\Delta^f_{i,j}=f_i-f_j$  for all $i,j\in [N]$.

We first show the equivalence between Items \ref{item:alpha_N_limit}
and 
\ref{item:curl-free}.
For all $i,j\in[N]$, $i\neq j$, $t\in[0,T]$, $z=(x_l,a_l)_{l\in[N]}\in\mathbb{R}^{Nn}\times \RR^{Nk}$,
by \cite[Proposition 5.35]{carmona2018probabilistic1},
\cite[Remark 4.16]{carmona2018probabilistic2},
and the boundedness of $\partial^2_{\nu}f$,  
\begin{equation}
\label{eq:hessian_Deltaf}
\begin{aligned}
&\partial^2_{(x_i,a_i)(x_j,a_j)}\Delta_{i,j}^f(t,z) 
\\
&= \frac{1}{N}\left[ \partial_{(x,a)}\partial_{\nu} f\left(t,z_i,\frac{1}{N}\sum_{l\in[N]}\delta_{z_l}\right)(z_j)-\partial_{\nu}\partial_{(x,a)} f\left(t,z_j,\frac{1}{N}\sum_{l\in[N]}\delta_{z_l}\right)(z_i) \right] \\
&\quad + \frac{1}{N^2}\left[ \partial^2_{\nu} f\left(t,z_i,\frac{1}{N}\sum_{l\in[N]}\delta_{z_l}\right)(z_j,z_i) - \partial^2_{\nu} f\left(t,z_j,\frac{1}{N}\sum_{l\in[N]}\delta_{z_l}\right)(z_j,z_i) \right] \\
&= \frac{1}{N} \hat f \left(t, z_i,z_j, \frac{1}{N}\sum_{l\in[N]}\delta_{z_l} \right)+ O\left( \frac{1}{N^2}\right),
\end{aligned}
\end{equation}
where  
$O\left({1}/{N^2}\right)$ denotes a term that vanishes at rate $1/N^2$ as $N\to\infty$ (uniformly in $(t,z,\nu)$), and  
$\hat f:[0,T]\times\RR^{n+k}\times\mathbb{R}^{n+k}\times\mathcal{P}_2(\RR^n\times\RR^k)\rightarrow\mathbb{R}^{(n+k)\times(n+k)}$ is given by \begin{equation*}
\hat f(t,z,z',\nu)\coloneqq  \left[\partial_{\nu}\partial_{(x,a)} f\left(t,z,\nu\right)(z')\right]^{\top}-\partial_{\nu}\partial_{(x,a)} f\left(t,z',\nu\right)(z) .
\end{equation*} 
Note that under Item \ref{item:curl-free}, $\hat f=0$, which implies Item \ref{item:alpha_N_limit}.
It remains to show Item \ref{item:alpha_N_limit} implies Item \ref{item:curl-free}.
For all $N\in\mathbb{N}^*$ and $z,z'\in\mathbb{R}^{n+k}$, define the set of empirical measures generated by $N$ elements in $\mathbb{R}^{n+k}$ including $z$ and $z'$:
$$ 
\mathcal{P}_{z,z'}^N\coloneqq \left\{ \frac{1}{N}\sum_{j\in[N]} \delta_{z_j}\,\Big\vert\, z,z'\in \{z_j\}_{j\in [N]}\subset \mathbb{R}^{n+k} \right\}.
$$ 
By \eqref{eq:hessian_Deltaf},
$\Vert \partial^2_{(x_i,a_i) (x_j,a_j)}\Delta_{i,j}^f\Vert_{\infty}=\Vert \partial^2_{(x_1,a_1) (x_2,a_2)}\Delta_{1,2}^f\Vert_{\infty}$ for all $j\not =i$, and hence
\begin{align}
\label{eq:limsup_max_f}
\begin{split}
& \limsup_{N\rightarrow\infty}\max_{i\in [N]}\sum_{j\in [N]}\Vert \partial^2_{(x_i,a_i) (x_j,a_j)}\Delta_{i,j}^f\Vert_{\infty} =\limsup_{N\rightarrow\infty}N \Vert \partial^2_{(x_1,a_1) (x_2,a_2)}\Delta_{1,2}^f\Vert_{\infty} \\
&\leq \limsup_{N\rightarrow\infty}\esssup_{t\in [0,T]}\sup_{z\in\RR^{N(n+k)}} \left\vert \hat f \left( t,z_1,z_2, \frac{1}{N}\sum_{l\in[N]}\delta_{z_l} \right)\right\vert
\\
&=\limsup_{N\rightarrow\infty}\esssup_{t\in [0,T]}\sup_{z,z'\in\mathbb{R}^{n+k},\nu\in\mathcal{P}_{z,z'}^N} \vert \hat f(t,z,z',\nu)\vert \\
& \leq \esssup_{t\in [0,T]}\sup_{z,z'\in\mathbb{R}^{n+k},\nu\in\mathcal{P}_2(\mathbb{R}^{n+k})} \vert \hat f(t,z,z',\nu)\vert=\|\hat f\|_\infty.
\end{split}
\end{align}
We now show 
$\liminf_{N\rightarrow\infty}\max_{i\in [N]}\sum_{j\in [N]}\Vert \partial^2_{(x_i,a_i) (x_j,a_j)}\Delta_{i,j}^f\Vert_{\infty} =\|\hat f\|_\infty$.
Denote by  $\operatorname{Leb}_{\RR}$   the Lebesgue measure on $[0,T]$. For each $N\in\NN^*$, there exists $D_N\in\cB([0,T])$ such that $\operatorname{Leb}_{\RR}([0,T]\setminus D_N)=0$, and $$\sup_{t\in D_N,z\in\RR^{N(n+k)} }\vert \partial^2_{(x_1,a_1) (x_2,a_2)}\Delta_{1,2}^f(t,z)\vert = \Vert \partial^2_{(x_1,a_1) (x_2,a_2)}\Delta_{1,2}^f\Vert_{\infty}.$$ Let $D\coloneqq \bigcap_{N\in\NN^*} D_N$, then $\operatorname{Leb}_{\RR}([0,T]\setminus D)=0$, and
\begin{align*}
&\liminf_{N\rightarrow\infty}\max_{i\in [N]}\sum_{j\in [N]}\Vert \partial^2_{(x_i,a_i) (x_j,a_j)}\Delta_{i,j}^f\Vert_{\infty} =\liminf_{N\rightarrow\infty}N \Vert \partial^2_{(x_1,a_1) (x_2,a_2)}\Delta_{1,2}^f\Vert_{\infty} \\
=& \liminf_{N\rightarrow\infty}\sup_{t\in D_N,z\in\RR^{N(n+k)} }N\vert \partial^2_{(x_1,a_1) (x_2,a_2)}\Delta_{1,2}^f(t,z)\vert 
\\
\geq &\liminf_{N\rightarrow\infty}\sup_{t\in D,z\in\RR^{N(n+k)}  }N\vert \partial^2_{(x_1,a_1) (x_2,a_2)}\Delta_{1,2}^f(t,z)\vert \\
=& \liminf_{N\rightarrow\infty}\sup_{t\in D,z,z'\in\RR^{n+k} ,\nu\in\mathcal{P}_{z,z'}^N } \vert \hat f(t,z,z',\nu)\vert.
\end{align*}
For any $\varepsilon >0$, let $t_{\varepsilon}\in D$, $z_{\varepsilon},z_{\varepsilon}'\in \RR^{n+k}$, and  $ \nu_{\varepsilon}\in \mathcal{P}_2(\mathbb{R}^{n+k})$ be such that $$|\hat f(t_{\varepsilon},z_{\varepsilon},z_{\varepsilon}', \nu_{\varepsilon})| > \sup_{t\in D,z,z'\in\mathbb{R}^{n+k} ,\nu\in\cP_2(\RR^{n+k}) } \vert \hat f(t,z,z',\nu)\vert-\varepsilon.$$
Let $(\nu_{\varepsilon}^N)_{N\in\NN^*}\subset \cP_2(\RR^{n+k})$ 
 be   such that $\nu_{\varepsilon}^N \in \mathcal{P}_{z_{\varepsilon},z_{\varepsilon}'}^N$ for all $N \in\mathbb{N}^*$, and $\lim_{N\rightarrow\infty} \mathcal{W}_2(\nu_{\varepsilon}^N, \nu_{\varepsilon})=0$.
In particular, one can choose $\nu_{\varepsilon}^N=\frac{1}{N}(\delta_{z_{\varepsilon}}+\delta_{z_{\varepsilon}'}+\sum_{m=3}^N\delta_{Z_m})$,
where $(Z_m)_{m\ge 3}$
are realizations of i.i.d.~random variables with   distribution $\nu_{\varepsilon}\in\cP_2(\RR^{n+k})$ (see \cite[Section 5.1.2]{carmona2018probabilistic1}). Then  by the continuity of $\hat f$ in the measure component, 
\begin{align*}
&\liminf_{N\rightarrow\infty}\sup_{t\in D,z,z'\in\RR^{n+k} ,\nu\in\mathcal{P}_{z,z'}^N } \vert \hat f(t,z,z',\nu)\vert
\geq \liminf_{N\rightarrow\infty} |\hat f(t_{\varepsilon},z_{\varepsilon},z_{\varepsilon}',\nu_{\varepsilon}^N)| = |\hat f(t_{\varepsilon},z_{\varepsilon},z_{\varepsilon}', \nu_{\varepsilon})| \\
&> \sup_{t\in D,z,z'\in\mathbb{R}^{n+k} ,\nu\in\cP_2(\RR^{n+k}) } \vert \hat f(t,z,z',\nu)\vert-\varepsilon \geq \esssup_{t\in [0,T]}\sup_{ z,z'\in\mathbb{R}^{n+k} ,\nu\in\cP_2(\RR^{n+k}) } \vert \hat f(t,z,z',\nu)\vert-\varepsilon
\\
&=\|\hat f\|_\infty-\varepsilon.
\end{align*}
This along with \eqref{eq:limsup_max_f} implies that $\lim_{N\rightarrow\infty}\max_{i\in [N]}\sum_{j\in [N]}\Vert \partial^2_{(x_i,a_i)(x_j,a_j)}\Delta_{i,j}^f\Vert_{\infty} = \|\hat f\|_\infty$. Hence,   Item \ref{item:alpha_N_limit} implies Item \ref{item:curl-free}.

We then show 
the equivalence between Items \ref{item:curl-free} and \ref{item:cost_potential}.
Suppose Item  \ref{item:cost_potential} holds.
The regularity of $\frac{\delta}{\delta \nu} F$ implies that 
$\nu\mapsto \partial_{\nu}  F(t,\nu)(z)$ is L-differentiable, and hence 
for a.e.~$t\in [0,T]$,
\begin{equation*}
\partial_{\nu}\partial_{(x,a)} f(t,z,\nu)(z') = \partial^2_{\nu}  F(t,\nu)(z,z') = [\partial^2_{\nu} F(t,\nu)(z',z)]^{\top} = [\partial_{\nu}\partial_{(x,a)} f(t,z',\nu)(z)]^{\top},
\end{equation*}
where the second identity used \cite[Corollary 5.89]{carmona2018probabilistic1}. This shows that 
Item  \ref{item:cost_potential}
implies 
 Item \ref{item:curl-free}. 
 Conversely, suppose that 
 Item \ref{item:curl-free} holds. 
We shall verify that $F$ given by \eqref{eq:mf_potential_integrated} satisfies Item    \ref{item:cost_potential}. 
Let 
$(\Omega,\mathcal{F},\mathbb{P})$
be an atomless probability space. 
Consider the lifted function $\tilde F:[0,T]\times L^2(\Omega,\mathcal{F},\mathbb{P};\mathbb{R}^{n+k})\rightarrow\mathbb{R}$ given by 
\begin{equation*}
\tilde F(t,X)\coloneqq F(t,\mathcal{L}(X)) = \int_0^1 \mathbb{E}\left[ \partial_{(x,a)} f\left(t,rX, \mathcal{L}(rX) \right)^{\top} X \right] dr.
\end{equation*}
and calculate the G\^ateaux derivative of $\tilde F(t,\cdot)$ for each $t\in [0,T]$. Let $Y\in L^2(\Omega,\mathcal{F},\mathbb{P};\mathbb{R}^{n+k})$, and $(\tX,\tY) \in L^2(\Tilde{\Omega},\Tilde{\mathcal{F}},\Tilde{\mathbb{P}};\mathbb{R}^{n+k})\times L^2(\Tilde{\Omega},\Tilde{\mathcal{F}},\Tilde{\mathbb{P}};\mathbb{R}^{n+k})$ be an independent copy of the pair $(X,Y)$, then  
\begin{equation*}
\begin{aligned}
&\frac{d}{d\varepsilon} \tilde F(t,X+\varepsilon Y)\vert_{\varepsilon=0} =  \frac{d}{d\varepsilon} \left.\left(\int_0^1 \mathbb{E}\left[ \partial_{(x,a)} f\left(t,rX+\varepsilon r Y, \mathcal{L}(rX+\varepsilon r Y) \right)^{\top} (X+\varepsilon Y) \right] dr \right)\right\vert_{\varepsilon=0} \\
&\quad= \int_0^1 \mathbb{E}\left[ \frac{d}{d\varepsilon}\left\{ \partial_{(x,a)} f\left(t,rX+\varepsilon r Y, \mathcal{L}(rX+\varepsilon r Y) \right)^{\top} (X+\varepsilon Y) \right\} \bigg\vert_{\varepsilon=0}\right] dr \\
&\quad= \int_0^1 \mathbb{E}\left[ \partial_{(x,a)} f(t,rX,\mathcal{L}(rX))^{\top} Y + rX^{\top}\partial^2_{(x,a)(x,a)}f(t,rX,\mathcal{L}(rX))Y \right.\\
&\quad\quad\left.+ r\Tilde{\mathbb{E}}\left[ X^{\top} \partial_{\nu}\partial_{(x,a)} f(t,rX,\mathcal{L}(rX))(r\tX) \tY \right] \right] dr \\
&\quad= \int_0^1 \mathbb{E}\left[\left\{ \partial_{(x,a)} f(t,rX,\mathcal{L}(rX)) + r \partial^2_{(x,a)(x,a)} f(t,rX,\mathcal{L}(rX)) X \right.\right.\\
&\quad\quad\left.\left.+ r \Tilde{\mathbb{E}}\left[ \partial_{\nu}\partial_{(x,a)} f(t,r\Tilde{X},\mathcal{L}(rX))(rX)^{\top}\Tilde{X}  \right] \right\}^{\top} Y   \right] dr.
\end{aligned}
\end{equation*}
Moreover, the derivative of $r\mapsto \mathbb{E}\left[ r\partial_{(x,a)} f(t,rX,\mathcal{L}(rX))^{\top} Y \right]$ is given by
\begin{equation*}
\begin{aligned}
&\frac{d}{dr}\mathbb{E}\left[ r\partial_{(x,a)} f(t,rX,\mathcal{L}(rX))^{\top} Y \right] \\
=& \mathbb{E}\left[ \left\{ \partial_{(x,a)} f(t,rX,\mathcal{L}(rX)) + r\partial^2_{(x,a)(x,a)}f(t,rX,\mathcal{L}(rX))X \right.\right.\\
&\left.\left.+ r\Tilde{\mathbb{E}}\left[  \partial_{\nu}\partial_{(x,a)} f(t,rX,\mathcal{L}(rX))(r\tX) \tX \right] \right\}^{\top} Y\right] \\
=& \mathbb{E}\left[ \left\{ \partial_{(x,a)} f(t,rX,\mathcal{L}(rX)) + r\partial^2_{(x,a)(x,a)}f(t,rX,\mathcal{L}(rX))X \right.\right.\\
&\left.\left.+ r\Tilde{\mathbb{E}}\left[  \partial_{\nu}\partial_{(x,a)} f(t,r\tX,\mathcal{L}(rX))(rX)^{\top} \tX \right] \right\}^{\top} Y\right], \\
\end{aligned}
\end{equation*}
where the second equality used \eqref{eq:curl_free}. Thus   for  a.e.~$t\in [0,T]$,
\begin{align*}
\frac{d}{d\varepsilon} \tilde F(t,X+\varepsilon Y)\vert_{\varepsilon=0} 
&= \int_0^1 \frac{d}{dr}\mathbb{E}\left[ r\partial_{(x,a)} f(t,rX,\mathcal{L}(rX))^{\top} Y \right] dr 
\\
&= \mathbb{E}\left[ \partial_{(x,a)} f(t,X,\mathcal{L}(X))^{\top} Y\right].
\end{align*}
We further claim that $L^2(\Omega,\mathcal{F},\mathbb{P};\mathbb{R}^{n+k})\ni X \mapsto \partial_{(x,a)} f(t,X,\mathcal{L}(X)) \in L^2(\Omega,\mathcal{F},\mathbb{P};\mathbb{R}^{n+k})$ is continuous. Let $\{X_n\}_{n\in\mathbb{N}^*}$ be a sequence of random variables in $L^2(\Omega,\mathcal{F},\mathbb{P};\mathbb{R}^{n+k})$ converging to $X$ in $L^2$, then $\mathcal{L}(X_n)$ converges to $\mathcal{L}(X)$ under $\mathcal{W}_2$. By the continuity of $\partial_{(x,a)} f(t,\cdot)$ and the continuous mapping theorem \cite[Lemma 5.3]{kallenbert2021foundations}, $\partial_{(x,a)} f(t,X_n,\mathcal{L}(X_n))$ converges to $\partial_{(x,a)} f(t,X,\mathcal{L}(X))$ in probability. The  linear growth of $\partial_{(x,a)} f(t,\cdot)$ implies that  $\{ \vert \partial_{(x,a)} f(t,X_n,\mathcal{L}(X_n))\vert^2\}_{n\in\mathbb{N}^*}$ is uniformly integrable, which along with 
Vitali's theorem yields that 
$\partial_{(x,a)} f(t,X_n,\mathcal{L}(X_n))$ converges to $\partial_{(x,a)} f(t,X,\mathcal{L}(X))$ in $L^2$. Therefore, for  a.e.~$t\in [0,T]$, $\tilde F(t,\cdot)$ is Fr\'echet differentiable, 
$\partial_{\nu}F(t,\mathcal{L}(X))(X) =\partial_{(x,a)} f(t,X,\mathcal{L}(X))$,
which implies that $\partial_{\nu}F(t,\nu)(x,a)=\partial_{(x,a)} f(t,x,a,\nu)$, $\nu$-almost surely. This shows that $F$ satisfies \eqref{eq:carmona_condition}.

For the regularity of $F$, applying \cite[Proposition 5.51]{carmona2018probabilistic1}, for a.e.~$t\in [0,T]$, 
\begin{equation*}
\frac{\delta F}{\delta \nu}(t,\nu)(x,a) = f(t,x,a,\nu)-\int_{\RR^n\times\RR^k}f(t,x',a',\nu)\nu(dx',da').
\end{equation*}
The regularity of $f$ then implies the desired regularity of $\frac{\delta F}{\delta \nu}$. This finishes the proof.
\end{proof}

\section*{Declaration of Generative AI and AI-Assisted Technologies in the Manuscript Preparation Process}
During the preparation of this work, the authors used ChatGPT in order to improve the academic writing and verify the correctness of the mathematical proofs and calculations. After using this tool/service, the authors reviewed and edited the content as needed and take full responsibility for the content of the published article.

\section*{Acknowledgements}
The authors thank Mao Fabrice Djete,
Alp\'ar R. M\'esz\'aros,
Xiaolu Tan, and  
the XIII Bachelier World Congress of the Bachelier Finance Society and the 3rd ETH-HK-Imperial Joint Workshop on Quantitative Finance, for their valuable suggestions and comments.

\begin{appendix}
\section{Proofs of Examples \ref{eg:common_noise} and \ref{eg:no_common_noise}}\label{sec:eg_proof}
\begin{proof}[Proof of Example \ref{eg:common_noise}]

Let $X_i^N\coloneqq X_i^{u_i^N}$ for each $i\in [N]$. First, define $(Y_i)_{i\in\NN^*}\subset C([0,1];\RR)$ and $Y\in C([0,1];\RR)$ as: for each $i\in\NN^*$ and $t\in [0,1]$, let $Y_{t,i}\coloneqq \frac{t}{2}+W_{t,i}^1+B_t$, and $Y_t\coloneqq \frac{t}{2}+W_t^1+B_t$, then
\begin{equation*}
X_{t,i}^{N,1}-Y_{t,i} = 1_{t\geq \frac{1}{2}}1_{B_{\frac{1}{2}}<0}\left(\min\left\{ t-\frac{\lfloor 2Nt \rfloor}{2N}, \frac{1}{4N}\right\} -\frac{1}{2}\left( t-\frac{\lfloor 2Nt \rfloor}{2N}\right)\right), \quad t\in [0,1],
\end{equation*}
and $\Vert X_i^{N,1}-Y_i \Vert_\infty \leq \frac{1}{8N}$ for all $N\in \NN^*$ and $i\in [N]$. For all bounded Lipschitz function $\phi:C([0,1];\RR^2)\to\RR$,
\begin{align*}
&\lim_{N\to\infty }\frac{1}{N}\sum_{i\in [N]} \phi((X_i^{N,1},X_i^{N,2})) \\
= &\lim_{N\to\infty }\frac{1}{N}\sum_{i\in [N]} \left[\phi((X_i^{N,1},W_i^2+B)) - \phi((Y_i,W_i^2+B)) \right] + \lim_{N\to\infty }\frac{1}{N}\sum_{i\in [N]} \phi((Y_i,W_i^2+B)) \\
=& \lim_{N\to\infty }\frac{1}{N}\sum_{i\in [N]} \phi((Y_i,W_i^2+B)) = \EE\left[\phi((Y,W^2+B))\vert B \right],\quad \PP\text{-almost surely},
\end{align*}
where the second equality is due to the Lipschitz continuity of $\phi$, and the last equality is by the conditional strong law of large numbers. By \cite[Theorem 11.3.3]{dudley2002real} and the proof of \cite[Theorem 11.4.1]{dudley2002real}, there exists a countable convergence-determining family of bounded Lipschitz functions on $C([0,1];\RR^2)$  for weak convergence in $\cP(C([0,1];\RR^2))$. Therefore, $\frac{1}{N}\sum_{i\in[N]} \delta_{X_i^N}$ weakly converges to $\cL(Y,W^2+B\vert B)$, $\PP$-almost surely. Moreover, by \cite[Proposition A.1]{lacker2015mean} and standard SDE estimates,
\begin{align*}
&\lim_{N\to\infty}\cW_2\left(\frac{1}{N}\sum_{i\in[N]} \delta_{X_i^N},\cL(Y,W^2+B\vert B) \right)=0, \\
&\lim_{N\to\infty}\sup_{t\in [0,1]}\cW_2\left(\varphi_t^{N,X}, \cL(Y_t,W_t^2+B_t\vert B) \right)=0, \quad \PP\text{-almost surely}.
\end{align*}
For any bounded Lipschitz function $\psi: [0,1]\times \cP_2(\RR^2\times A) \to\RR$,
\begin{align*}
  \int_0^1 \psi\left(t, \varphi_t^N \right)dt =& \int_0^{\frac{1}{2}}\psi\left(t, \varphi_t^{N,X}\otimes \delta_{\frac{1}{2}}  \right)dt + 1_{B_{\frac{1}{2}}\geq 0} \int_{\frac{1}{2}}^1 \psi\left(t, \varphi_t^{N,X}\otimes \delta_{\frac{1}{2}}  \right)dt \\
&+ 1_{B_{\frac{1}{2}}<0} \left[ \sum_{j=N}^{2N-1}\int_{\frac{j}{2N}}^{\frac{2j+1}{4N}} \psi\left(t, \varphi_t^{N,X}\otimes \delta_1  \right)dt +\int_{\frac{2j+1}{4N}}^{\frac{j+1}{2N}} \psi\left(t, \varphi_t^{N,X}\otimes \delta_0  \right)dt   \right].
\end{align*}
Denote $\Tilde{\mu}_t\coloneqq \cL(Y_t,W_t^2+B_t\vert B)$ for each $t\in [0,1]$, then it follows that $\PP$-almost surely,
\begin{align*}
&\lim_{N\to\infty}\left[\int_0^{\frac{1}{2}}\psi\left(t, \varphi_t^{N,X}\otimes \delta_{\frac{1}{2}}  \right)dt + 1_{B_{\frac{1}{2}}\geq 0} \int_{\frac{1}{2}}^1 \psi\left(t, \varphi_t^{N,X}\otimes \delta_{\frac{1}{2}}  \right)dt \right] \\
=& \int_0^{\frac{1}{2}}\psi\left(t, \Tilde{\mu}_t\otimes \delta_{\frac{1}{2}}  \right)dt + 1_{B_{\frac{1}{2}}\geq 0} \int_{\frac{1}{2}}^1 \psi\left(t, \Tilde{\mu}_t\otimes \delta_{\frac{1}{2}}  \right)dt,
\end{align*}
and
\begin{align*}
& \lim_{N\to\infty} \left[ \sum_{j=N}^{2N-1}\int_{\frac{j}{2N}}^{\frac{2j+1}{4N}} \psi\left(t, \varphi_t^{N,X}\otimes \delta_1  \right)dt +\int_{\frac{2j+1}{4N}}^{\frac{j+1}{2N}} \psi\left(t, \varphi_t^{N,X}\otimes \delta_0  \right)dt   \right] \\
=& \lim_{N\to\infty} \left[ \sum_{j=N}^{2N-1}\int_{\frac{j}{2N}}^{\frac{2j+1}{4N}} \psi\left(t, \Tilde{\mu}_t\otimes \delta_1  \right)dt +\int_{\frac{2j+1}{4N}}^{\frac{j+1}{2N}} \psi\left(t, \Tilde{\mu}_t\otimes \delta_0  \right)dt   \right] \\
=&\int_{\frac{1}{2}}^1 \frac{\psi(t,\Tilde{\mu}_t\otimes\delta_1) + \psi(t,\Tilde{\mu}_t\otimes\delta_0)}{2}dt  + \lim_{N\to\infty} \sum_{j=N}^{2N-1}\int_{\frac{j}{2N}}^{\frac{2j+1}{4N}} \left[\Tilde{\psi}(t) - \Tilde{\psi}\left(t+\frac{1}{4N}\right) \right]dt \\
=& \int_{\frac{1}{2}}^1 \frac{\psi(t,\Tilde{\mu}_t\otimes\delta_1) + \psi(t,\Tilde{\mu}_t\otimes\delta_0)}{2}dt,
\end{align*}
where for each $t\in [0,1]$, $\Tilde{\psi}(t)\coloneqq \frac{\psi(t,\Tilde{\mu}_t\otimes\delta_1)-\psi(t,\Tilde{\mu}_t\otimes\delta_0)}{2}$, the first equality is due to the Lipschitz continuity of $\psi$, and the last equality is by the Lipschitz continuity of $\psi$ and the uniform continuity of $t\mapsto \Tilde{\mu}_t$ under $\cW_2$. Therefore, it follows that $\PP$-almost surely,
\begin{align*}
&\lim_{N\to\infty} \int_0^1 \psi\left(t, \varphi_t^N \right)dt \\
&\quad = \int_0^1 \left[\left(1_{t<\frac{1}{2}} + 1_{t\geq\frac{1}{2}}1_{B_{\frac{1}{2}}\geq 0} \right)\psi\left(t, \Tilde{\mu}_t\otimes \delta_{\frac{1}{2}}  \right) +  1_{t\geq\frac{1}{2}}1_{B_{\frac{1}{2}}<0} \frac{\psi(t,\Tilde{\mu}_t\otimes\delta_1) + \psi(t,\Tilde{\mu}_t\otimes\delta_0)}{2}\right] dt.
\end{align*}
By weak convergence together with uniform integrability of the second moments, \\$\lim_{N\to\infty} \cW_2 (\delta_{\varphi_t^N}(dm)dt, \Tilde{\Lambda})=0$, $\PP$-almost surely. This proves \eqref{eq:eg1_measure}.

For any bounded continuous function $\zeta:[0,1]\times\RR^2\times A\to \RR$, define $\psi_{\zeta}:[0,1]\times\cP_2(\RR^2\times A)\to\RR$ by: $\psi_{\zeta}(t,\nu)\coloneqq \int_{\RR^2\times A}\zeta(t,x,a)\nu(dx,da)$ for all $(t,\nu)\in [0,1]\times \cP_2(\RR^2\times A)$, then $\psi_{\zeta}$ is bounded continuous. Consequently,
\begin{align*}
&\lim_{N\to\infty} \int_0^1 \frac{1}{N}\sum_{i\in [N]}\zeta(t,X_{t,i}^N,u_{t,i}^N)dt =\lim_{N\to\infty}\int_0^1 \psi_{\zeta}(t,\varphi_t^N) dt \\
&\quad = \int_0^1 \int_{\cP_2(\RR^2\times A)}\psi_{\zeta}(t,m)\Tilde{\Lambda}_t(dm)dt = \int_0^1 \int_{\RR^2\times A}\zeta(t,x,a)\overline{\Lambda}_t(dx,da)dt,
\end{align*}
which implies \eqref{eq:eg1_measure_compressed}.

For the convergence of value functions, 
  the definition of $f$ and $g$ yields
\begin{align*}
& \partial_x f(t,x,a,\nu) = 0, \quad \partial_a f(t,x,a,\nu) = 2a - 2\left( \int_A a'\nu(\RR^2,da') \right), \\
& \partial_{\nu} f(t,x,a,\nu)(x',a') = \left(0, 0, -2a \right),\\
& \partial_x g(x,\mu) = 2\int_{\RR^2}\begin{pmatrix} 1 & -1 \\ -1 & 1 \end{pmatrix} x' \mu(dx') +\begin{pmatrix} -1 \\ 1 \end{pmatrix},\quad \partial_{\mu} g(x,\mu)(x') = 2 \begin{pmatrix} 1 & -1 \\ -1 & 1 \end{pmatrix} x.
\end{align*}
Hence by \eqref{eq:F_G_N},
\begin{align*}
& F^N(t,x,a,\nu) = \left(1- \frac{1}{N}\right)a^2 - a\int_A a' \nu(\RR^2,da'),  \\
& G^N(x,\mu) = \begin{pmatrix} 1\\ -1 \end{pmatrix}^{\top} x \left(\int_{\RR^2 }\begin{pmatrix} 1\\ -1 \end{pmatrix}^{\top} x' \mu(dx') -1 \right) + \frac{1}{N}  x^{\top} \begin{pmatrix} 1 & -1 \\ -1 & 1 \end{pmatrix} x .
\end{align*}
Note that for all $N\in\NN^*$, $\int_{\RR^n\times A}F^N(t,x,a,\nu)\nu(dx,da) \geq -\frac{1}{N}$ due to   $A=[0,1]$,  and $\int_{\RR^2}G^N(x,\mu)\mu(dx)\geq -\frac{1}{4}$ by completing the square, which implies that  $V^N\geq -\frac{1}{N}-\frac{1}{4}$. Moreover,
\begin{align*}
\Phi^N(\bu^N) & = \EE\bigg[ -\frac{1}{N}\left( \frac{1}{8}+\frac{1}{8}1_{B_{\frac{1}{2}}\geq 0} + \frac{1}{4}1_{B_{\frac{1}{2}}< 0} \right)+\left( \frac{1}{N}\sum_{i\in[N]}(W_{1,i}^1-W_{1,i}^2) \right)^2
\\
&\quad-\frac{1}{4} +\frac{1}{N^2}\sum_{i\in[N]} \left(\frac{1}{2}+W_{1,i}^1-W_{1,i}^2 \right)^2 \bigg]\\
& =\frac{63}{16N}-\frac{1}{4},
\end{align*}
which implies that $V^N\leq \Phi^N(\bu^N)=\frac{63}{16N}-\frac{1}{4}$. Therefore, 
$\Phi^N(\bu^N)\le V^N+\frac{79}{16N}$ and 
$\lim_{N\to\infty}V^N=-\frac{1}{4}$.

For the MFC problem, 
a direct computation shows that $F^\infty$ and $G^\infty$ in \eqref{eq:mf_potential_cost} are given by
\begin{align*}
F^{\infty}(t,x,a,\nu)=& a^2 - a \int_{A} a'  \nu(\RR^2,da'), \\
G^{\infty}(x,\mu)=& \begin{pmatrix} 1\\ -1 \end{pmatrix}^{\top} x \left(\int_{\RR^2 }\begin{pmatrix} 1\\ -1 \end{pmatrix}^{\top} x' \mu(dx') -1 \right).
\end{align*}
Since
$\int_{\RR^n\times A}F^{\infty}(t,x,a,\nu)\nu(dx,da)\geq 0$ and $\int_{\RR^n}G^{\infty}(x,\mu)\mu(dx)\geq -\frac{1}{4}$,  $V_V\geq -\frac{1}{4}$. This along with  $\Tilde{\Phi}(\PP^{\infty})=-\frac{1}{4}$ implies that  $V_V=\Tilde{\Phi}(\PP^{\infty})=-\frac{1}{4}$.
\end{proof}

\begin{proof}[Proof of Example \ref{eg:no_common_noise}]
Let $X_i^N\coloneqq X_i^{u_i^N}$ for each $i\in [N]$. First, define $(Y_i)_{i\in\NN^*}\subset C([0,1];\RR)$ and $Y\in C([0,1];\RR)$ such that  for all $t\in [0,1]$,   $Y_{t,i}\coloneqq 1_{W_{\frac{1}{2},1}\geq 0}\left[(2t-1)^+\wedge\frac{1}{2}\right] + 1_{W_{\frac{1}{2},1} <  0} \left(t-\frac{1}{2} \right)^+ + W_{t,i}$, and $Y_t \coloneqq 1_{W_{\frac{1}{2},1}\geq 0}\left[(2t-1)^+\wedge\frac{1}{2}\right] + 1_{W_{\frac{1}{2},1} <  0} \left(t-\frac{1}{2} \right)^+ + W_t$. Then $\Vert X_i^N-Y_i\Vert \leq \frac{1}{4N}$ for all $N\in\NN^*$ and $i\in[N]$. For all bounded Lipschitz function $\phi:C([0,1];\RR)\to\RR$,
\begin{align*}
\lim_{N\to\infty}\frac{1}{N}\sum_{i\in [N]} \phi(X_i^N) = \lim_{N\to\infty}\frac{1}{N}\sum_{i\in [N]} \phi(Y_i) = \lim_{N\to\infty}\frac{1}{N-1}\sum_{i\in [N]\backslash\{1\}} \phi(Y_i) = \EE\left[ \phi(Y)\vert 1_{W_{\frac{1}{2},1}\geq 0} \right],
\end{align*}
where the first equality is due to the Lipschitz continuity of $\phi$, and the last equality is by the conditional strong law of large numbers. It follows by similar arguments to those used in the proof of Example \ref{eg:common_noise} that $\PP${-almost surely},
\begin{align*}
\lim_{N\to\infty}\cW_2\left(\frac{1}{N}\sum_{i\in[N]} \delta_{X_i^N}, \cL(Y \vert 1_{W_{\frac{1}{2},1}\geq 0}) \right)=0, \quad \lim_{N\to\infty}\sup_{t\in [0,1]}\cW_2\left(\varphi_t^{N,X}, \cL(Y_t \vert 1_{W_{\frac{1}{2},1}\geq 0}) \right)=0.
\end{align*}
For any bounded Lipschitz function $\psi: [0,1]\times \cP_2(\RR \times A) \to\RR$,
\begin{align*}
  &\int_0^1 \psi\left(t, \varphi_t^N \right)dt \\
  =& \int_0^{\frac{1}{2}}\psi\left(t, \varphi_t^{N,X}\otimes \delta_{0}  \right)dt + 1_{W_{\frac{1}{2},1}\geq 0} \left[\int_{\frac{1}{2}}^{\frac{3}{4}} \psi\left(t, \varphi_t^{N,X}\otimes \delta_2  \right)dt + \int_{\frac{3}{4}}^1 \psi\left(t, \varphi_t^{N,X}\otimes \delta_0  \right)dt\right] \\
&+ 1_{W_{\frac{1}{2},1}<0} \left[ \sum_{j=N}^{2N-1}\int_{\frac{j}{2N}}^{\frac{2j+1}{4N}} \psi\left(t, \varphi_t^{N,X}\otimes \delta_2  \right)dt +\int_{\frac{2j+1}{4N}}^{\frac{j+1}{2N}} \psi\left(t, \varphi_t^{N,X}\otimes \delta_0  \right)dt   \right].
\end{align*}
Denote $\Tilde{\mu}_t\coloneqq \cL(Y_t\vert 1_{W_{\frac{1}{2},1}\geq 0})$, then it follows by similar arguments to those used in the proof of Example \ref{eg:common_noise} that
\begin{align*}
&\lim_{N\to\infty}\int_0^1\psi(t,\varphi_t^N)dt 
\\
&= \int_0^{\frac{1}{2}}\psi(t,\Tilde{\mu}_t\otimes \delta_0)dt + 1_{W_{\frac{1}{2},1}\geq 0} \left[\int_{\frac{1}{2}}^{\frac{3}{4}} \psi\left(t, \Tilde{\mu}_t\otimes \delta_2  \right)dt + \int_{\frac{3}{4}}^1 \psi\left(t, \Tilde{\mu}_t\otimes \delta_0  \right)dt\right] \\
&\quad  + 1_{W_{\frac{1}{2},1}< 0} \int_{\frac{1}{2}}^1 \frac{\psi(t,\Tilde{\mu}_t\otimes\delta_2) + \psi(t,\Tilde{\mu}_t\otimes\delta_0)}{2} dt.
\end{align*}
Combined with the fact that $\cL(\theta,W)=\cL(1_{W_{\frac{1}{2},1}\geq 0},W)$, this yields \eqref{eq:eg2_measure}.

For the convergence of value functions, the definitions of $f$ and $g$ yield 
\begin{align*}
& \partial_x f(t,x,a,\nu) = 0, \quad \partial_a f(t,x,a,\nu) = 2a - 2\left( \int_A a'\nu(\RR,da') \right), \\
& \partial_{\nu} f(t,x,a,\nu)(x',a') = \left(0, -2a \right),\\
& \partial_x g(x,\mu) = 2\int_{\RR} x' \mu(dx') -1,\quad \partial_{\mu} g(x,\mu)(x') = 2  x.
\end{align*}
Hence by \eqref{eq:F_G_N},
\begin{align*}
 & F^N(t,x,a,\nu) = \left(1- \frac{1}{N}\right)a^2 - a \int_{A} a'  \nu(\RR,da'),  \\
 & G^N(x,\mu) = x\left( \int_{\RR} x'\mu(dx') -1 \right) + \frac{x^2}{N},
\end{align*}
Note that $\int_{\RR\times A}F^N(t,x,a,\nu)\nu(dx,da)\geq -\frac{4}{N}$ and $\int_{\RR}G^N(x,\mu)\mu(dx)\geq -\frac{1}{4}$ for each $N\in\NN^*$, it follows that $V^N\geq -\frac{4}{N}-\frac{1}{4}$. Moreover,
\begin{align*}
\Phi^N(\bu^N)=\EE\left[ -\frac{1}{N}+\left( \frac{1}{N}\sum_{i\in [N]} W_{1,i} \right)^2 -\frac{1}{4} + \frac{1}{N^2} \sum_{i\in [N]} \left( \frac{1}{2}+W_{1,i} \right)^2 \right] = \frac{5}{4N}-\frac{1}{4},
\end{align*}
then $V^N\leq \Phi^N(\bu^N)= \frac{5}{4N}-\frac{1}{4}$. Therefore, $\lim_{N\to\infty} V^N=-\frac{1}{4}$, and $\bu^N$ is an $\varepsilon_N$-optimal control for some $\varepsilon_N\leq \frac{21}{4N}$.

For the MFC problem, 
a direct computation shows that $F^\infty$ and $G^\infty$ in \eqref{eq:mf_potential_cost} are given by
\begin{equation*}
F^{\infty}(t,x,a,\nu)=a^2 - a \int_{A} a'  \nu(\RR,da'), \quad G^{\infty}(x,\mu)=   x  \left( \int_{\RR} x'\mu(dx') -1 \right).
\end{equation*}
Since
$\int_{\RR\times A}F^{\infty}(t,x,a,\nu)\nu(dx,da)\geq 0$ and $\int_{\RR}G^{\infty}(x,\mu)\mu(dx)\geq -\frac{1}{4}$,  $V_V\geq -\frac{1}{4}$. This along with $\Tilde{\Phi}(\PP^{\infty})=-\frac{1}{4}$ yields $V_V=\Tilde{\Phi}(\PP^{\infty})=-\frac{1}{4}$. 

The proof of \eqref{eq:eg2_measure_compressed} is analogous to that of \eqref{eq:eg1_measure_compressed}. This finishes the proof.
\end{proof}
\end{appendix}


\begin{thebibliography}{10}

\bibitem{aurell2018mean}
A.~Aurell and B.~Djehiche.
\newblock Mean-field type modeling of nonlocal crowd aversion in pedestrian
  crowd dynamics.
\newblock {\em SIAM Journal on Control and Optimization}, 56(1):434--455, 2018.

\bibitem{cardaliaguet2019master}
P.~Cardaliaguet, F.~Delarue, J.-M. Lasry, and P.-L. Lions.
\newblock {\em The Master Equation and the Convergence Problem in Mean Field
  Games: (AMS-201)}, volume~2.
\newblock Princeton University Press, 2019.

\bibitem{cardaliaguet2017learning}
P.~Cardaliaguet and S.~Hadikhanloo.
\newblock Learning in mean field games: The fictitious play.
\newblock {\em ESAIM: COCV}, 23(2):569--591, 2017.

\bibitem{cardaliaguet2018mean}
P.~Cardaliaguet and C.-A. Lehalle.
\newblock Mean field game of controls and an application to trade crowding.
\newblock {\em Mathematics and Financial Economics}, 12(3):335--363, 2018.

\bibitem{carmona2023synchronization}
R.~Carmona, Q.~Cormier, and H.~M. Soner.
\newblock Synchronization in a {K}uramoto mean field game.
\newblock {\em Communications in Partial Differential Equations},
  48(9):1214--1244, 2023.

\bibitem{carmona2018probabilistic1}
R.~Carmona and F.~Delarue.
\newblock {\em Probabilistic Theory of Mean Field Games with Applications I}.
\newblock Springer Cham, 2018.

\bibitem{carmona2018probabilistic2}
R.~Carmona and F.~Delarue.
\newblock {\em Probabilistic Theory of Mean Field Games with Applications II}.
\newblock Springer Cham, 2018.

\bibitem{carmona2015probabilistic}
R.~Carmona and D.~Lacker.
\newblock {A probabilistic weak formulation of mean field games and
  applications}.
\newblock {\em The Annals of Applied Probability}, 25(3):1189 -- 1231, 2015.

\bibitem{cecchin2022selection}
A.~Cecchin and F.~Delarue.
\newblock Selection by vanishing common noise for potential finite state mean
  field games.
\newblock {\em Communications in Partial Differential Equations},
  47(1):89--168, 2022.

\bibitem{cecchin2022weak}
A.~Cecchin and F.~Delarue.
\newblock Weak solutions to the master equation of potential mean field games.
\newblock {\em arXiv preprint arXiv:2204.04315}, 2022.

\bibitem{cecchin2026convergence}
A.~Cecchin and J.~Dianetti.
\newblock Convergence for linear quadratic potential mean field games.
\newblock {\em arXiv preprint arXiv:2602.14842}, 2026.

\bibitem{delarue2020selection}
F.~Delarue and R.~F. Tchuendom.
\newblock Selection of equilibria in a linear quadratic mean-field game.
\newblock {\em Stochastic Processes and their Applications}, 130(2):1000--1040,
  2020.

\bibitem{di2025alpha}
X.~Di, A.~Hu, Z.~Wang, and Y.~Zhang.
\newblock $\alpha$-potential games for decentralized control of connected and
  automated vehicles.
\newblock {\em arXiv preprint arXiv:2512.05712}, 2025.

\bibitem{djete2022extended}
M.~F. Djete.
\newblock {Extended mean field control problem: a propagation of chaos result}.
\newblock {\em Electronic Journal of Probability}, 27:1 -- 53, 2022.

\bibitem{djete2023large}
M.~F. Djete.
\newblock Large population games with interactions through controls and common
  noise: convergence results and equivalence between open-loop and closed-loop
  controls.
\newblock {\em ESAIM: Control, Optimisation and Calculus of Variations}, 29:39,
  2023.

\bibitem{djete2023mean}
M.~F. Djete.
\newblock {Mean field games of controls: On the convergence of Nash
  equilibria}.
\newblock {\em The Annals of Applied Probability}, 33(4):2824 -- 2862, 2023.

\bibitem{djete2025non}
M.~F. Djete.
\newblock A non-exchangeable mean field control problem with controlled
  interactions.
\newblock {\em arXiv preprint arXiv:2511.00288}, 2025.

\bibitem{djete2026non}
M.~F. Djete.
\newblock Non--exchangeable mean field games with moderate interactions and
  common noise.
\newblock {\em arXiv preprint arXiv:2605.14901}, 2026.

\bibitem{djete2022mckean2}
M.~F. Djete, D.~Possama\"{\i}, and X.~Tan.
\newblock {McKean–Vlasov} optimal control: Limit theory and equivalence
  between different formulations.
\newblock {\em Mathematics of Operations Research}, 47(4):2891--2930, 2022.

\bibitem{djete2026approximate}
M.~F. Djete and N.~Touzi.
\newblock On approximate {N}ash equilibria in mean field games.
\newblock {\em arXiv preprint arXiv:2601.20910}, 2026.

\bibitem{dudley2002real}
R.~M. Dudley.
\newblock {\em Real Analysis and Probability}.
\newblock Cambridge Studies in Advanced Mathematics. Cambridge University
  Press, 2 edition, 2002.

\bibitem{durrett2019probability}
R.~Durrett.
\newblock {\em Probability: Theory and Examples}.
\newblock Cambridge Series in Statistical and Probabilistic Mathematics.
  Cambridge University Press, 5 edition, 2019.

\bibitem{flandoli2022n}
F.~Flandoli, M.~Ghio, and G.~Livieri.
\newblock \( {N}\)-player games and mean field games of moderate interactions.
\newblock {\em Applied Mathematics \& Optimization}, 85(3):38, 2022.

\bibitem{graber2025remarks}
P.~J. Graber.
\newblock Remarks on potential mean field games.
\newblock {\em Research in the Mathematical Sciences}, 12(1):13, 2025.

\bibitem{guo2025Markov}
X.~Guo, X.~Li, C.~Maheshwari, S.~Sastry, and M.~Wu.
\newblock Markov $\alpha$-potential games.
\newblock {\em IEEE Transactions on Automatic Control}, 71(1):275--290, 2025.

\bibitem{2025guobsde}
X.~Guo, X.~Li, and L.~Zhang.
\newblock {BSDE} approach for $\alpha$-potential stochastic differential games.
\newblock {\em arXiv preprint arXiv:2507.13256}.

\bibitem{guo2025alpha}
X.~Guo, X.~Li, and Y.~Zhang.
\newblock An \({\alpha }\)-potential game framework for \( {N}\)-player dynamic
  games.
\newblock {\em SIAM Journal on Control and Optimization}, 63(4):2964--3005,
  2025.

\bibitem{guo2025distributed}
X.~Guo, X.~Li, and Y.~Zhang.
\newblock Distributed games with jumps: An $\alpha$-potential game approach.
\newblock {\em arXiv preprint arXiv:2508.01929}, 2025.

\bibitem{guo2025towards}
X.~Guo and Y.~Zhang.
\newblock Towards an analytical framework for dynamic potential games.
\newblock {\em SIAM Journal on Control and Optimization}, 63(2):1213--1242,
  2025.

\bibitem{hofer2025optimal}
F.~H\"{o}fer and H.~M. Soner.
\newblock Optimal control and potential games in the mean field.
\newblock {\em arXiv preprint arXiv:2408.00733}, 2025.

\bibitem{huang2006large}
M.~Huang, R.~P. Malham{\'e}, and P.~E. Caines.
\newblock {Large population stochastic dynamic games: closed-loop
  {McKean-Vlasov} systems and the {Nash} certainty equivalence principle}.
\newblock {\em Communications in Information \& Systems}, 6(3):221 -- 252,
  2006.

\bibitem{hofer2026optimal}
F.~Höfer and H.~M. Soner.
\newblock Optimal control and potential games in the mean field.
\newblock {\em Stochastic Processes and their Applications}, 199:104971, 2026.

\bibitem{jordan2026independent}
P.~Jordan and M.~Kamgarpour.
\newblock Independent learning of {N}ash equilibria in partially observable
  {M}arkov potential games with decoupled dynamics.
\newblock {\em arXiv preprint arXiv:2605.06377}, 2026.

\bibitem{kalaria2025alpha}
D.~Kalaria, C.~Maheshwari, and S.~Sastry.
\newblock $\alpha$-racer: Real-time algorithm for game-theoretic motion
  planning and control in autonomous racing using $\alpha$-potential function.
\newblock {\em Proceedings of Machine Learning Research vol}, 283:1--18, 2025.

\bibitem{kallenbert2021foundations}
O.~Kallenberg.
\newblock {\em Foundations of Modern Probability}.
\newblock Springer Cham, 2021.

\bibitem{lacker2015mean}
D.~Lacker.
\newblock Mean field games via controlled martingale problems: Existence of
  {M}arkovian equilibria.
\newblock {\em Stochastic Processes and their Applications}, 125(7):2856--2894,
  2015.

\bibitem{lacker2017limit}
D.~Lacker.
\newblock Limit theory for controlled {McKean--Vlasov} dynamics.
\newblock {\em SIAM Journal on Control and Optimization}, 55(3):1641--1672,
  2017.

\bibitem{lasry2007mean}
J.-M. Lasry and P.-L. Lions.
\newblock Mean field games.
\newblock {\em Japanese Journal of Mathematics}, 2(1):229--260, 2007.

\bibitem{lauriere2022convergence}
M.~Lauri{\`e}re and L.~Tangpi.
\newblock Convergence of large population games to mean field games with
  interaction through the controls.
\newblock {\em SIAM Journal on Mathematical Analysis}, 54(3):3535--3574, 2022.

\bibitem{learning205potential}
C.~Maheshwari, M.~We, D.~Pai, and S.~Sastry.
\newblock Independent and decentralized learning in potential games.
\newblock {\em IEEE Transactions on Automatic Control}, 70(11):7538--7553,
  2025.

\bibitem{monderer1996potential}
D.~Monderer and L.~S. Shapley.
\newblock Potential games.
\newblock {\em Games and Economic Behavior}, 14(1):124--143, 1996.

\bibitem{neuman2026potential}
E.~Neuman and S.~Tuschmann.
\newblock Potential games on unimodular random graphs.
\newblock {\em arXiv preprint arXiv:2604.13836}, 2026.

\bibitem{plank2026learning}
P.~Plank and Y.~Zhang.
\newblock Learning distributed equilibria in linear-quadratic stochastic
  differential games: An $\alpha $-potential approach.
\newblock {\em arXiv preprint arXiv:2602.16555}, 2026.

\bibitem{possamai2025non}
D.~Possama{\"\i} and L.~Tangpi.
\newblock Non-asymptotic convergence rates for mean-field games: weak
  formulation and {McKean--Vlasov BSDEs}.
\newblock {\em Applied Mathematics \& Optimization}, 91(3):58, 2025.

\bibitem{santambrogio2021cucker}
F.~Santambrogio and W.~Shim.
\newblock A {C}ucker--{S}male inspired deterministic mean field game with
  velocity interactions.
\newblock {\em SIAM Journal on Control and Optimization}, 59(6):4155--4187,
  2021.

\bibitem{zhang2017backward}
J.~Zhang.
\newblock {\em Backward Stochastic Differential Equations}.
\newblock Springer New York, NY, 2017.

\end{thebibliography}
\end{document}